\documentclass[12pt]{amsart}

\usepackage{amssymb,amscd,amsthm,verbatim,amsmath,color,fancyhdr,mathrsfs}
\usepackage{graphicx}
\usepackage{turnstile}
\usepackage[disable]{todonotes}
\usepackage{csquotes}
\usepackage[mathscr]{euscript}

\usepackage{centernot}

\usepackage{bm}

\usepackage{tikz-cd}

\usepackage[letterpaper,left=2.5cm,right=2.5cm,top=2.5cm,bottom=2.5cm,dvips]{geometry}
\usepackage[bookmarks=true,linktocpage=true,bookmarksdepth=3]{hyperref}

\tikzset{verticalsouth/.style={anchor=south,rotate=90}}
\tikzset{verticalnorth/.style={anchor=north,rotate=90}}

\newcommand*\centermathcell[1]{\omit\hfil$\displaystyle#1$\hfil\ignorespaces}

\usepackage{cleveref}[2012/02/15]
\crefformat{footnote}{#2\footnotemark[#1]#3}

\makeatletter
\patchcmd{\@bibitem}{\ignorespaces}{\label{bib-#1}\ignorespaces}{}{}
\makeatother

\newtheorem{thm}{Theorem}[section]
\newtheorem{prop}[thm]{Proposition}
\newtheorem{lem}[thm]{Lemma}
\newtheorem{cor}[thm]{Corollary}
\newtheorem{fact}[thm]{Fact}

\theoremstyle{remark}
\newtheorem{rem}[thm]{Remark}
\newtheorem{ex}[thm]{Example}
\newtheorem{rem-prob}[thm]{Remark*}
\newtheorem{notation}[thm]{Notation}
\newtheorem{conv}[thm]{Convention}
\newtheorem{question}[thm]{Question}

\theoremstyle{definition}

\newtheorem{defn}[thm]{Definition}

\title[Skew divided difference operators in the Nichols algebra]{Skew divided difference operators in the Nichols algebra associated to a finite Coxeter group}

\author{Christoph B\"arligea}
\address{Institut \'Elie Cartan de Lorraine (I\'ECL)\\ UMR~7502\\ FST\\ Campus Aiguillettes\\ B.P.~70239\\ 54506 Vand{\oe}vre-l\`es-Nancy\\ France}
\email{christoph.baerligea@rub.de}
\thanks{The research was supported by the German Research Foundation (DFG) under the project number 345815019}

\date{April~17,~2018}

\begin{document}

\begin{abstract}
\pdfbookmark[1]{\abstractname}{pdf:abstract}

Let $(W,S)$ be a finite Coxeter system with root system $R$ and with set of positive roots $R^+$. For $\alpha\in R$, $v,w\in W$, we denote by $\partial_\alpha$, $\partial_w$ and $\partial_{w/v}$ the divided difference operators and skew divided difference operators acting on the coinvariant algebra of $W$. Generalizing the work of Liu \cite{liu}, we prove that $\partial_{w/v}$ can be written as a polynomial with nonnegative coefficients in $\partial_\alpha$ where $\alpha\in R^+$. In fact, we prove the stronger and analogous statement in the Nichols-Woronowicz algebra model for Schubert calculus on $W$ after Bazlov \cite{bazlov1}. We draw consequences of this theorem on saturated chains in the Bruhat order, and partially treat the question when $\partial_{w/v}$ can be written as a monomial in $\partial_\alpha$ where $\alpha\in R^+$. In an appendix, we study related combinatorics on shuffle elements and Bruhat intervals of length two.

\end{abstract}

\maketitle
\setcounter{tocdepth}{1}
\pdfbookmark[1]{\contentsname}{pdf:contents}
\tableofcontents

\section{Introduction}

Let $\mathbb{S}_m$ be the symmetric group on $m$ letters. In \cite{liu}, Liu proves that skew divided difference operators $\partial_{w/v}$ where $v,w\in\mathbb{S}_m$ can be written as polynomials with nonnegative integer coefficients in divided difference operators of the form $\partial_{ij}$ where $1\leq i<j\leq m$. 
Liu actually proves the stronger and analogous statement for the Fomin-Kirillov algebra $\mathscr{E}_m$ as introduced in \cite{quadratic} 
(i.e.\ he uses considerably less relations among the generators) by constructing an explicit and manifestly positive formula for the elements $x_{w/v}\in\mathscr{E}_m$ which act via $\partial_{w/v}$. 
Thereby, he gives an affirmative answer to a conjecture of Kirillov \cite[Conjecture~2]{kirillov}

Closely following these ideas, we generalize in the present work Liu's results from type $\mathsf{A}$ to an arbitrary finite Coxeter group $W$.
%
Instead of the Fomin-Kirillov algebra restricted to the symmetric group, we use the Nichols-Woronowicz algebra model for Schubert calculus on $W$
after Bazlov \cite{bazlov1} as an algebra model for elements with controlled relations acting on the coinvariant algebra $S_W$ of $W$ via divided difference 
operators.

In this introduction, we give an account of the generalization
of positivity 
in classical terms, i.e.\ in terms of operators acting on $S_W$. The reader should bear in mind that almost all arguments in this article take place instead in Bazlov's Nichols algebra associated to $W$, and are related to the introduction only in the very end. We present the material here in this classical language just to make it accessible to everyone. After we have recalled basic definitions, we will come to the precise statements of the results in Subsection~\ref{subsec:summary}. 

\todo[inline,color=green]{Many notions (like braided tensor product algebra, braided Hopf algebra, Hopf duality pairing, etc.) go back to Majid although we don't quote the original sources because we don't know them. We give a greater generality than needed to make axiomatics clear to everyone. Maybe we quote the book pars pro toto. We will certainly also cite the Andrus paper as reference for the generalities.

We quote the sources which provide a formulation which is convenient for us and which we adopted in this text. The reader can go to the references (Majid, Andrus, Bazlov2) and the references therein for proofs and further explanations. In the end, it is not at all original.}

\subsection{Coxeter groups}
\label{subsec:coxeter}

Let $(W,S)$ be a Coxeter system such that $W$ is a finite Coxeter groups. We call the elements of $S$ simple reflections. For each pair of simple reflections $s$ and $s'$, we denote by $m(s,s')$ the Coxeter integer, i.e.\ the necessarily finite order of $ss'$ in $W$ (cf.~\cite[Proposition~5.3]{humphreys-coxeter}). By definition, we clearly have $m(s,s)=1$ for all $s\in S$. Let $\mathfrak{h}_{\mathbb{R}}$ be a real vector space of dimension $|S|$. We choose a basis $\Delta$ of $\mathfrak{h}_{\mathbb{R}}$ and a bijection between $\Delta$ and $S$. We denote the image of $\beta\in\Delta$ under the bijection $\Delta\overset{\sim}{\to}S$ by $s_\beta$. We define $m(\beta,\beta')=m(s_\beta,s_{\beta'})$ for all $\beta,\beta'\in\Delta$. We define a symmetric bilinear form $B$ on $\mathfrak{h}_{\mathbb{R}}$ by setting $B(\beta,\beta')=-\cos\frac{\pi}{m(\beta,\beta')}$ and by extending bilinearly. There exists a well-defined action of the Coxeter group $W$ on $\mathfrak{h}_{\mathbb{R}}$ which is uniquely determined on generators by the assignment $s_\beta(x)=x-2B(x,\beta)\beta$ for all $\beta\in\Delta$ and for all $x\in\mathfrak{h}_{\mathbb{R}}$ (cf.~\cite[loc.~cit.]{humphreys-coxeter}). The corresponding representation $W\to\operatorname{Aut}_{\mathbb{R}}(\mathfrak{h}_{\mathbb{R}})$ is faithful by \cite[Corollary~5.4]{humphreys-coxeter}. This representation is called the geometric representation of $W$. The geometric representation is irreducible if and only if $(W,S)$ is irreducible (\cite[Chapter~V, \S~4, n\textsuperscript{o}~8, Corollary of Theorem~2]{bourbaki_roots}). Finally, note, since $W$ is assumed to be finite, the bilinear form $B$ becomes a $W$-invariant scalar product on $\mathfrak{h}_{\mathbb{R}}$ (\cite[Chapter~V, \S~4, n\textsuperscript{o}~8, Theorem~2]{bourbaki_roots} and \cite[Proposition~5.3]{humphreys-coxeter}).

To the situation above, we attach a reduced root system $R$ and a partial order \enquote{$\leq$} on $R$ as in \cite[Section~5.4]{humphreys-coxeter}. The set of positive elements of $R$ with respect to this partial order is denoted by $R^+$. The elements of $R^+$ are called positive roots. The elements of $\Delta$ are called simple roots. For a root $\alpha$, we denote by $s_\alpha\in W$ the reflection associated to $\alpha$ (cf.\ \cite[Section~5.7]{humphreys-coxeter}). The reflection $s_\alpha$ acts on $\mathfrak{h}_{\mathbb{R}}$ via the familiar formula (analogous to the formula for $s_\beta$ where $\beta\in\Delta$).

We denote the length function on $W$ by $\ell$. For two elements $v,w\in W$, we define $\ell(v,w)=\ell(w)-\ell(v)$ for brevity. We denote the strong Bruhat order on $W$ by \enquote{$\leq$}. We will call it simply \enquote{Bruhat order}. We denote the covering relation in the Bruhat order by \enquote{$\lessdot$}. We denote the weak left Bruhat order by \enquote{$\leq_\ell$}.

\todo[inline,color=green]{I call the strong Bruhat order simply Bruhat order. And except this, I also speak about the weak left Bruhat order. If I mean the latter one, I explicitly say weak left Bruhat order. Part of the notation from the appendix is already used and needed before, such as, for example, $\ell(v,w),\ell,\lessdot$.}

\todo[inline,color=green]{Elements of $\Delta$ are called simple roots -- to be said later in the terminology section on roots. Also $\mathfrak{h}$ a bit later. $B$ is also nondegenerate because $W$ is finite. $B$ is also defined on $\mathfrak{h}$. We have the whole passage from real to complex which still has to be investigated. [We call $\mathfrak{h}$ the Cartan algebra over $\mathbb{C}$ corresponding to $(W,S)$ (relevant for the abstract) [not needed].]}

\begin{notation}

Throughout the text, we denote by $w_o$ the unique longest element of $W$.

\end{notation}

\begin{conv}

Throughout the text, we repeatedly use the fundamental results \cite[Proposition~5.2 and 5.7]{humphreys-coxeter} without referencing.  

\end{conv}

\subsection{Coinvariant algebra \texorpdfstring{{\normalfont(\cite[Subsection~1.2]{bazlov1})}}{([\ref{bib-bazlov1},~Subsection~1.2])}}

Let $\mathfrak{h}=\mathfrak{h}_{\mathbb{R}}\otimes_{\mathbb{R}}\mathbb{C}$ considered as a $\mathbb{C}$-vector space. The scalar product $B$ on $\mathfrak{h}_{\mathbb{R}}$ extends to a nondegenerate symmetric bilinear form on $\mathfrak{h}$ which we still denote by $B$. Consequently, the geometric representation extends to a faithful representation $W\to\operatorname{Aut}_{\mathbb{C}}(\mathfrak{h})$, and the bilinear form $B$ on $\mathfrak{h}$ becomes $W$-invariant under the corresponding action. Let $S(\mathfrak{h})$ be the symmetric algebra on $\mathfrak{h}$ considered as a $\mathbb{Z}_{\geq 0}$-graded algebra in the category of $W$-modules. We refer to elements of $S(\mathfrak{h})$ as polynomials. Let $I_W$ be the ideal in $S(\mathfrak{h})$ generated by $W$-invariant polynomials without constant term. The ideal $I_W$ is a $\mathbb{Z}_{\geq 0}$-graded and $W$-stable ideal in $S(\mathfrak{h})$. Hence, the quotient $S(\mathfrak{h})/I_W$, denoted by $S_W$ from now on, becomes a $\mathbb{Z}_{\geq 0}$-graded algebra in the category of $W$-modules. This algebra is called the coinvariant algebra of $W$.

\begin{rem}
\label{rem:coinv-ident}

Note that $S_W^0\cong\mathbb{C}$ and $S_W^1\cong\mathfrak{h}$ as $W$-modules. We will from now on identify the $\mathbb{Z}_{\geq 0}$-degree zero and one component of $S_W$ with $\mathbb{C}$ respectively $\mathfrak{h}$.

\end{rem}

\begin{rem}

Note that both algebras $S(\mathfrak{h})$ and thus $S_W$ are commutative. Furthermore, note that $S_W$ considered as a $\mathbb{C}$-vector space is of finite dimension $|W|$. 

\end{rem}

\subsection{Divided difference and skew divided difference operators}
\label{subsec:partial}

Let $\alpha\in R$. We define a $\mathbb{C}$-linear endomorphism $\partial_\alpha$ acting from the left on $S_W$ by the formula $\partial_\alpha=\frac{1-s_\alpha}{\alpha}$. The operator $\partial_\alpha$ is called a divided difference operator. Let $w\in W$, and let $s_{\beta_1}\cdots s_{\beta_\ell}$ be a reduced expression of $w$. Then, we define $\partial_w=\partial_{\beta_1}\cdots\partial_{\beta_\ell}$. The operator $\partial_w$ is well-defined, i.e.\ independent of the choice of the reduced expression of $w$, because the divided difference operators $\partial_\beta$ where $\beta\in\Delta$ satisfy the nilCoxeter relations, and in particular the braid relations. 

\begin{notation}[Notation~\ref{not:varphi}]

Let $w\in W$. Let $s_{\beta_1}\ldots s_{\beta_\ell}$ be a fixed reduced expression of $w$. Let $\bm{\beta}=(\beta_1,\ldots,\beta_\ell)$ be the sequence of simple roots corresponding to the fixed reduced expression of $w$. Let $J$ be any subset of $\{1,\ldots,\ell\}$. For all $1\leq j\leq\ell$, we define $\partial_j^{\bm{\beta}}(J)=s_{\beta_j}$ if $j\in J$ and $=\partial_{\beta_j}$ if $j\notin J$. Using this notation, we further define $\partial_J^{\bm{\beta}}=\partial_1^{\bm{\beta}}(J)\cdots\partial_\ell^{\bm{\beta}}(J)$. All the defined operators are supposed to be $\mathbb{C}$-linear endomorphism acting from the left on $S_W$.

\end{notation}

\begin{conv}

Let $J\subseteq\mathbb{Z}$. From now on, whenever we write a product $\prod_J$, we mean the product taken according to the natural total order on $J$ induced by the natural total order on $\mathbb{Z}$.

\end{conv}

\begin{defn}[Corollary~\ref{cor:independent-Dwvcirc}]
\label{def:partialwv}

Let $v,w\in W$ be such that $v\leq w$. Let $s_{\beta_1}\cdots s_{\beta_\ell}$ be a fixed reduced expression of $w$. Let $\bm{\beta}=(\beta_1,\ldots,\beta_\ell)$ be the sequence of simple roots corresponding to the fixed reduced expression of $w$. Then, the expression $v^{-1}\sum_J\partial_J^{\bm{\beta}}$, where the sum ranges over all subsets $J$ of $\{1,\ldots,\ell\}$ such that the product $\prod_{j\in J}s_{\beta_j}$ is a reduced expression of $v$, is independent of $\bm{\beta}$, i.e.\ independent of the choice of the reduced expression of $w$, and we can define $\partial_{w/v}$ to be equal to this expression.
If $v,w\in W$ are such that $v\not\leq w$, then we define further $\partial_{w/v}=0$.
We call $\partial_{w/v}$ where $v,w\in W$ a skew divided difference operator.

\end{defn}

\begin{ex}[Example~\ref{ex:top-low-degree-xw}]
\label{ex:top-low-degree-xw-intro}

Let $w\in W$. It follows directly from the definition that we have $\partial_{w/1}=\partial_w$ and $\partial_{w/w}=1$.

\end{ex}

\begin{thm}[Corollary~\ref{cor:leibniz}]

Let $w\in W$ and $\phi,\psi\in S_W$. Then we have
\[
\partial_w(\phi\psi)=\sum_{v\leq w}(\partial_v(\phi))(v\partial_{w/v}(\psi))\,.
\]

\end{thm}

\subsection{Summary of results}
\label{subsec:summary}

Since $\partial_{-\alpha}=-\partial_\alpha$ and $w\partial_\alpha w^{-1}=\partial_{w(\alpha)}$ for all $w\in W$ and all $\alpha\in R$, it is clear from Definition~\ref{def:partialwv} that $\partial_{w/v}$ where $v,w\in W$ can be written as a polynomial in $\partial_\alpha$ where $\alpha\in R^+$. However, the naive expression for $\partial_{w/v}$ resulting from Definition~\ref{def:partialwv} will be in general not positive, in the sense that all coefficients are nonnegative integers. This problem is solved by the literal generalizations of Liu's theorems.

\begin{thm}[Corollary~\ref{cor:xwv-circ-pos}]
\label{thm:xwv-circ-pos-intro}

Let $v,w\in W$ be such that $v\leq w$. Let $s_{\beta_1}\cdots s_{\beta_\ell}$ be a reduced expression of $w_ov$. Let $\alpha_i=s_{\beta_\ell}\cdots s_{\beta_{i+1}}(\beta_i)$ for all $1\leq i\leq\ell$. Then we have
\[
\partial_{w/v}=\sum_J\prod_{j\in J}\partial_{\alpha_j}
\]
where the sum ranges over all subsets $J$ of $\{1,\ldots,\ell\}$ such that the product $\prod_{j\in\bar{J}}s_{\beta_j}$ is a reduced expression of $w_ow$. Here, the set $\bar{J}$ denotes the complement of a subset $J$ in $\{1,\ldots,\ell\}$.

\end{thm}

\begin{cor}[Corollary~\ref{cor:recursive-circ}]
\label{cor:recursive-circ-intro}

Let $v,w\in W$ be such that $v\leq w$. Let $\beta\in\Delta$ be such that $v<s_\beta v$. Let $v'=s_\beta v$, $w'=s_\beta w$, $\alpha=v^{-1}(\beta)$ for brevity. Then we have $\partial_{w/v}=\partial_\alpha\partial_{w/v'}$ if $w>w'$ and $=\partial_\alpha\partial_{w/v'}+\partial_{w'/v'}$ if $w<w'$.

\end{cor}

The recursive formula in Corollary~\ref{cor:recursive-circ-intro} serves equally well as the formula in Theorem~\ref{thm:xwv-circ-pos-intro} to produce a positive expression for $\partial_{w/v}$ where $v\leq w$.
Indeed, this can be seen by induction on $\ell(v)$.
The induction base is given by Example~\ref{ex:top-low-degree-xw-intro}.


\begin{ex}[Corollary~\ref{cor:monomial-one}]

Let $v,w\in W$ be such that $v\lessdot w$. Let $\alpha\in R^+$ be uniquely determined by the equation $v=w s_\alpha$. Then we have $\partial_{w/v}=\partial_\alpha$.

\end{ex}

\begin{thm}[Theorem~\ref{thm:monomial-two}]
\label{thm:monomial-two-partial}

Let $W$ be a simply laced Weyl group. Let $v,w\in W$ be such that $v\leq w$ and such that $\ell(v,w)=2$. Then, there exist positive roots $\alpha,\gamma$ such that $\alpha\neq\gamma$ and such that $\partial_{w/v}=\partial_\alpha\partial_\gamma$.

\end{thm}

\subsection*{Organization}

In Section~\ref{sec:hopf} and~\ref{sec:group}, we recall basic notions about the Yetter-Drinfeld category over a Hopf algebra with invertible antipode, and more specific, over the group algebra of an arbitrary group. In Section~\ref{sec:bazlov}, we recall the setup and some theorems of \cite{bazlov1} concerning the Nichols algebra $\mathscr{B}_W$ associated to a finite Coxeter group $W$. In Section~\ref{sec:skew}, we introduce 
the elements $x_{w/v}\in\mathscr{B}_W$ where $v,w\in W$ which parametrize skew partial derivatives acting on $\mathscr{B}_W$ -- the \enquote{Nichols analogue} of skew divided difference operators as introduced in Definition~\ref{def:partialwv}. Generalizing and following \cite{liu}, we prove in Section~\ref{sec:positivity} that $x_{w/v}$ are positive elements of $\mathscr{B}_W$ for all $v\leq w$. In Section~\ref{sec:monomial}, we prove monomial properties, most notably Theorem~\ref{thm:monomial-two} -- the analogue of Theorem~\ref{thm:monomial-two-partial} in $\mathscr{B}_W$. Section~\ref{sec:strict-positivity} is devoted to further generalizations to the relative setting $x_{w/v}$. Section~\ref{sec:comb} is devoted to combinatorial consequences of positivity on saturated chains in the Bruhat order. Finally, in Section~\ref{sec:passage}, we relate the abstract results to the classical situation, thereby proving all of the claims in this introduction.

Appendix~\ref{appendix:shuffle} on shuffle elements and Bruhat intervals of length two is completely independent from the rest of the text. Vice versa, we use as only input from the appendix Theorem~\ref{thm:main-interval2} once in the proof of Theorem~\ref{thm:monomial-two}. A reader only interested in the combinatorics of Weyl groups can go directly after this introduction to Appendix~\ref{appendix:shuffle}. 

\todo[inline,color=green]{The appendix is independent from the main corpus of the text and vice versa. We use Theorem~\ref{thm:main-interval2} as the only input from the appendix in the main corpus only once in the proof of Theorem~\ref{thm:monomial-two}, and this latter theorem is not needed elsewhere in the text [the part of the sentence after the comma is not mentioned]. We drop some open Problems in Remarks* which are maybe easy to solve but we cannot yet. 

{\bf Depends (I adopted this point of view):} If I upload a new version with the necessary counterexamples, this is said to much. These remarks should stay independent of the rest of the text (no references).}

\subsection*{Acknowledgment}

The support of the German Research Foundation (DFG) is gratefully acknowledged. The author thanks Pierre-Emmanuel Chaput, the host of this research project in Nancy. The author thanks Nicolas Perrin for an invitation to Versailles to give a talk on the subject of this paper.


\section{Yetter-Drinfeld category over a Hopf algebra}
\label{sec:hopf}

\todo[inline,color=green]{General references: like on monoidal and braided monoidal categories, not specific references in the text, among them \cite{majid}, Etingof, Andrus, Bazlov2,\ldots}

In this section, we recall general notions around the Yetter-Drinfeld category over a Hopf algebra (with invertible antipode). For more detailed informations on the axiomatics of monoidal and braided monoidal categories, we refer the reader to \cite{etingof,majid}. In case of more specific topics, such as Yetter-Drinfeld modules, free braided groups and Nichols algebras, we follow the references \cite{andrus,bazlov1,bazlov2}. A reader familiar with this setup can go directly to the next section.


Throughout the discussion, we fix a field $\mathbf{k}$ and a Hopf algebra $H$ over $\mathbf{k}$ with invertible antipode $\mathscr{S}$. We denote by $\nabla$ the multiplication, by $\eta$ the unit, by $\Delta$ the comultiplication and by $\epsilon$ the counit. 

\begin{conv}
\label{conv:generic}

We use the notation $\nabla,\eta,\Delta,\epsilon,\mathscr{S}$ for the fixed Hopf algebra $H$ but also \enquote{generically} for every other Hopf algebra which will arise, may it be braided (cf.\ Definition~\ref{def:braided-hopf}) or not. 
%
%
Similarly, we denote the coaction on a $H$-comodule generically by $\delta$.

\end{conv}

\begin{conv}

From now on, we will always use Sweedler notation. We denote the $m$-fold coproduct of an element $h$ (regardless of which Hopf algebra or braided Hopf algebra) by $h_{(1)}\otimes\cdots\otimes h_{(m)}$. We denote the coaction on an element $v$ (regardless of which $H$-comodule) by $v_{(-1)}\otimes v_{(0)}$.

\end{conv}

\todo[inline,color=green]{Note that we use $\Delta$ to denote the comultiplication of a Hopf algebra, whichever it may be, as well as to denote the set of simple roots of a given Coxeter group (cf.\ Subsection~\ref{subsec:coxeter}). No problem will arise from this double meaning of $\Delta$ because a map can hardly be confused with a set of roots.

I ignore this issue everywhere and do not mention it. The same remark applies for $W$. I use it everywhere for the Weyl group / Coxeter group and once and a while for a Yetter-Drinfeld module if $V$ is already taken.}

\begin{defn}

A Yetter-Drinfeld module $V$ over $H$ or simply a Yetter-Drinfeld $H$-module $V$ is simultaneously a left $H$-module and left $H$-comodule such that the compatibility condition
\[
\delta(hv)=h_{(1)}v_{(-1)}\mathscr{S}(h_{(3)})\otimes h_{(2)}v_{(0)}
\]
is satisfied for all $h\in H$ and all $v\in V$. A morphism of Yetter-Drinfeld $H$-modules is simultaneously an $H$-module and $H$-comodule homomorphism. The category of all Yetter-Drinfeld $H$-modules with these morphisms is denoted by ${}_H^H\mathscr{YD}$, and called Yetter-Drinfeld category over $H$.

\end{defn}

\begin{ex}
\label{ex:trivial}

We equip $\mathbf{k}$ with a $H$-module and $H$-comodule structure defined by the equations $h\lambda=\epsilon(h)\lambda$ and $\delta(\lambda)=1\otimes\lambda$ for all $h\in H$ and all $\lambda\in\mathbf{k}$. These structures are clearly compatible and make $\mathbf{k}$ into a Yetter-Drinfeld $H$-module. We call this Yetter-Drinfeld $H$-module the trivial Yetter-Drinfeld $H$-module. From now on, whenever we speak about $\mathbf{k}$ in the context of Yetter-Drinfeld $H$-modules, it will be endowed with the structure of the trivial Yetter-Drinfeld $H$-module.

\end{ex}

\begin{proof}[The category ${}_H^H\mathscr{YD}$ is a braided monoidal category]\renewcommand{\qedsymbol}{$\triangle$}

The tensor product $V\otimes W$ of two Yetter-Drinfeld $H$-modules $V$ and $W$ is given by the usual tensor product $V\otimes_{\mathbf{k}} W$ of $\mathbf{k}$-vector spaces equipped with the $H$-module and $H$-comodule structure defined by the equations
\[
h(v\otimes w)=h_{(1)}v\otimes h_{(2)}w\quad\text{and}\quad\delta(v\otimes w)=v_{(-1)}w_{(-1)}\otimes v_{(0)}\otimes w_{(0)}
\]
for all $h\in H$, $v\in V$, $w\in W$. One can check that this $H$-module and $H$-comodule structure makes $V\otimes W$ into a Yetter-Drinfeld $H$-module, and that the corresponding bifunctor $\otimes$ along with the unit object $\mathbf{k}$ as in Example~\ref{ex:trivial} and along with the obvious natural equivalences make ${}_H^H\mathscr{YD}$ into a monoidal category. The braiding $\Psi_{V,W}\colon V\otimes W\to W\otimes V$ of two Yetter-Drinfeld $H$-modules $V$ and $W$ is given by
$
\Psi_{V,W}(v\otimes w)=v_{(-1)}w\otimes v_{(0)}
$
for all $v\in V$ and all $w\in W$. The braiding $\Psi_{V,W}$ of $V$ and $W$ 
is an isomorphism in the Yetter-Drinfeld category over $H$ with inverse $\Psi_{V,W}^{-1}$ given by 
$
\Psi_{V,W}^{-1}(w\otimes v)=v_{(0)}\otimes\mathscr{S}^{-1}(v_{(-1)})w
$
for all $v\in V$ and all $w\in W$. One can check that the corresponding natural equivalence $\Psi$ makes the monoidal category ${}_H^H\mathscr{YD}$ into a braided monoidal category. The natural equivalence $\Psi$ is called the braiding of ${}_H^H\mathscr{YD}$.
\end{proof}

\begin{conv}

For simplicity, whenever we work with monoidal categories, we will always suppress the natural associativity isomorphisms in our notation, as well as all other natural isomorphisms related to monoidal categories. 

\end{conv}

\begin{rem}

Less interesting, the category ${}_H^H\mathscr{YD}$ is in addition to $\Psi$ equipped with a symmetric braiding, denoted by $\tau$, which makes it into a symmetric monoidal category. For two Yetter-Drinfeld $H$-modules $V$ and $W$, the component $\tau_{V,W}\colon V\otimes W\to W\otimes V$ is simply given by swapping the two tensor factors. The Yetter-Drinfeld category over $H$ will always be considered as a braided monoidal category equipped with the nontrivial braiding $\Psi$ as in the paragraph above. The symmetric braiding $\tau$ is only explicitly needed in the next section.

\end{rem}

\begin{conv}

We often suppress the subscripts in the braiding of two Yetter-Drinfeld $H$-modules $V$ and $W$. If it is understood from context which $V$ and $W$ are meant, we simply write $\Psi$ instead of $\Psi_{V,W}$, with the effect that $\Psi$ has a double meaning as natural equivalence and its component at $(V,W)$. The very same convention as for $\Psi$ also applies to $\tau$.

\end{conv}

A Yetter-Drinfeld $H$-module is said to be rigid if it admits a left dual in the sense of \cite[Definition~9.3.1]{majid}.

\begin{proof}[A Yetter-Drinfeld $H$-module is rigid if and only if it is finite dimensional]\renewcommand{\qedsymbol}{$\triangle$}

Let $V$ be a finite dimensional Yetter-Drinfeld $H$-module. We equip the usual dual $V^*$ of $V$ as a $\mathbf{k}$-vector space with the structure of an $H$-module and $H$-comodule defined by the equations:
\begin{align*}
(hf)(v)&=f(\mathscr{S}(h)v)&&\text{for all $h\in H$, $f\in V^*$, $v\in V$,}\\
\delta(f)&=f_{(-1)}\otimes f_{(0)}&&\text{for all $f\in V^*$, uniquely determined by}\\
f_{(-1)}f_{(0)}(v)&=\mathscr{S}^{-1}(v_{(-1)})f(v_{(0)})&&\text{for all $v\in V$.}
\end{align*}
More explicitly, the $H$-comodule structure can be defined in coordinates as follows. Let $e_1,\ldots,e_n$ be a basis of $V$ with dual basis $e_1^*,\ldots,e_n^*$ of $V^*$. Let $h_{ij}\in H$ be such that $\delta(e_i)=\sum_j h_{ij}\otimes e_j$ where here and in the following equations in this paragraph $i$ and $j$ range through $\{1,\ldots,n\}$. Then we define $\delta(e_i^*)=\sum_j\mathscr{S}^{-1}(h_{ji})\otimes e_j^*$ and extend linearly. One can check that this latter definition is independent of the choice of the basis of $V$ and coincides with the abstract definition of the coaction in the displayed equation above. Furthermore, one can check that the compatibility condition is satisfied for the $H$-module and $H$-comodule structure defined in this way on $V^*$. Thus, $V^*$ becomes a Yetter-Drinfeld $H$-module. The familiar evaluation and coevaluation maps in the category of vector spaces turn out to be morphisms in the category of Yetter-Drinfeld modules over $H$, and make $V^*$ into the left dual of $V$. For example, the coevaluation morphism
$
\mathbf{k}\to V\otimes V^*
$
is given by 
$
\lambda\mapsto\lambda\sum_i e_i\otimes e_i^*
$
where $e_1,\ldots,e_n$ is a basis of $V$ and $e_1^*,\ldots,e_n^*$ is the corresponding dual basis of $V^*$. Note that this definition is independent of the choice of the basis of $V$.
\end{proof}

\begin{conv}

The left dual $V^*$ of a finite dimensional Yetter-Drinfeld $H$-module $V$ is unique up to isomorphism. We always denote by $V^*$ the left dual of $V$ equipped with the structure of a Yetter-Drinfeld $H$-module as in the paragraph above.


\end{conv}

\begin{conv}
\label{conv:dual}

Let $V$ and $W$ be finite dimensional Yetter-Drinfeld $H$-modules. Then, $V$ and $W$ as well as $V\otimes W$ are rigid.  As in \cite[Figure~9.7(c) and equation on page~444]{majid}, we always identify the left dual of $V\otimes W$ with $W^*\otimes V^*$ (mind the order of the factors); in formulas $(V\otimes W)^*=W^*\otimes V^*$.


\end{conv}

\begin{notation}

Let $V$ be a finite dimensional Yetter-Drinfeld $H$-module. We always denote by $\mathrm{ev}_V$ the evaluation pairing $V^*\otimes V\to\mathbf{k}$ of $V$.

\end{notation}

\begin{defn}[{\cite[Lemma~9.2.12]{majid}}]

Let $A$ and $B$ be two algebras in ${}_H^H\mathscr{YD}$. We equip the tensor produce $A\otimes B$ with the multiplication defined by the composite morphism
\[
A\otimes B\otimes A\otimes B\xrightarrow{1\otimes\Psi\otimes 1}A\otimes A\otimes B\otimes B\to A\otimes B
\]
where the second arrow is given by the tensor product of multiplication in $A$ with multiplication in $B$. We also equip $A\otimes B$ with the unit defined by $1\mapsto 1\otimes 1$. These structures make $A\otimes B$ into an algebra in ${}_H^H\mathscr{YD}$ which is denoted by $A\underline{\otimes}B$ and called braided tensor product algebra. The multiplication in $A\underline{\otimes}B$ is called twisted multiplication in $A\underline{\otimes}B$.

\end{defn}

\begin{defn}[{\cite[Definition~1.7]{andrus}}]
\label{def:braided-hopf}

A braided Hopf algebra $A$ in ${}_H^H\mathscr{YD}$ is simultaneously an algebra and coalgebra in ${}_H^H\mathscr{YD}$ such that $\Delta\colon A\to A\underline{\otimes}A$ and $\epsilon\colon A\to\mathbf{k}$ are algebra homomorphisms and such that $A$ admits an antipode (which is clearly also assumed to be a morphism in the Yetter-Drinfeld category over $H$).
%
A braided $\mathbb{Z}_{\geq 0}$-graded Hopf algebra $A$ in ${}_H^H\mathscr{YD}$ is a braided Hopf algebra in ${}_H^H\mathscr{YD}$ which admits a $\mathbb{Z}_{\geq 0}$-grading $A=\bigoplus_{m=0}^\infty A^m$ of Yetter-Drinfeld $H$-modules such that all of the structure morphisms of $A$ become $\mathbb{Z}_{\geq 0}$-graded morphisms in the Yetter-Drinfeld category over $H$. (Note that the tensor product of $\mathbb{Z}_{\geq 0}$-graded Yetter-Drinfeld $H$-modules is equipped with the obvious $\mathbb{Z}_{\geq 0}$-grading, as well as that $\mathbf{k}$ is equipped with the trivial $\mathbb{Z}_{\geq 0}$-grading concentrated in degree zero.)

\end{defn}

\begin{rem}
\label{rem:anti-homo}

Similarly as the antipode in an ordinary Hopf algebra is an anti algebra and anti coalgebra homomorphism (cf.~\cite[Proposition~5.3.6]{etingof}), the antipode in a braided Hopf algebra in ${}_H^H\mathscr{YD}$ is a braided anti algebra and braided anti coalgebra homomorphism, in the sense that the identities
\[
\mathscr{S}\circ\nabla=\nabla\circ\Psi\circ(\mathscr{S}\otimes\mathscr{S})\quad\text{and}\quad\Delta\circ\mathscr{S}=(\mathscr{S}\otimes\mathscr{S})\circ\Psi\circ\Delta
\]
are fulfilled (cf.~\cite[Equation~(9.39) on page~477]{majid}).

\end{rem}

\begin{defn}[{\cite[Subsection~5.3]{bazlov2}}]

Let $A$ and $B$ be braided Hopf algebras in ${}_H^H\mathscr{YD}$. A Hopf duality pairing between $A$ and $B$ is a morphism $\left<-,-\right>\colon A\otimes B\to\mathbf{k}$ in the Yetter-Drinfeld category over $H$ such that
\begin{gather*}
\left<\phi\psi,x\right>=\left<\phi,x_{(2)}\right>\left<\psi,x_{(1)}\right>\,,\quad\left<\phi,xy\right>=\left<\phi_{(2)},x\right>\left<\phi_{(1)},y\right>\,,\\
\left<1,x\right>=\epsilon(x)\,,\quad\left<\phi,1\right>=\epsilon(\phi)\,,\quad\left<\mathscr{S}(\phi),x\right>=\left<\phi,\mathscr{S}(x)\right>\,,
\end{gather*}
for all $\phi,\psi\in A$ and all $x,y\in B$.

\end{defn}

\subsection{Braided factorial}

Let $m\in\mathbb{Z}_{>0}$ and let $n=m-1$. Let $\mathbb{S}_m$ be the symmetric group on $m$ letters. For all $1\leq i\leq n$, we denote by $s_i$ the $i$th simple transposition, i.e.\ the permutation which swaps the $i$th and the $(i+1)$th letter and fixes all other letters.

Let $V$ be a Yetter-Drinfeld $H$-module. Let $1\leq i\leq n$. We denote by $\Psi_i$ the $\mathbf{k}$-linear endomorphism of $V^{\otimes m}$ defined by
\[
\underbrace{1\otimes\cdots\otimes 1}_{(i-1)\text{ factors}}{}\otimes{}\Psi_{V,V}\otimes\!\underbrace{1\otimes\cdots\otimes 1}_{(m-i-1)\text{ factors}}
\]
where $\Psi_{V,V}$ acts on the $i$th and $(i+1)$th tensor factor of $V^{\otimes m}$. Let $\sigma\in\mathbb{S}_m$. Let $s_{i_1}\cdots s_{i_\ell}$ be a reduced expression of $\sigma$ (where $1\leq i_1,\ldots,i_\ell\leq n$). Then we define a $\mathbf{k}$-linear endomorphism of $V^{\otimes m}$ by the equation $\Psi_\sigma=\Psi_{i_1}\cdots\Psi_{i_\ell}$. The endomorphism $\Psi_\sigma$ is well-defined, i.e.\ independent of the choice of the reduced expression of $\sigma$, because the $\Psi_i$'s satisfy the braid relations as a consequence of the axioms imposed on the braiding $\Psi$ in the braided monoidal category ${}_H^H\mathscr{YD}$. Finally, we define a $\mathbf{k}$-linear endomorphism $[m]!_V$ of $V^{\otimes m}$ by the equation 
\[
[m]!_V=\sum_{\sigma\in\mathbb{S}_m}\Psi_\sigma\,.
\]
The endomorphism $[m]!_V$ is called braided factorial or braided (Woronowicz) symmetrizer. Note that the braided factorial has a factorization in terms of braided integers (cf.~\cite[Subsection~2.5]{bazlov1} and \cite[Subsection~5.5]{bazlov2}). We will often drop the subscript in the braided factorial and simply write $[m]!$ instead of $[m]!_V$ if it is clear from context on the $m$th tensor power of which Yetter-Drinfeld $H$-module $V$ the endomorphism $[m]!$ acts.



\subsection{Free braided group \texorpdfstring{{\normalfont(\cite[Subsection~5.4 and 5.5]{bazlov2})}}{([\ref{bib-bazlov2},~Subsection~5.4 and 5.5])}}

Let $V$ be a Yetter-Drinfeld $H$-module. The $\mathbb{Z}_{\geq 0}$-graded tensor algebra $T(V)$ can be equipped with the structure of a braided $\mathbb{Z}_{\geq 0}$-graded Hopf algebra in ${}_H^H\mathscr{YD}$ which is uniquely determined by the requirement that we keep the algebra structure and the $\mathbb{Z}_{\geq 0}$-grading from the the tensor algebra and by the requirement that the additional structure morphisms satisfy
\[
\Delta(v)=v\otimes 1+1\otimes v\,,\quad\epsilon(v)=0\,,\quad \mathscr{S}(v)=-v
\]
for all $v\in V$. The tensor algebra equipped with this structure is called free braided group. From now on, we think of $T(V)$ always as a free braided group. Note that the comultiplication in  the free braided group can be explicitly described in terms of braided binomial coefficients (cf.~\cite[Subsection~5.4]{bazlov2}).

Let $V$ be a finite dimensional Yetter-Drinfeld $H$-module. Then, there exists a Hopf duality pairing $\left<-,-\right>$ between $T(V^*)$ and $T(V)$ which is uniquely determined by the requirement that the restriction of $\left<-,-\right>$ to $V^*\otimes V$ equals the evaluation pairing of $V$.
While the uniqueness of $\left<-,-\right>$
is rather obvious, we want to recall its explicit construction as follows: By definition, the pairing $\left<-,-\right>$ vanishes on $V^{*\otimes m}\otimes V^{\otimes m'}$ whenever $m\neq m'$ and its restriction to $V^{*\otimes m}\otimes V^{\otimes m}$ is given by
\[
\mathrm{ev}_{V^{\otimes m}}\circ(1\otimes [m]!_V)=
\mathrm{ev}_{V^{\otimes m}}\circ([m]!_{V^*}\otimes 1)
\]
for all $m,m'\in\mathbb{Z}_{\geq 0}$. Here, the last displayed equation follows since $\Psi_{V,V}$ and $\Psi_{V^*,V^*}$ are adjoint with respect to the evaluation pairing of $V^{\otimes 2}$. Note that, for example, in the last displayed equation of the construction, Convention~\ref{conv:dual} is in force. Concretely, this means that $\mathrm{ev}_{V^{\otimes m}}$ considered as a morphism $V^{*\otimes m}\otimes V^{\otimes m}\to\mathbf{k}$ 
must be evaluated from inner tensors to outer tensors. Starting from the explicit construction, using the description of the comultiplication in the free braided group in terms of braided binomial coefficients (cf.~\cite[loc.~cit.]{bazlov2}) and using the braided binomial theorem (cf.~\cite[Subsection~5.5]{bazlov2}), one can verify that $\left<-,-\right>$ is indeed a Hopf duality pairing between $T(V^*)$ and $T(V)$ whose restriction to $V^*\otimes V$ equals the evaluation pairing of $V$. We denote this pairing from now on and everywhere by $\left<-,-\right>$. It will be clear from context which free braided groups are paired.

\subsection{Nichols algebra}

Let $V$ be a finite dimensional Yetter-Drinfeld $H$-module. Let $\operatorname{Wor}(V)$ be defined as $\bigoplus_{m=0}^\infty\operatorname{ker}{}[m]!_V$. Using the axioms of the Hopf duality pairing between $T(V^*)$ and $T(V)$ and arguments as in its explicit construction in the previous subsection, one can verify that $\operatorname{Wor}(V)$ is a $\mathbb{Z}_{\geq 0}$-graded Hopf ideal in $T(V)$ which is called the Woronowicz ideal of $V$. The quotient $T(V)/\operatorname{Wor}(V)$, denoted by $\mathscr{B}(V)$ from now on, therefore becomes a braided $\mathbb{Z}_{\geq 0}$-graded Hopf algebra with the structure inherited from the free braided group. The algebra $\mathscr{B}(V)$ is called the Nichols algebra of $V$. The Hopf duality pairing between $T(V^*)$ and $T(V)$ clearly descends to a Hopf duality pairing between $\mathscr{B}(V^*)$ and $\mathscr{B}(V)$ which is still denoted by $\left<-,-\right>$. By definition, the Hopf duality pairing between $\mathscr{B}(V^*)$ and $\mathscr{B}(V)$ is nondegenerate.
The Nichols algebra of $V$ satisfies the following properties:
\begin{itemize}
\item 
$\mathscr{B}(V)^0\cong\mathbf{k}$ and $\mathscr{B}(V)^1\cong V$ are isomorphic as Yetter-Drinfeld $H$-modules.
\item
$\mathscr{B}(V)$ is generated as an algebra by $\mathscr{B}(V)^1$.
\item
The primitive elements of $\mathscr{B}(V)$ are precisely the element in $\mathscr{B}(V)$ of $\mathbb{Z}_{\geq 0}$-degree one.
\end{itemize}
The first two properties are obvious from the definition while the third corresponds to \cite[Proposition~2.2(i)]{andrus}.
Note that these properties characterize the Nichols algebra of $V$ up to isomorphism (cf.\ \cite[Proposition~2.2(iv)]{andrus}). Therefore, they often serve as an axiomatic definition of $\mathscr{B}(V)$ (cf.\ \cite[Definition~2.1]{andrus}). If no confusion arises, we will from now on identify the $\mathbb{Z}_{\geq 0}$-degree zero and one component of $\mathscr{B}(V)$ with $\mathbf{k}$ respectively $V$ as depicted in the first item above.


\begin{rem}
\label{rem:vanishing}

Let $V$ be a finite dimensional Yetter-Drinfeld $H$-module. From the explicit construction in the previous subsection, it is clear that the Hopf duality pairing between $\mathscr{B}(V^*)$ and $\mathscr{B}(V)$ vanishes on $\mathscr{B}(V^*)^m\otimes\mathscr{B}(V)^{m'}$ where $m,m'\in\mathbb{Z}_{\geq 0}$ such that $m\neq m'$.

\end{rem}

\subsection{Braided differential calculus}

Let $V$ be a finite dimensional Yetter-Drinfeld $H$-module. We define a left action of the algebra $\mathscr{B}(V^*)$ on $\mathscr{B}(V)$ by sending an element $\xi\in\mathscr{B}(V^*)$ to the $\mathbf{k}$-linear endomorphism $\overrightarrow{D}_{\xi}$ of $\mathscr{B}(V)$ defined by the equation
\[
\overrightarrow{D}_{\xi}(x)=\left<\xi,x_{(1)}\right>x_{(2)}
\]
for all $x\in\mathscr{B}(V)$. Similarly, we define a right action of the algebra $\mathscr{B}(V)$ on $\mathscr{B}(V^*)$ by sending an element $x\in\mathscr{B}(V)$ to the $\mathbf{k}$-linear endomorphism $\overleftarrow{D}_x$ of $\mathscr{B}(V^*)$ defined by the equation
\[
(\phi)\overleftarrow{D}_x=\phi_{(1)}\left<\phi_{(2)},x\right>
\]
for all $\phi\in\mathscr{B}(V^*)$. We call $\overrightarrow{D}_{\xi}$ and $\overleftarrow{D}_x$ where $x\in\mathscr{B}(V)$ and $\xi\in\mathscr{B}(V^*)$ braided left and right partial derivatives.

\begin{rem}
\label{rem:degree-partial}

Let $\xi\in\mathscr{B}(V^*)$ be homogeneous of $\mathbb{Z}_{\geq 0}$-degree $m$. Then, the operator $\overrightarrow{D}_\xi$ decreases the $\mathbb{Z}_{\geq 0}$-degree of homogeneous elements of $\mathscr{B}(V)$ by $m$. Consistently with this behavior, we call $m$ the order of $\overrightarrow{D}_\xi$. Similarly, let $x\in\mathscr{B}(V)$ be homogeneous of $\mathbb{Z}_{\geq 0}$-degree $m$. Then, the operator $\overleftarrow{D}_x$ decreases the $\mathbb{Z}_{\geq 0}$-degree of homogeneous elements of $\mathscr{B}(V^*)$ by $m$, and we call $m$ the order of $\overleftarrow{D}_x$.

\end{rem}

Let $v\in V$. The braided right partial derivative $\overleftarrow{D}_v$ is uniquely determined via the condition $(\psi)\overleftarrow{D}_v=\psi(v)$ for all $\psi\in V^*$ and the so-called braided Leibniz rule
\[
\boxed{\,\;(\phi\psi)\overleftarrow{D}_v=\phi\big((\psi)\overleftarrow{D}_v\big)+\big((\phi)\overleftarrow{D}_{v_{(0)}}\big)\big(\mathscr{S}^{-1}(v_{(-1)})\psi\big)\vphantom{\int}\;\,}
\]
for all $\phi,\psi\in\mathscr{B}(V^*)$.

Finally, let $x\in\mathscr{B}(V)$ and let $\xi\in\mathscr{B}(V^*)$. Then, we note that $\overrightarrow{D}_\xi$ and multiplication by $\xi$ from the right in $\mathscr{B}(V^*)$ are adjoint with respect to the Hopf duality pairing, and similarly that $\overleftarrow{D}_x$ and multiplication by $x$ from the left in $\mathscr{B}(V)$ are adjoint. In formulas, we have
\[
\langle\phi,\overrightarrow{D}_\xi(y)\rangle=\left<\phi\xi,y\right>\quad\text{and}\quad\langle(\phi)\overleftarrow{D}_x,y\rangle=\left<\phi,xy\right>
\]
for all $y\in\mathscr{B}(V)$ and for all $\phi\in\mathscr{B}(V^*)$. In particular, these formulas yield the equation
\begin{equation}
\label{eq:pairing-derivative}
\left<\xi,x\right>=\epsilon(\overrightarrow{D}_\xi(x))=\epsilon((\xi)\overleftarrow{D}_x)\,.
\end{equation}

\todo[inline,color=green]{Axiomatic definition of Nichols algebras -- seems to be a point that I can make \enquote{green} in the future. I don't really know where I used these properties in the text. I used probably less, namely $\mathscr{B}(V)^1\subseteq P(\mathscr{B}(V))$ -- which is obvious, but wanted to mention the whole thing. [I use this inclusion in the proof of Lemma~\ref{lem:prelim-exp} and in the proof of Theorem~\ref{thm:first-step}.] For example, the descend of Proposition~\ref{prop:rhoands}\eqref{item:identity} from $T(V)$ to $\mathscr{B}(V)$ as in Remark~\ref{rem:rhoandsonnichols} works only because of the first point of the axiomatics.}

\todo[inline,color=green]{Notation on Sweedler notation. Same convention as Convention~\ref{conv:generic} holds for $\delta$ -- the coaction.}

\todo[inline,color=green]{For the proof of Proposition~\ref{prop:reflection-ordering} it is somewhat relevant to know that the Hopf duality pairing between a dual Nichols algebra and itself is actually a morphism in the Yetter-Drinfeld category where $\mathbf{k}$ is considered as a trivial $H$-module. I should explain this somewhere and maybe even refer to it in the proof of Proposition~\ref{prop:reflection-ordering} (to be decided!). In the graded case we are looking at for Nichols algebra, the Hopf duality pairing even respects the grading in the sense of Remark~\ref{rem:vanishing}. [I mentioned now (22.03.2018) both facts in the proof of Proposition~\ref{prop:reflection-ordering}: morphism in $\mathscr{YD}$ and Remark~\ref{rem:vanishing}.]}

\todo[inline,color=green]{I use the structure of a dual $V^*$ of a finite dimensional Yetter-Drinfeld $H$-module freely everywhere. It must be introduced and said that $*$ has precisely this meaning -- with Yetter-Drinfeld $H$-module structure implicit and understood. [Realized as convention after the paragraph \enquote{A YD-$H$-module is rigid iff it is finite dimensional.}.]}

\todo[inline,color=green]{Maybe I should by passing by mention the axiomatic definition of the Nichols algebra (with reference), in the regard that I often use the axiom that the primitive elements of the Nichols algebra are given by the degree one component. We also use that the degree one part is actually the same as the defining $V$, etc.}

\todo[inline,color=green]{It will be clear from context that all algebras are associative with unit (but clearly not necessarily commutative).}

\section{Yetter-Drinfeld category over a group}
\label{sec:group}

In this section, we specialize the setting from the previous section from an arbitrary Hopf algebra with invertible antipode to the group algebra of a group $\Gamma$. We investigate structure specific to the free braided group and the Nichols algebra associated to a finite dimensional Yetter-Drinfeld module $V$ over $\mathbf{k}\Gamma$, most notably the maps $\rho$ and $\bar{\mathscr{S}}$ as endomorphisms of $T(V)$ and $\mathscr{B}(V)$ (cf.\ Definition~\ref{def:rhoandsbar} and Remark~\ref{rem:rhoandsonnichols}). These maps already occur in \cite{liu} with the sole difference that there a nonstandard braiding of ${}^{\mathbf{k}\Gamma}_{\mathbf{k}\Gamma}\mathscr{YD}$ (different from the braiding introduced in Section~\ref{sec:hopf}, to be precise: its inverse) is in use. Nevertheless, the results and their proofs are analogous. With reference to \cite{liu}, we sometimes leave the obvious modifications to the reader.

Throughout the discussion, we fix a group $\Gamma$. From now on, we specialize to the setting where $H$ is given by the group algebra $\mathbf{k}\Gamma$. For brevity, we write ${}_\Gamma^\Gamma\mathscr{YD}={}_{\mathbf{k}\Gamma}^{\mathbf{k}\Gamma}\mathscr{YD}$, and call ${}_\Gamma^\Gamma\mathscr{YD}$ the Yetter-Drinfeld category over $\Gamma$. We call the objects of ${}_\Gamma^\Gamma\mathscr{YD}$ Yetter-Drinfeld modules over $\Gamma$ or simply Yetter-Drinfeld $\Gamma$-modules.

A Yetter-Drinfeld $\Gamma$-module $V$ will be identified with a $\Gamma$-module $V$ equipped with a $\Gamma$-grading $V=\bigoplus_{g\in\Gamma}V_g$ of $\mathbf{k}$-vector spaces such that the compatibility condition $hV_g=V_{hgh^{-1}}$ holds for all $g,h\in\Gamma$. Here, we set $V_g=\{v\in V\mid\delta(v)=g\otimes v\}$. All other structure attached to this situation as in the last section, such as the tensor product and the braiding in ${}_\Gamma^\Gamma\mathscr{YD}$, specializes nicely under this identification. The explicit formulas are written in \cite[Subsection~4.1]{bazlov1}. For example, the left dual $V^*$ of a finite dimensional Yetter-Drinfeld $\Gamma$-module $V$ has $\Gamma$-grading given by $V^*=\bigoplus_{g\in\Gamma}V_g^*$ where $V_g^*$ is the linear dual of $V_g$ for all $g\in\Gamma$.

\begin{rem}
\label{rem:equivariant}

Let $A$ and $B$ be braided Hopf algebras in ${}_\Gamma^\Gamma\mathscr{YD}$. Let $\left<-,-\right>$ be a Hopf duality pairing between $A$ and $B$. Then, the Hopf duality pairing is $\Gamma$-equivariant, in the sense that $\left<g\phi,gx\right>=\left<\phi,x\right>$ for all $g\in\Gamma,\phi\in A,x\in B$. Indeed, this follows directly because the Hopf duality pairing is a morphism in the Yetter-Drinfeld category over $\Gamma$.  

\end{rem}

\todo[inline,color=green]{It depends pretty much on the latter proofs whether I have to specialize the braided Leibniz rule to the group setting (keeping the convention of the $\Gamma$-right-action in mind) or not. I decide this later. [See Remark~\ref{rem:braided-leibniz}.]}

\todo[inline,color=green]{Terminology: Yetter-Drinfeld $\Gamma$-module, Yetter-Drinfeld category over $\Gamma$, compatibility condition, $\mathbf{k}$ is left out of the notation in the category (and elsewhere), I don't need to assume that $\Gamma$ is finite a priori -- fixed throughout the discussion.}

\todo[inline,color=green]{$\mathbf{k},\otimes=\otimes_{\mathbf{k}},\nabla,\tau,\Psi,\Delta,\epsilon,\eta,\delta,T(V),\left<-,-\right>\colon T(V^*)\otimes T(V)\to\mathbf{k},\mathrm{Wor}$, Sweedler notation, we usually suppress the indices in natural transformations, total Woronowicz symmetrizer = Woronowicz ideal $\mathrm{Wor}(V)$ (of $V$) (later on I will need: $[m]!_V$ -- using this notation -- it seems that I never suppress $V$ in $\mathrm{Wor}(V)$ or in $[m]!_V$ -- this rule [not to suppress $V$] applies only to $\mathrm{Wor}(V)$, in the case of $[m]!$ I suppress $V$ for example in Example~\ref{ex:relations}), braided $\mathbb{Z}_{\geq 0}$-graded Hopf algebra in the Yetter-Drinfeld category over $H$ -- in general, one says: natural equivalence or natural isomorphism.}

\todo[inline,color=green]{$\mathscr{B}(V),\protect\overrightarrow{D}_\xi,\protect\overleftarrow{D}_x$, braided left and right partial derivatives, free braided group, braided tensor product algebra with twisted multiplication or braided multiplication (I stay with the term twisted multiplication everywhere, for example, so far, I have never said braided multiplication, braided product, or twisted product [the term \enquote{twisted product} is also taboo and redundant because I would call it simply coproduct in all my formulas], although I use both the terms coproduct and comultiplication to denote either $\Delta(x)$ or $\Delta$ itself, i.e.\ the result or the operation).}

\todo[inline,color=green]{Convention: Many steps in our proofs are analogously to the ones in \cite{liu}. We have adapted the formulas and refer to the proofs of Liu which can be in each case repeated in one or another form (although the results are literally not the same concerning the statements and the proofs). The difference is due to a different braiding chosen in \cite{liu}. No ideas are needed to adjust those proofs. If the proof somehow differs, we have given a complete proof, like in the proof of Proposition~\ref{prop:adjoint}.}

\todo[inline,color=green]{I mention the braided Leibniz rule here in the preamble. There is nothing to say about the circle action as $g$ is not necessarily an involution. [I actually never mention the braided Leibniz for the opposite braided left partial derivatives -- although it would be possible to write it out. I also don't mention the braided Leibniz rule for $\partial_\alpha$ -- which would be a trivial special case.]}

\begin{defn}

Let $V$ be a Yetter-Drinfeld $\Gamma$-module. Then we can define a new Yetter-Drinfeld $\Gamma$-module $V^{\mathrm{op}}$ which is equal to $V$ as $\Gamma$-module but the $\Gamma$-grading is defined as $V^{\mathrm{op}}_g=V_{g^{-1}}$ for all $g\in\Gamma$, i.e.\ the coaction is given by $\delta^{\mathrm{op}}=(\mathscr{S}\otimes 1)\circ\delta$. It is clear that the compatibility condition is satisfied for $V^{\mathrm{op}}$. 

\end{defn}


\begin{rem}

Let $V$ and $W$ be Yetter-Drinfeld $\Gamma$-modules. A morphism $\rho\colon V^{\mathrm{op}}\to W$ (or equivalently a morphism $\rho\colon V\to W^{\mathrm{op}}$) in the Yetter-Drinfeld category over $\Gamma$ is simply a $\Gamma$-module homomorphism $V\to W$ which satisfies in addition $\delta\circ\rho=(\mathscr{S}\otimes\rho)\circ\delta$.

\end{rem}

\begin{conv}
\label{conv:homogeneous}



Let $A$ be a braided $\mathbb{Z}_{\geq 0}$-graded Hopf algebra in the Yetter-Drinfeld category over $\Gamma$. We call an element $x\in A$ homogeneous if it is homogeneous with respect to the $\mathbb{Z}_{\geq 0}$-grading and with respect to the $\Gamma$-grading.

\end{conv}

\begin{defn}
\label{def:rhoandsbar}

Let $V$ be a Yetter-Drinfeld $\Gamma$-module of dimension $n$. Let $v_1,\ldots,v_n$ be a basis of $V$. Then we define a $\mathbf{k}$-linear map $\rho\colon T(V)\to T(V)$ by setting $\rho(v_{i_1}\cdots v_{i_m})=v_{i_m}\cdots v_{i_1}$ for each sequence $1\leq i_1,\ldots, i_m\leq n$ and by extending linearly. Furthermore, we define a $\mathbf{k}$-linear map $\bar{\mathscr{S}}\colon T(V)\to T(V)$ by setting $\bar{\mathscr{S}}(x)=(-1)^m\rho(\mathscr{S}(x))$ for each homogeneous element $x\in T(V)$ of $\mathbb{Z}_{\geq 0}$-degree $m$ and by extending linearly. It is easy to see that the map $\rho$ and consequently the map $\bar{\mathscr{S}}$ do not depend on the choice of the basis of $V$.

\end{defn}

\begin{conv}

Similarly as in Convention~\ref{conv:generic}, we use the notation $\rho$ and $\bar{\mathscr{S}}$ generically for the maps just defined on free braided groups, regardless which free braided group (see also Remark~\ref{rem:rhoandsonnichols}).

\end{conv}





\begin{prop}[{\cite[Proposition~2.10(a),(b),(c)]{liu}}]
\label{prop:rhoands}

Let $V$ be a Yetter-Drinfeld $\Gamma$-module. We have the following properties of $\rho$ and $\bar{\mathscr{S}}$:

\begin{enumerate}

\item
\label{item:morphop}

The maps $\rho$ and $\bar{\mathscr{S}}$ are $\mathbb{Z}_{\geq 0}$-graded morphisms $T(V)^{\mathrm{op}}\to T(V)$ in the Yetter-Drinfeld category over $\Gamma$.

\item
\label{item:identity}

The maps $\rho$ and $\bar{\mathscr{S}}$ are the identity in $\mathbb{Z}_{\geq 0}$-degree zero and one.

\item
\label{item:rhoantialgebra}

The map $\rho$ is an anti algebra homomorphism, i.e.\ we have $\rho\circ\nabla=\nabla\circ\tau\circ(\rho\otimes\rho)$, i.e.\ we have $\rho(xy)=\rho(y)\rho(x)$ for all $x,y\in T(V)$.

\item
\label{item:sbaralebgraprop}

Let $x\in T(V)$ be a homogeneous element of $\Gamma$-degree $g$. Then we have $\bar{\mathscr{S}}(xy)=\bar{\mathscr{S}}(x)(g\bar{\mathscr{S}}(y))$ for all $y\in T(V)$. Diagrammatically, this means that $\bar{\mathscr{S}}\circ\nabla=\nabla\circ\tau\circ\Psi\circ(\bar{\mathscr{S}}\otimes\bar{\mathscr{S}})$.

\item
\label{item:sbaranticolagebra}

The map $\bar{\mathscr{S}}$ is an anti coalgebra homomorphism, i.e.\ we have $\Delta\circ\bar{\mathscr{S}}=(\bar{\mathscr{S}}\otimes\bar{\mathscr{S}})\circ\tau\circ\Delta$, i.e.\ we have $\bar{\mathscr{S}}(x)_{(1)}\otimes\bar{\mathscr{S}}(x)_{(2)}=\bar{\mathscr{S}}(x_{(2)})\otimes\bar{\mathscr{S}}(x_{(1)})$ for all $x\in T(V)$.

\item
\label{item:involution}

The maps $\rho$ and $\bar{\mathscr{S}}$ are involutions, i.e.\ $\rho^2=\bar{\mathscr{S}}^2=1$. The antipode $\mathscr{S}\colon T(V)\to T(V)$ is invertible.

\end{enumerate}

\end{prop}

\begin{proof}

Item~\eqref{item:morphop},\eqref{item:identity},\eqref{item:rhoantialgebra} are obvious from the definition of the maps $\rho$ and $\bar{\mathscr{S}}$. Item~\eqref{item:sbaralebgraprop} corresponds to \cite[Proposition~2.10(a)]{liu} and can be deduced immediately from Remark~\ref{rem:anti-homo}. Item~\eqref{item:sbaranticolagebra} corresponds to \cite[Proposition~2.10(c)]{liu}. Ad~Item~\eqref{item:involution}. The statement that the map $\rho$ is an involution can be easily seen from the original definition or using induction on the $\mathbb{Z}_{\geq 0}$-degree and Item~\eqref{item:rhoantialgebra}. The statement that the map $\bar{\mathscr{S}}$ is an involution corresponds to \cite[Proposition~2.10(b)]{liu}. The statement that the antipode on $T(V)$ is invertible follows from the definition of $\bar{\mathscr{S}}$ and the fact that the maps $\rho$ and $\bar{\mathscr{S}}$ are invertible (as they are involutions as it was just explained).
\end{proof}

\begin{rem}
\label{rem:comp-sbar}

Let $V$ be a Yetter-Drinfeld $\Gamma$-module. Let $v_1,\ldots,v_m\in V$ be homogeneous elements of $\Gamma$-degrees $g_1,\ldots,g_m$ respectively. Then it follows immediately from Proposition~\ref{prop:rhoands}\eqref{item:identity},\eqref{item:sbaralebgraprop} that $\bar{\mathscr{S}}(v_1\cdots v_m)=v_1(g_1v_2)\cdots(g_1\cdots g_{m-2}v_{m-1})(g_1\cdots g_{m-1}v_m)$.

\end{rem}

\begin{rem}
\label{rem:degree01}

Let $V$ be a finite dimensional Yetter-Drinfeld $\Gamma$-module. Then, it follows immediately from Proposition~\ref{prop:rhoands}\eqref{item:morphop},\eqref{item:identity} that $\left<\phi,\rho(x)\right>=\left<\phi,\bar{\mathscr{S}}(x)\right>=\left<\phi,x\right>$ for all $x\in T(V)$ and for all homogeneous $\phi\in T(V^*)$ of $\mathbb{Z}_{\geq 0}$-degree less or equal than one. (By the explicit construction of the Hopf duality pairing between $T(V^*)$ and $T(V)$, note that all three terms of the previous equation are zero if $\phi$ is homogeneous of $\mathbb{Z}_{\geq 0}$-degree $>1$.)

\end{rem}

\begin{prop}[{\cite[Proposition~2.10(e)]{liu}}]
\label{prop:adjoint}

Let $V$ be a finite dimensional Yetter-Drinfeld $\Gamma$-module. Then the morphism $\rho$ and $\bar{\mathscr{S}}$ are adjoint with respect to the Hopf duality pairing in the sense that we have $\left<\phi,\bar{\mathscr{S}}(x)\right>=\left<\rho(\phi),x\right>$ and $\left<\bar{\mathscr{S}}(\phi),x\right>=\left<\phi,\rho(x)\right>$ for all $x\in T(V)$ and all $\phi\in T(V^*)$.

\end{prop}

\begin{proof}

Let $x\in T(V)$ and $\phi\in T(V^*)$. We prove the first claimed equality. Under the identification $V^{**}=V$ (cf.~\cite[Proposition~9.3.2]{majid}), the second follows from the first applied to $V^*$. To this end, by linearity, we may assume that $\phi=f_1\cdots f_m$ for some $f_i\in V^*$. In view of Proposition~\ref{prop:rhoands}\eqref{item:rhoantialgebra},\eqref{item:sbaranticolagebra} and Remark~\ref{rem:degree01}, we compute
\begin{align*}
\left<\phi,\bar{\mathscr{S}}(x)\right>&=\left<f_1,\bar{\mathscr{S}}(x)_{(m)}\right>\cdots\left<f_m,\bar{\mathscr{S}}(x)_{(1)}\right>\\
&=\left<f_1,\bar{\mathscr{S}}(x_{(1)})\right>\cdots\left<f_m,\bar{\mathscr{S}}(x_{(m)})\right>\\
&=\left<f_m,x_{(m)}\right>\cdots\left<f_1,x_{(1)}\right>\\
&=\left<f_m\cdots f_1,x\right>=\left<\rho(\phi),x\right>\,.\qedhere
\end{align*}
\end{proof}


\begin{cor}
\label{cor:ideal}

Let $V$ be a finite dimensional Yetter-Drinfeld $\Gamma$-module. Let $I=\mathrm{Wor}(V)$ be the Woronowicz ideal of $V$. Then we have $\mathscr{S}(I)=\rho(I)=\bar{\mathscr{S}}(I)=I$. 

\end{cor}


\begin{proof}

From Proposition~\ref{prop:adjoint}, it is immediate that $\rho(I)\subseteq I$ and that $\bar{\mathscr{S}}(I)\subseteq I$. From Proposition~\ref{prop:rhoands}\eqref{item:involution}, we see that both inclusions are actually equalities. Since $I$ is a Hopf ideal, we know in advance that $\mathscr{S}(I)\subseteq I$. Since $\mathscr{S}$ is self-adjoint and invertible (cf. Proposition~\ref{prop:rhoands}\eqref{item:involution}), it is clear that $\mathscr{S}^{-1}$ is also self-adjoint. Hence, we find that $\mathscr{S}^{-1}(I)\subseteq I$ which gives the desired equality for the antipode.
\end{proof}

\begin{rem}

Let $V$ be as in the statement of Corollary~\ref{cor:ideal}. The antipode on $T(V)$ is actually adjoint to the antipode on $T(V^*)$ and similarly for their inverses. We used the term \enquote{self-adjoint} in a sloppy but suggestive way in the proof of Corollary~\ref{cor:ideal}.

\end{rem}

\begin{rem}
\label{rem:rhoandsonnichols}

Let $V$ and $I$ be as in the statement of Corollary~\ref{cor:ideal}. In view of Corollary~\ref{cor:ideal}, we see that the maps $\rho$ and $\bar{\mathscr{S}}$ pass from the free braided group $T(V)$ to the quotient $\mathscr{B}(V)$ simply by defining $x+I\mapsto\rho(x)+I$ and $x+I\mapsto\bar{\mathscr{S}}(x)+I$ where $x\in T(V)$. We denote the maps on the level of the quotient again by $\rho$ and $\bar{\mathscr{S}}$ and take care that no confusion arises. In particular, all the properties discussed so far for the maps $\rho$ and $\bar{\mathscr{S}}$ on the level of the free braided group (e.g.\ Proposition~\ref{prop:rhoands}) hold equally well on the level of the Nichols algebra. If we need one of the above properties of $\rho$ or $\bar{\mathscr{S}}$ on the level of the Nichols algebra, we will simply refer to one of the results explicitly stated on the level of the free braided group.

\end{rem}


\begin{defn}

Let $V$ be a finite dimensional Yetter-Drinfeld $\Gamma$-module. We define a left action of the algebra $\mathscr{B}(V)$ on $\mathscr{B}(V^*)$ by sending an element $x\in\mathscr{B}(V)$ to the $\mathbf{k}$-linear endomorphism $\overrightarrow{D}_x^\circ$ of $\mathscr{B}(V^*)$ defined by the equation $\overrightarrow{D}_x^\circ(\phi)=(\phi)\overleftarrow{D}_{\rho(x)}$ for all $\phi\in\mathscr{B}(V^*)$. This algebra action is well-defined in view of Proposition~\ref{prop:rhoands}\eqref{item:identity},\eqref{item:rhoantialgebra}. Similarly, we define a right action of the algebra $\mathscr{B}(V^*)$ on $\mathscr{B}(V)$ by sending an element $\xi\in\mathscr{B}(V^*)$ to the $\mathbf{k}$-linear endomorphism $\overleftarrow{D}_\xi^\circ$ of $\mathscr{B}(V)$ defined by the equation $(x)\overleftarrow{D}_\xi^\circ=\overrightarrow{D}_{\rho(\xi)}(x)$ for all $x\in\mathscr{B}(V)$. We call $\overrightarrow{D}_x^\circ$ and $\overleftarrow{D}_\xi^\circ$ where $x\in\mathscr{B}(V)$ and $\xi\in\mathscr{B}(V^*)$ opposite braided left and right partial derivatives.


\end{defn}

\begin{conv}
\label{conv:homothety}

Let $V$ be a Yetter-Drinfeld $\Gamma$-module. Via the antipode on $\mathbf{k}\Gamma$, the vector space $V$ is also naturally equipped with a right $\Gamma$-action.
Explicitly, we define $vg=g^{-1}v$ for all $g\in\Gamma$ and all $v\in V$. For a given $g\in\Gamma$, we will denote by the same symbol $g\in\operatorname{Aut}_{\mathbf{k}}(V)$ either the homothety $v\mapsto gv$ or the homothety $v\mapsto vg$. It will be clear from the context what is meant by $g$ -- either an element of $\Gamma$ or one of the two homotheties. It will also be clear which of the two homotheties is meant.

\end{conv}

\begin{rem}
\label{rem:conjugation}

Let $V$ be a finite dimensional Yetter-Drinfeld $\Gamma$-module. With the Convention~\ref{conv:homothety} in mind, it is easy to see that we have $g\overrightarrow{D}_\xi g^{-1}=\overrightarrow{D}_{g\xi}$ and $g\overleftarrow{D}_x g^{-1}=\overleftarrow{D}_{gx}$ for all $g\in\Gamma,x\in\mathscr{B}(V),\xi\in\mathscr{B}(V^*)$. The equalities are meant to be as $\mathbf{k}$-linear endomorphisms of $\mathscr{B}(V)$ respectively of $\mathscr{B}(V^*)$. A similar remark applies to the opposite braided left and right partial derivatives. We have $g\overrightarrow{D}_x^\circ g^{-1}=\overrightarrow{D}_{gx}^\circ$ and $g\overleftarrow{D}_\xi^\circ g^{-1}=\overleftarrow{D}_{g\xi}^\circ$ for all $g\in\Gamma,x\in\mathscr{B}(V),\xi\in\mathscr{B}(V^*)$. To see the claims in this remark, we use Proposition~\ref{prop:rhoands}\eqref{item:morphop} and Remark~\ref{rem:equivariant}.

\end{rem}

\begin{rem}[Elaboration of Remark~\ref{rem:degree-partial}]
\label{rem:elab}

Let $V$ be a finite dimensional Yetter-Drinfeld $\Gamma$-module. Let $\xi\in\mathscr{B}(V^*)$ be homogeneous of 
$\Gamma$-degree $g$. Then, the operator $\overrightarrow{D}_\xi$ 
multiplies the $\Gamma$-degree of homogeneous elements of $\mathscr{B}(V)$ by $g$ from the left. Moreover, by Proposition~\ref{prop:rhoands}\eqref{item:morphop}, the operator $\overleftarrow{D}_\xi^\circ$ 
multiplies the $\Gamma$-degree of homogeneous elements of $\mathscr{B}(V)$ by $g^{-1}$ from the left. Similarly, let $x\in\mathscr{B}(V)$ be homogeneous of 
$\Gamma$-degree $g$. Then, the operator $\overleftarrow{D}_x$ 
multiplies the $\Gamma$-degree of homogeneous elements of $\mathscr{B}(V^*)$ by $g$ from the right. Moreover, by proposition loc.\ cit., the operator $\overrightarrow{D}_x^\circ$
multiplies the $\Gamma$-degree of homogeneous elements of $\mathscr{B}(V^*)$ by $g^{-1}$ from the right.

\end{rem}

\begin{rem}[Braided Leibniz rule over $\Gamma$]
\label{rem:braided-leibniz}

Let $V$ be a finite dimensional Yetter-Drinfeld $\Gamma$-module. Let $v\in V$ be homogeneous of $\Gamma$-degree $g$ (where $V$ is considered as $\mathbb{Z}_{\geq 0}$-degree one component of $\mathscr{B}(V)$). In view of Convention~\ref{conv:homothety}, the braided Leibniz rule now takes the form
\[
(\phi\psi)\overleftarrow{D}_v=\phi((\psi)\overleftarrow{D}_v)+((\phi)\overleftarrow{D}_v)(\psi g)
\]
for all $\phi,\psi\in\mathscr{B}(V^*)$.


\end{rem}

\todo[inline,color=green]{At this point, I should exemplify the braided Leibniz rule over $\Gamma$ only for the braided right partial derivatives (something else does not make sens, in particular not for the opposite left partial derivative because $g$ is not necessarily an involution) [Remark~\ref{rem:braided-leibniz}], and explain what structure $\protect\overrightarrow{D}_\xi$ or $\protect\overleftarrow{D}_x$ for homogeneous elements $x,\xi$ has, i.e.\ $\Gamma$-degree gets multiplied by $g=\mathrm{deg}(\bullet)$ and $\mathbb{Z}_{\geq 0}$-degree gets subtracted by $m$ (which is true in general, so, what I mention here holds in addition over $\Gamma$, and the $m$-subtraction is already explained above) [Remark~\ref{rem:degree-partial} and~\ref{rem:elab}]. I don't need to say anything about the circle action as this becomes obvious from the other remarks (to be decided!: I decided to spell out at least the part in this section, when the opposite version is already introduced, the $m$-subtraction is left, cf. remark after Definition~\ref{def:braided-skew}). Action by $\mathbf{k}$-linear endomorphisms stays the right framework (even in the subsection on the \enquote{Braided factorial}, cf. commented passage there).}

\section{The Nichols algebra associated to a finite Coxeter group}
\label{sec:bazlov}

In this section, we recall the construction of the Yetter-Drinfeld module and the Nichols algebra $\mathscr{B}_W$ associated to a finite Coxeter group $W$ which are due to Bazlov \cite{bazlov1}. We also recall the embedding of the nilCoxeter algebra into the Nichols algebra $\mathscr{B}_W$ which correspond to a result of \cite[Theorem~6.3]{bazlov1}. This will be the framework in which we work from the next section on and for the rest of this article. 

From now on and throughout the discussion, we fix a finite Coxeter system $(W,S)$ with all of the notation attached to this situation as in Subsection~\ref{subsec:coxeter}.

\subsection{The Yetter-Drinfeld module associated to a finite Coxeter group \texorpdfstring{{\normalfont(\cite[Subsection~4.2]{bazlov1})}}{([\ref{bib-bazlov1},~Subsection~4.2])}}

Let $[\alpha]$ where $\alpha\in R$ be a collection of linearly independent symbols. We define a $\mathbf{k}$-vector space $V_W$ by the equation 
$$
V_W=\bigoplus_{\alpha\in R}\mathbf{k}[\alpha]\bigg/\mathrm{span}_{\mathbf{k}}\{[\alpha]+[-\alpha]\mid\alpha\in R\}\,.
$$
For all $\alpha\in R$, let $x_\alpha$ be the image of $[\alpha]$ in $V_W$. We clearly have $x_{-\alpha}=-x_{\alpha}$ for all $\alpha\in R$. It is also clear that the collection $x_\alpha$ where $\alpha\in R^+$ forms a basis of $V_W$. We introduce a $W$-action $V_W$ by defining $wx_\alpha=x_{w(\alpha)}$ for all $w\in W$ and all $\alpha\in R$. We introduce a $W$-grading on $V_W$ by assigning for all $\alpha\in R$ to $x_\alpha$ the $W$-degree $s_\alpha$. Both the $W$-action as well as the $W$-grading are clearly well-defined. Moreover, they satisfy the compatibility condition. In this way, the vector space $V_W$ becomes a Yetter-Drinfeld $W$-module. We call $V_W$ the Yetter-Drinfeld module associated to the finite Coxeter group $W$.

\subsection{Identification \texorpdfstring{{\normalfont(\cite[Subsection~4.4]{bazlov1})}}{([\ref{bib-bazlov1},~Subsection~4.4])}}

In case of the Yetter-Drinfeld $W$-module $V_W$, we have a canonical choice of a homogeneous basis, i.e.\ of a basis consisting of homogeneous vectors, namely $x_\alpha$ where $\alpha\in R^+$. This choice of a basis induces a canonical isomorphism $V_W\stackrel{\sim}{\to} V_W^*$ of Yetter-Drinfeld $W$-modules by sending for each $\alpha\in R^+$ the basis vector $x_\alpha$ to its natural dual $x_\alpha^*$. As in \cite[Subsection~4.4]{bazlov1}, we will from now on identify $V_W^*$ with $V_W$ via the inverse of this isomorphism. Under this identification, $\Psi_{V_W^*,V_W^*}$ becomes identified with $\Psi_{V_W,V_W}$, $\mathrm{Wor}(V_W^*)$ with $\mathrm{Wor}(V_W)$, and $\mathscr{B}(V_W^*)$ with $\mathscr{B}(V_W)$. We simply write $\mathscr{B}_W=\mathscr{B}(V_W)=\mathscr{B}(V_W^*)$ and call $\mathscr{B}_W$ the Nichols algebra associated to the finite Coxeter group $W$. This identification makes the Hopf duality pairing between $\mathscr{B}(V_W^*)$ and $\mathscr{B}(V_W)$ into a nondegenerate symmetric bilinear form $\mathscr{B}_W\otimes\mathscr{B}_W\to\mathbf{k}$ whose restriction to $V_W\otimes V_W$ is given by $x_\alpha\otimes x_{\alpha'}\mapsto\delta_{\alpha,\alpha'}$ for all $\alpha,\alpha'\in R^+$.

\begin{rem}[{\cite[Subsubsection~2.3.3]{liu}}]
\label{rem:left-right-simult}

Using the coassociativity in $\mathscr{B}_W$, it is easy to see that we have
$$
(\overrightarrow{D}_\xi(\phi))\overleftarrow{D}_x=\overleftarrow{D}_\xi((\phi)\overrightarrow{D}_x)=\left<\xi,\phi_{(1)}\right>\phi_{(2)}\left<\phi_{(3)},x\right>
$$
for all $x,\phi,\xi\in\mathscr{B}_W$. This equality can be rephrased by saying that the left and right action of $\mathscr{B}_W$ on itself via braided partial derivatives are compatible. Hence, we can simply write the above displayed expression as $\overrightarrow{D}_\xi(\phi)\overleftarrow{D}_x$. A similar remark applies to the actions via opposite braided left and right partial derivatives. We can simply write $\overrightarrow{D}_x^\circ(\phi)\overleftarrow{D}_\xi^\circ$ for all $x,\phi,\xi\in\mathscr{B}_W$. 

\end{rem}

\begin{prop}[{\cite[Proposition~2.10(d)]{liu}}]
\label{prop:liu2.10(d)}

We have the equation $\overrightarrow{D}_\xi(\bar{\mathscr{S}}(x))=\bar{\mathscr{S}}((x)\overleftarrow{D}_{\rho(\xi)})$ for all $x,\xi\in\mathscr{B}_W$.

\end{prop}

\begin{proof}

Let $x,\xi\in\mathscr{B}_W$. Bearing the symmetry of the Hopf duality paring $\mathscr{B}_W\otimes\mathscr{B}_W\to\mathbf{k}$ in mind, we easily compute using Proposition~\ref{prop:rhoands}\eqref{item:sbaranticolagebra} and Proposition~\ref{prop:adjoint} that
\begin{align*}
\overrightarrow{D}_\xi(\bar{\mathscr{S}}(x))&=\left<\xi,\bar{\mathscr{S}}(x)_{(1)}\right>\bar{\mathscr{S}}(x)_{(2)}\\
&=\left<\xi,\bar{\mathscr{S}}(x_{(2)})\right>\bar{\mathscr{S}}(x_{(1)})\\
&=\left<\rho(\xi),x_{(2)}\right>\bar{\mathscr{S}}(x_{(1)})\\
&=\bar{\mathscr{S}}(x_{(1)})\left<x_{(2)},\rho(\xi)\right>=\bar{\mathscr{S}}((x)\overleftarrow{D}_{\rho(\xi)})\,.\qedhere
\end{align*}
\end{proof}


\begin{rem}
\label{rem:liu2.10(d)}

We can state the result of Proposition~\ref{prop:liu2.10(d)} more compactly. For all $\xi\in\mathscr{B}_W$, we have $\overrightarrow{D}_\xi\circ\bar{\mathscr{S}}=\bar{\mathscr{S}}\circ\overrightarrow{D}_\xi^\circ$ as $\mathbf{k}$-linear endomorphisms of $\mathscr{B}_W$. By Proposition~\ref{prop:rhoands}\eqref{item:involution} we can equivalently say: For all $\xi\in\mathscr{B}_W$, we have $\bar{\mathscr{S}}\circ\overrightarrow{D}_\xi=\overrightarrow{D}_\xi^\circ\circ\bar{\mathscr{S}}$ as $\mathbf{k}$-linear endomorphisms of $\mathscr{B}_W$.

\end{rem}

\todo[inline,color=green]{Let $\alpha\in R$. Then we have $x_\alpha^2=0$ in $\mathscr{B}_W$. Indeed, it suffices to show that the element $x_\alpha\otimes x_\alpha\in V_W^{\otimes 2}$ lies in $\mathrm{Wor}(V_W)$, that is to say lies in $\mathrm{ker}([2]!_{V_W})$. To this end, we compute
\[
[2]!_{V_W}(x_\alpha\otimes x_\alpha)=(1+\Psi)(x_\alpha\otimes x_\alpha)=x_\alpha\otimes x_\alpha+x_{s_\alpha(\alpha)}\otimes x_\alpha=0
\]
where the last equality follows since $x_{s_\alpha(\alpha)}=x_{-\alpha}=-x_\alpha$. [First try for Example~\ref{ex:relations}.]}

\begin{ex}
\label{ex:relations}

We give examples of some simple relations which hold in $\mathscr{B}_W$. Let $\alpha,\gamma,\delta\in R$ be such that $\alpha$ and $\gamma$ are orthogonal. Then we have $x_\delta^2=0$ and $x_\alpha x_\gamma=x_\gamma x_\alpha$. Indeed, a simple computation shows that $x_\delta\otimes x_\delta$ and $x_\alpha\otimes x_\gamma-x_\gamma\otimes x_\alpha$ considered as elements of $V_W^{\otimes 2}$ lie in the kernel of $[2]!=1+\Psi$. Similarly, one can check the Fomin-Kirillov relation in $\mathscr{B}_W$ (i.e.\ the nontrivial defining relation of the Fomin-Kirillov algebra, cf.\ \cite[Proposition~2.1(d), Subsection~2.3]{liu}, \cite[Section~6, Equation~(ii)]{kirillov} or \cite[Definition~2.1, Equation~(2.2)]{quadratic}). Indeed, let $\alpha,\gamma\in R$ be such that $\alpha$ and $\gamma$ span a root subsystem of $R$ of type $\mathsf{A}_2$ with base $\{\alpha,\gamma\}$. Then we have
\[
x_\alpha x_{\alpha+\gamma}+x_{\alpha+\gamma}x_\gamma-x_\gamma x_\alpha=0\,.
\]

\end{ex}

\todo[inline,color=green]{In the previous example, I am supposed to add further references for the Fomin-Kirillov relations. One of them might be \cite{kirillov}. Another one is yet to be inserted as soon I really need it as background material. (The one I call \enquote{quadratic-algebras} in my file system.)}

\subsection{NilCoxeter algebra} 

From now on and for the rest of the article, we specialize to the case where $\mathbf{k}=\mathbb{C}$. We do so because the results of \cite{bazlov1} which we need to use are formulated in this situation and we want to match up with the references. Strictly speaking, many of the results in \cite{bazlov1} (except the embedding of the coinvariant algebra of $W$ into $\mathscr{B}_W$ which we recall in Section~\ref{sec:passage}) and thus our considerations based on them stay valid for an arbitrary field. 
%
We now recall the definition and basic facts of the nilCoxeter algebra and its relation to the Nichols algebra associated to a finite Coxeter group after Bazlov. For more informations on the nilCoxeter algebra, its original motivation and its relation to Schubert polynomials via the \enquote{Schubert expression} in type $\mathsf{A}$, we refer the reader to \cite{nilcoxeter}.

Let $\mathscr{N}_W$ be the $\mathbb{C}$-algebra with one generator $e_\beta$ for each simple root $\beta$ subject to the relations $e_\beta^2=0$ and
\[
\underset{m(\beta,\beta')\text{ factors on each side}}{\underbrace{e_\beta e_{\beta'}e_\beta\cdots}=\underbrace{e_{\beta'}e_\beta e_{\beta'}\cdots}}
\]
for all $\beta,\beta'\in\Delta$. The algebra $\mathscr{N}_W$ is called the nilCoxeter algebra of $W$ and the relations among the generators of this algebra are called the nilCoxeter relations. Let $w$ be a Coxeter group element with reduced expression $s_{\beta_1}\cdots s_{\beta_\ell}$. Then we define $e_w=e_{\beta_1}\cdots e_{\beta_\ell}$. This element is well-defined, i.e.\ independent of the choice of the reduced expression of $w$, because of the nilCoxeter relations. The elements $e_w$ where $w$ runs through $W$ form a basis of $\mathscr{N}_W$ as a $\mathbb{C}$-vector space. The multiplication table with respect to this basis is given by $e_v e_w=e_{vw}$ if $\ell(vw)=\ell(v)+\ell(w)$ and $=0$ otherwise. By assigning to each vector $e_w$ the $\mathbb{Z}_{\geq 0}$-degree $\ell(w)$ and the $W$-degree $w$, the algebra $\mathscr{N}_W$ becomes a graded\footnote{Away from the Yetter-Drinfeld philosophy, graded means in this context graded with respect to $\mathbb{Z}_{\geq 0}$ and with respect to $W$. Similarly, as in Convention~\ref{conv:homogeneous}, we also call an element of $\mathscr{N}_W$ homogeneous if it is homogeneous with respect to the $\mathbb{Z}_{\geq 0}$-grading and with respect to the $W$-grading.\label{note:graded}} algebra with homogeneous basis $e_w$ where $w$ runs through $W$. In particular, the dimension of $\mathscr{N}_W$ is given by $|W|$. By \cite[Theorem~6.3(ii)]{bazlov1}, there exists a well-defined injective graded\cref{note:graded} algebra homomorphism $\mathscr{N}_W\hookrightarrow\mathscr{B}_W$ given on generators by $e_\beta\mapsto x_\beta$ where $\beta\in\Delta$, and extended linearly and multiplicatively. We denote the image of this homomorphism by $\tilde{\mathscr{N}}_W$. Also, we denote the image of $e_w$ in $\tilde{\mathscr{N}}_W$ by $x_w$ where $w\in W$. With this notation, the vectors $x_w$ where $w$ runs through $W$ become a homogeneous basis of $\tilde{\mathscr{N}}_W$. 

\begin{rem}

Directly by the definition of $\rho$ or by Proposition~\ref{prop:rhoands}\eqref{item:identity},\eqref{item:rhoantialgebra}, we see that $\rho(x_w)=x_{w^{-1}}$ for all $w\in W$. In particular, this implies that $\rho(\tilde{\mathscr{N}}_W)=\tilde{\mathscr{N}}_W$ by Proposition~\ref{prop:rhoands}\eqref{item:involution}. We will use this trivial observation repeatedly from now on (without referencing).

\end{rem}

\todo[inline,color=green]{$x_w$ forms a basis of the subalgebra of $\mathscr{B}_W$ isomorphic to $\mathscr{N}_W$ where $w$ runs through $W$, nilCoxeter relations. Note that $\rho$ restricts to an endomorphism of $\mathscr{N}_W$ which acts via inversion on the basis. [It is clear that $\rho|_{\tilde{\mathscr{N}}_W}\colon\tilde{\mathscr{N}}_W\to\tilde{\mathscr{N}}_W$ is a $\mathbb{Z}_{\geq 0}$-graded isomorphism of $\mathbb{C}$-vector spaces. It is \enquote{anti algebraic}. None of this is needed anywhere in the text. It is actually clear from Proposition~\ref{prop:rhoands} -- + $W$-degree inversion.] (This last fact is used without reference, since it is basic, especially in the section on \enquote{braided skew partial derivatives}. I also use the basis $\{x_w\}_{w\in W}$ of $\tilde{\mathscr{N}}_W$ repeatedly without reference.)}

\section{Braided skew partial derivatives}
\label{sec:skew}

In this section, we introduce the elements $x_{w/v}$ of $\mathscr{B}_W$ where $v,w\in W$ following the coproduct approach in \cite{liu} (Definition~\ref{def:xwv}). The braided partial derivatives induced by the elements $x_{w/v}$ are called braided skew partial derivatives (cf.\ Definition~\ref{def:braided-skew}). As we will see in Section~\ref{sec:passage}, these braided skew partial derivatives are a generalization of skew divided difference operators (as in the introduction, Definition~\ref{def:partialwv}) in the sense that their restrictions to $S_W$ (to be more precise: the restriction of the left version induced by the circle variant of $x_{w/v}$, cf.\ Notation~\ref{not:xwvcirc} and~\ref{not:dwv}) coincide. We investigate the basic properties of braided skew partial derivatives, e.g.\ the generalized braided Leibniz rule (Theorem~\ref{thm:leibniz}). Our list of properties is inspired by well-known classical properties of skew divided difference operators as they are presented for example in \cite[Chapter 2]{macdonald1991notes}. 

Let $\alpha_1,\ldots,\alpha_m\in R^+$. On the level of the tensor algebra $T(V_W)$, it easy to see that
$$
\Delta(x_{\alpha_1}\cdots x_{\alpha_m})\in\bigoplus_{1\leq i_1<\cdots<i_\ell\leq m}T(V_W)\otimes\mathbb{C}x_{\alpha_{i_1}}\cdots x_{\alpha_{i_\ell}}\,.
$$
This can be seen for example by induction on $m$. For the induction step where $m>0$, we expand 
$$
\Delta(x_{\alpha_1}\cdots x_{\alpha_m})=\Delta(x_{\alpha_1})\Delta(x_{\alpha_2}\cdots x_{\alpha_m})
$$ 
using the definition of the comultiplication restricted to $V_W$ (a vector space which consists of primitive elements of $T(V_W)$), the induction hypothesis applied to the second factor and the twisted multiplication in $T(V_W)\underline{\otimes}T(V_W)$.
Let $w\in W$. Let $s_{\beta_1}\cdots s_{\beta_m}$ be a reduced expression of $w$. If we apply the above expansion of the coproduct to the sequence of simple roots $\beta_1,\ldots,\beta_m$ and pass to the quotient $\mathscr{B}_W$, we find in view of the characterization of the Bruhat order in terms of reduced subexpressions (cf. \cite[Corollary~5.8(b), Theorem~5.10]{humphreys-coxeter}, note that all subexpressions $1\leq i_1<\cdots<i_\ell\leq m$ which do not result in a reduced expression $s_{\beta_{i_1}}\cdots s_{\beta_{i_\ell}}$ lead to a vanishing term $x_{\beta_{i_1}}\cdots x_{\beta_{i_\ell}}$ because of the nilCoxeter relations) that
$$
\Delta(x_w)\in\bigoplus_{v\leq w}\mathscr{B}_W\otimes\mathbb{C}x_v\,.
$$

\begin{defn}[{\cite[Proposition~2.7]{liu}}]
\label{def:xwv}

Let $w\in W$. Thanks to the preceding discussion, we can define for all $v\in W$ such that $v\leq w$ uniquely determined elements $x_{w/v}\in\mathscr{B}_W$ which satisfy $\Delta(w)=\sum_{v\leq w}x_{w/v}\otimes x_v$. Furthermore, for all $v\in W$ such that $v\not\leq w$, we define $x_{w/v}=0$.

\end{defn}

\begin{rem}
\label{rem:degree-xwv}

Let $v,w\in W$ be such that $v\leq w$. Since $\Delta\colon\mathscr{B}_W\to\mathscr{B}_W\underline{\otimes}\mathscr{B}_W$ is a $\mathbb{Z}_{\geq 0}$-graded algebra morphism in the Yetter-Drinfeld category over $W$,
it follows immediately from the above definition that $x_{w/v}$ is a homogeneous element of $\mathbb{Z}_{\geq 0}$-degree $\ell(v,w)$ and of $W$-degree $wv^{-1}$. To see this, just note that for each $u\in W$, the element $x_u$ is homogeneous of $\mathbb{Z}_{\geq 0}$-degree $\ell(u)$ and $W$-degree $u$.

\end{rem}

\begin{rem}
\label{rem:coproduct-xwv}

From the coassociativity of the comultiplication in $\mathscr{B}_W$, it is easy to see that $\Delta(x_{w/v})=\sum_{v\leq u\leq w}x_{w/u}\otimes x_{u/v}$ for all $v,w\in W$ such that $v\leq w$.


\end{rem}

\begin{ex}[{\cite[Chapter~2, Example~1,~2]{macdonald1991notes}}]
\label{ex:top-low-degree-xw}

Let $w\in W$. It follows directly from the definition that we have $x_{w/1}=x_w$ and $x_{w/w}=1$.

\end{ex}

\todo[inline,color=green]{What follows in this subsection should be revised. In particular, I think the remark on the superfluous notation should come directly after the notation itself (Notation~\ref{not:dwv}). [No superfluous notation anymore.] It is unclear what notation will be actually needed in the end. For example, one might think about $\protect\overleftarrow{D}_\alpha^\circ$ as well and similarly for $\protect\overrightarrow{D}_\alpha^\circ$. Although these are morally the same as already introduced.}

\begin{notation}
\label{not:xwvcirc}

For all $v,w\in W$, we write $x_{w/v}^\circ=\rho(x_{w^{-1}/v^{-1}})$. This \enquote{circle variant} of $x_{w/v}$ will occur every now and then in our treatment. 

\end{notation}

\begin{notation}
\label{not:dwv}

At this point, it is useful to introduce some new notation. For all $\alpha\in R$, we set $\overrightarrow{D}_\alpha=\overrightarrow{D}_{x_\alpha},\overleftarrow{D}_\alpha=\overleftarrow{D}_{x_\alpha},\overrightarrow{D}_\alpha^\circ=\overrightarrow{D}_{x_\alpha}^\circ,\overleftarrow{D}_\alpha^\circ=\overleftarrow{D}_{x_\alpha}^\circ$. For all $w\in W$, we set $\overrightarrow{D}_w=\overrightarrow{D}_{x_w},\overleftarrow{D}_w=\overleftarrow{D}_{x_w},\overrightarrow{D}_w^\circ=\overrightarrow{D}_{x_w}^\circ,\overleftarrow{D}_w^\circ=\overleftarrow{D}_{x_w}^\circ$. For all $v,w\in W$, we set $\overleftarrow{D}_{w/v}=\overleftarrow{D}_{x_{w/v}}$ and $\overrightarrow{D}_{w/v}^\circ=\overrightarrow{D}_{x_{w/v}^\circ}^\circ$.

\end{notation}





\begin{defn}
\label{def:braided-skew}

Let $v,w\in W$. We call the operators $\overrightarrow{D}_{w/v}^\circ$ and $\overleftarrow{D}_{w/v}$ braided left and right skew partial derivative. 


\end{defn}

\begin{rem}

Let $v,w\in W$ be such that $v\leq w$. If we extend the order terminology from Remark~\ref{rem:degree-partial} in the obvious way to opposite braided partial derivatives, then the order of $\overrightarrow{D}_{w/v}^\circ$ and $\overleftarrow{D}_{w/v}$ is in both cases $\ell(v,w)$. From Remark~\ref{rem:elab} and~\ref{rem:degree-xwv}, one can also easily infer how 
braided skew partial derivatives
modify the $W$-degree of homogeneous elements of $\mathscr{B}_W$.

\end{rem}

\todo[inline,color=green]{Even in the previous definition, I am supposed to give a name to $\protect\overrightarrow{D}_{w/v}^\circ$ and to differentiate between braided left and right skew partial derivatives.}






\begin{lem}
\label{lem:prelim-exp}

Let $w\in W$. Let $\beta\in\Delta$ be such that $s_\beta w<w$. Let $w'=s_\beta w$ for brevity. Then we have
\[
\Delta(x_w)=\sum_{\substack{v'\leq w'\\ s_\beta v'>v'}}x_\beta x_{w'/v'}\otimes x_{v'}+\sum_{\substack{v'\leq w'\\ s_\beta v'<v'}}(x_\beta x_{w'/v'}+s_\beta x_{w'/s_\beta v'})\otimes x_{v'}+\sum_{\substack{v\leq w\\\hphantom{'}v\not\leq w'\\ s_\beta v<v}}s_\beta x_{w'/s_\beta v}\otimes x_v\,.
\]
\end{lem}

\todo[inline,color=green]{where the summands in each of the three sums form together a system of linearly independent vectors of $\mathscr{B}_W^{\otimes 2}$. [I cannot really show this here or anywhere in this section because we don't know yet that $x_{w/v}\neq 0$ if $v\leq w$ (cf.\ Theorem~\ref{thm:strict-positivity}). Even if I already knew strict positivity, I had to use the next corollary which makes it trivial anyway, i.e.\ equivalent to Theorem~\ref{thm:strict-positivity}.]}

\begin{proof}

Let $w,\beta,w'$ be as in the statement. By definition of the comultiplication, we know that
\[
\Delta(x_w)=\Delta(x_\beta x_{w'})=(x_\beta\otimes 1+1\otimes x_\beta)\cdot\sum_{v'\leq w'}x_{w'/v'}\otimes x_{v'}
\]
where $\cdot$ denotes the twisted multiplication in $\mathscr{B}_W\underline{\otimes}\mathscr{B}_W$. If we multiply out, the first summand can be written as
\[
\sum_{\substack{v'\leq w'\\ s_\beta v'>v'}}x_\beta x_{w'/v'}\otimes x_{v'}+\sum_{\substack{v'\leq w'\\ s_\beta v'<v'}}x_\beta x_{w'/v'}\otimes x_{v'}
\]
and the second summand can be written as 
\[
\sum_{\substack{v'\leq w'\\ s_\beta v'>v'}}s_{\beta}x_{w'/v'}\otimes x_\beta x_{v'}
\] 
because all terms $v'\leq w'$ such that $s_\beta v'<v'$ vanish by the nilCoxeter relations. It is clear that we only have to examine the second summand more closely. Let $v'\in W$ be arbitrary such that $s_\beta v'>v'$. By \cite[Theorem~1.1(II)(ii)]{deodhar}, we then have $v'\leq w'$ if and only if $v\leq w$ where $v=s_\beta v'$. (Recall that clearly $w'<s_\beta w'=w$ by definition.) If we now replace $v'$ by $v=s_\beta v'$ in the summation in the previous displayed formula, we see because of the just mentioned equivalence that it equals 
\[
\sum_{\substack{v\leq w\\ s_\beta v<v}}s_\beta x_{w/s_\beta v}\otimes x_v\,.
\]
If we finally split this summation according to 
\[
\textstyle{\sum_{\substack{v'\leq w'\\ s_\beta v'<v'}}+\sum_{\substack{v\leq w\\\hphantom{'}v\not\leq w'\\ s_\beta v<v}}}
\]
and rearrange, we find the desired formula for $\Delta(x_w)$.
\end{proof}

\begin{cor}
\label{cor:prelim-exp}

Let $w\in W$. Let $\beta\in\Delta$ be such that $s_\beta w<w$. Let $w'=s_\beta w$ for brevity. Then we have
\begin{alignat*}{3}
x_{w/v'}&=x_\beta x_{w'/v'} &&\text{if}\quad v'\leq w',\text{ }s_\beta v'>v',\\
x_{w/v'}&=x_\beta x_{w'/v'}+s_\beta x_{w'/s_\beta v'}\quad &&\text{if}\quad v'\leq w',\text{ }s_\beta v'<v',\\
x_{w/v\hphantom{'}}&=s_\beta x_{w'/s_\beta v}&&\text{if}\quad v\hphantom{'}\leq w\hphantom{'},\text{ }s_\beta v\hphantom{'}<v\hphantom{'},\text{ }v\not\leq w'.
\end{alignat*}
Moreover, there exists no $v\in W$ such that $v\leq w$, $s_\beta v>v$, $v\not\leq w'$. This means that the three cases above are exhaustive.

\end{cor}

\begin{proof}

Let $w,\beta,w'$ be as in the statement. If we compare the expression for $\Delta(x_w)$ with respect to $\beta$ from Lemma~\ref{lem:prelim-exp} with the defining expansion $\Delta(x_w)=\sum_{v\leq w}x_{w/v}\otimes x_v$, we immediately find the formulas according to the three cases which express $x_{w/v}$ in terms of $v,w',\beta$. Here, we use that $x_v$ where $v\in W$ form a basis of $\tilde{\mathscr{N}}_W$. Lastly, suppose there exists $v\in W$ such that $v\leq w$, $s_\beta v>v$, $v\not\leq w'$. Then, it is clear that $s_\beta v\leq w$ by \cite[Proposition~5.9]{humphreys-coxeter}
applied to $v\leq w$. By \cite[Theorem~1.1(II)(ii)]{deodhar} applied to $s_\beta v\leq w$, it follows that $v\leq w'$ -- a contradiction.
\end{proof}

\begin{notation}
\label{not:varphi}

Let $w\in W$. Let $s_{\beta_1}\ldots s_{\beta_\ell}$ be a fixed reduced expression of $w$. Let $\bm{\beta}=(\beta_1,\ldots,\beta_\ell)$ be the sequence of simple roots corresponding to the fixed reduced expression of $w$. Let $J$ be any subset of $\{1,\ldots,\ell\}$. For all $1\leq j\leq\ell$, we define $\overrightarrow{\varphi}_j^{\bm{\beta}}(J)=s_{\beta_j}$ if $j\in J$ and $=\overrightarrow{D}_{\beta_j}$ if $j\notin J$. Similarly, for all $1\leq j\leq\ell$, we define $\overleftarrow{\varphi}_j^{\bm{\beta}}(J)=s_{\beta_j}$ if $j\in J$ and $=\overleftarrow{D}_{\beta_j}$ if $j\notin J$. Using this notation, we further define $\overrightarrow{\varphi}_J^{\bm{\beta}}=\overrightarrow{\varphi}_1^{\bm{\beta}}(J)\cdots\overrightarrow{\varphi}_\ell^{\bm{\beta}}(J)$ and $\overleftarrow{\varphi}_J^{\bm{\beta}}=\overleftarrow{\varphi}_1^{\bm{\beta}}(J)\cdots\overleftarrow{\varphi}_\ell^{\bm{\beta}}(J)$. All the defined operators are supposed to be $\mathbb{C}$-linear endomorphism acting from the left or the right on $\mathscr{B}_W$ as the arrows indicate in each case (cf.\ Convention~\ref{conv:homothety}).

\end{notation}

\todo[inline,color=green]{Convention about the order of products $\prod_J$ where $J\subseteq\mathbb{Z}$ -- so that I don't have to endlessly repeat the same sentences. The corresponding sentences which contain the word \enquote{total order} are to be removed after the convention is written/introduced. This convention might be even be placed in the introduction. Because already there, I will need it.}

\begin{thm}[{\cite[Chapter~2, Equation~(3)]{macdonald1991notes}}]
\label{thm:macdonald}

Let $v,w\in W$ be such that $v\leq w$. Let $s_{\beta_1}\cdots s_{\beta_\ell}$ be a fixed reduced expression of $w$. Let $\bm{\beta}=(\beta_1,\ldots,\beta_\ell)$ be the sequence of simple roots corresponding to the fixed reduced expression of $w$. Then we have
\[
\overleftarrow{D}_{w/v}=\left(\sum_J\overleftarrow{\varphi}_J^{\bm{\beta}}\right)v^{-1}
\]
where the sum ranges over all subsets $J$ of $\{1,\ldots,\ell\}$ such that the product $\prod_{j\in J}s_{\beta_j}$ 
is a reduced expression of $v$. In particular, the right side of the displayed equation is independent of $\bm{\beta}$, i.e.\ independent of the choice of the reduced expression of $w$. 

\end{thm}

\begin{proof}

Let $v,w,\bm{\beta}$ be as in the statement. We denote the right side of the displayed equation in the statement of Theorem~\ref{thm:macdonald} by $\overleftarrow{D}_{w/v}^{\bm{\beta}}$. We prove the theorem by induction on the length of $w$. If $w=1$, we clearly have $\overleftarrow{D}_{1/1}=\overleftarrow{D}_{1/1}^{\bm{\beta}}=1$ (cf.\ Example~\ref{ex:top-low-degree-xw}), and there is nothing to prove. Assume now that $\ell(w)>0$ and that the theorem is proved for all Coxeter group elements of length $<\ell(w)$. Let $\beta=\beta_1$ and $w'=s_\beta w$ for brevity. Furthermore, set $\bm{\beta}'=(\beta_2,\ldots,\beta_\ell)$, so that $\bm{\beta}'$ corresponds to the reduced expression $s_{\beta_2}\cdots s_{\beta_\ell}$ of $w'$. We now distinguish the three cases according to Corollary~\ref{cor:prelim-exp}.

Assume first that $v=v'$ where $v'\leq w'$, $s_\beta v'>v'$. By the induction hypothesis applied to $w'/v'$ and the first case of Corollary~\ref{cor:prelim-exp}, we find that
\[
\overleftarrow{D}_{w/v'}=\overleftarrow{D}_\beta\Bigg(\sum_{\substack{J'\subseteq\{2,\ldots,\ell\}\\ v'\overset{\text{red}}{=}\prod_{j\in J'}s_{\beta_j}}}\overleftarrow{\varphi}_{J'}^{\bm{\beta}'}\Bigg)v'^{-1}=\Bigg(\sum_{\substack{J\subseteq\{1,\ldots,\ell\}\\ v'\overset{\text{red}}{=}\prod_{j\in J}s_{\beta_j}}}\overleftarrow{\varphi}_J^{\bm{\beta}}\Bigg)v'^{-1}=\overleftarrow{D}_{w/v'}^{\bm{\beta}}
\]
where the symbol $\overset{\text{\tiny red}}{=}$ is supposed to mean that the right side is a reduced expression of the element on the left, and where the second equality follows since no reduced expression of $v'$ starts with $s_\beta$. Similarly to the first case, elementary analyses show that $\overleftarrow{D}_{w/v}=\overleftarrow{D}_{w/v}^{\bm{\beta}}$ where $v$ satisfies the conditions from the second or the third case of Corollary~\ref{cor:prelim-exp}. At some point in these analyses, one clearly has to use Remark~\ref{rem:conjugation} and the nilCoxeter relations. But this completes the proof because the three cases of Corollary~\ref{cor:prelim-exp} are exhaustive.
\end{proof}

\begin{cor}
\label{cor:independent-Dwvcirc}

Let $v,w\in W$ be such that $v\leq w$. Let $s_{\beta_1}\cdots s_{\beta_\ell}$ be a fixed reduced expression of $w$. Let $\bm{\beta}=(\beta_1,\ldots,\beta_\ell)$ be the sequence of simple roots corresponding to the fixed reduced expression of $w$. Then we have
\[
\overrightarrow{D}_{w/v}^\circ=v^{-1}\left(\sum_J\overrightarrow{\varphi}_J^{\bm{\beta}}\right)
\]
where the sum ranges over all subsets $J$ of $\{1,\ldots,\ell\}$ such that the product $\prod_{j\in J}s_{\beta_j}$ 
is a reduced expression of $v$. In particular, the right side of the displayed equation is independent of $\bm{\beta}$, i.e.\ independent of the choice of the reduced expression of $w$. 

\end{cor}

\begin{proof}

Let $v,w,\bm{\beta}$ be as in the statement. Let $\bm{\beta}^\circ=(\beta_\ell,\ldots,\beta_1)$. The sequence of simple roots $\bm{\beta}^\circ$ corresponds to the reduced expression $s_{\beta_\ell}\cdots s_{\beta_1}$ of $w^{-1}$. Let $J$ be a subset of $\{1,\ldots,\ell\}$ as it occurs in the summation in the statement of the corollary. Let $J^\circ=\{\ell-j+1\mid j\in J\}$. The set $J^\circ$ is a subset of $\{1,\ldots,\ell\}$ as it occurs in the summation in the statement of Theorem~\ref{thm:macdonald} applied to $v^{-1}\leq w^{-1}$ and $\bm{\beta}^\circ$. Moreover, we have $\overrightarrow{\varphi}_J^{\bm{\beta}}(\phi)=(\phi)\overleftarrow{\varphi}_{J^\circ}^{\bm{\beta}^\circ}$ for all $\phi\in\mathscr{B}_W$. The corollary follows from this observation and Theorem~\ref{thm:macdonald} applied to $v^{-1}\leq w^{-1}$ and $\bm{\beta}^\circ$ since we have by definition $\overrightarrow{D}_{w/v}^\circ(\phi)=(\phi)\overleftarrow{D}_{w^{-1}/v^{-1}}$ for all $\phi\in\mathscr{B}_W$ (cf.\ Proposition~\ref{prop:rhoands}\eqref{item:involution}).
\end{proof}

\begin{thm}[{Generalized braided Leibniz rule \cite[Chapter~2, Theorem~2.17]{macdonald1991notes}}]
\label{thm:leibniz}

Let $w\in W$ and $\phi,\psi\in\mathscr{B}_W$. Then we have
\[
(\phi\psi)\overleftarrow{D}_w=\sum_{v\leq w}((\phi)\overleftarrow{D}_v)((\psi)\overleftarrow{D}_{w/v}v)\,.
\]

\end{thm}

\begin{proof}

Let $w,\phi,\psi$ be as in the statement. We prove the theorem by induction on the length of $w$. If $w=1$, the statement is trivial. If $\ell(w)=1$, the generalized braided Leibniz rule is a special case of the ordinary braided Leibniz rule (cf.\ Remark~\ref{rem:braided-leibniz} and Example~\ref{ex:top-low-degree-xw}). Assume from now on that $\ell(w)>1$ and that the generalized braided Leibniz rule is known for all Coxeter group elements of length $<\ell(w)$. Let $\beta\in\Delta$ be such that $s_\beta w<w$. Let $w'=s_\beta w$ for brevity. If we apply the braided Leibniz rule for $\overleftarrow{D}_\beta$ (cf.\ remark loc.\ cit.) and the induction hypothesis to $w'$, we find in view of the decomposition $\overleftarrow{D}_w=\overleftarrow{D}_\beta\overleftarrow{D}_{w'}$ that
\[
(\phi\psi)\overleftarrow{D}_w=\sum_{v'\leq w'}((\phi)\overleftarrow{D}_{v'})((\psi)\overleftarrow{D}_\beta\overleftarrow{D}_{w'/v'}v')+\sum_{v'\leq w'}((\phi)\overleftarrow{D}_\beta\overleftarrow{D}_{v'})((\psi)s_\beta\overleftarrow{D}_{w'/v'}v')\,.
\]
By the nilCoxeter relations, all summands $v'\leq w'$ such that $s_\beta v'<v'$ in the second sum of the previous formula are zero. By \cite[Theorem~1.1(II)(ii)]{deodhar}, we therefore can rewrite the second sum of the previous formula as
\[
\sum_{\substack{v\leq w\\ s_\beta v<v}}((\phi)\overleftarrow{D}_v)((\psi)s_\beta\overleftarrow{D}_{w'/s_\beta v}s_\beta v)\,.
\]
If we now split the two summations according to
\[
\textstyle{\sum_{v'\leq w'}=\sum_{\substack{v'\leq w'\\ s_\beta v'>v'}}+\sum_{\substack{v'\leq w'\\ s_\beta v'<v'}}\quad\text{and}\quad\sum_{\substack{v\leq w\\ s_\beta v<v}}=\sum_{\substack{v'\leq w'\\ s_\beta v'<v'}}+\sum_{\substack{v\leq w\\\hphantom{'}v\not\leq w'\\ s_\beta v<v}}\,,}
\]
we find that $(\phi\psi)\overleftarrow{D}_w$ equals
\begin{align*}
\sum_{\substack{v'\leq w'\\ s_\beta v'>v'}}((\phi)\overleftarrow{D}_{v'})((\psi)\overleftarrow{D}_\beta\overleftarrow{D}_{w'/v'}v')&+\sum_{\substack{v'\leq w'\\ s_\beta v'<v'}}((\phi)\overleftarrow{D}_{v'})((\psi)(\overleftarrow{D}_\beta\overleftarrow{D}_{w'/v'}+s_\beta\overleftarrow{D}_{w'/s_\beta v'}s_\beta)v')\\
&+\sum_{\substack{v\leq w\\\hphantom{'}v\not\leq w'\\\hphantom{'}s_\beta v<v\hphantom{'}}}((\phi)\overleftarrow{D}_v)((\psi)s_\beta\overleftarrow{D}_{w'/s_\beta v}s_\beta v)\,.
\end{align*}
The formulas for $x_{w/v}$ in the three cases of Corollary~\ref{cor:prelim-exp} correspond precisely to the three sums in the previous displayed equation. If we plug them into the sums, we can express the second factor in each sum as $(\psi)\overleftarrow{D}_{w/v'}v'$ or as $(\psi)\overleftarrow{D}_{w/v}v$. (Here, we use Remark~\ref{rem:conjugation}.) But this completes the proof because the summation over the three sums in the last displayed formula is exhaustive (cf.\ Corollary~\ref{cor:prelim-exp}).
\end{proof}

\begin{cor}
\label{cor:leibniz}

Let $w\in W$ and $\phi,\psi\in\mathscr{B}_W$. Then we have
\[
\overrightarrow{D}_w^\circ(\phi\psi)=\sum_{v\leq w}(\overrightarrow{D}_v^\circ(\phi))(v\overrightarrow{D}_{w/v}^\circ(\psi))\,.
\]

\end{cor}

\begin{proof}

Let $w,\phi,\psi$ be as in the statement. If we apply Theorem~\ref{thm:leibniz} to $w^{-1}$, we immediately find that
\[
\overrightarrow{D}_w^\circ(\phi\psi)=\sum_{v\leq w^{-1}}((\phi)\overleftarrow{D}_v)((\psi)\overleftarrow{D}_{w^{-1}/v}v)\,.
\]
If we now replace in the summation above $v$ by $v^{-1}$, we find the desired formula in view of the definition of the opposite braided partial derivatives and the braided skew partial derivatives, Proposition~\ref{prop:rhoands}\eqref{item:involution}
and the fact that $v\leq w$ if and only if $v^{-1}\leq w^{-1}$.
\end{proof}

\begin{prop}
\label{prop:uniqueness-skew}

Suppose for all pairs $v,w\in W$, there are $\mathbb{C}$-linear endomorphisms $\overleftarrow{D}_{w/v}'$ of $\mathscr{B}_W$ given such that 
\[
(\phi\psi)\overleftarrow{D}_w=\sum_{v\in W}((\phi)\overleftarrow{D}_v)((\psi)\overleftarrow{D}_{w/v}'v)
\]
for all $\phi,\psi\in\mathscr{B}_W$. Then we have $\overleftarrow{D}_{w/v}'=\overleftarrow{D}_{w/v}$ for all $v,w\in W$.

\end{prop}

\begin{proof}

In order to keep the logical structure of this article clear, we anticipate in this proof one of the results from Section~\ref{sec:positivity}, namely Proposition~\ref{prop:sat-chain}. Let $w\in W$ be fixed but arbitrary. We show by induction on the length of $v\in W$ that the primed operators $\overleftarrow{D}_{w/v}'$ are determined as claimed. Suppose first that $v=1$. If we apply the generalized braided Leibniz rule, i.e.\ the assumption on the primed operators, on $\phi=1$, we find for simple $\mathbb{Z}_{\geq 0}$-degree reasons (cf. Remark~\ref{rem:degree-partial}) that $\overleftarrow{D}_{w/1}'=\overleftarrow{D}_w$. In view of Example~\ref{ex:top-low-degree-xw}, this shows $\overleftarrow{D}_{w/1}'=\overleftarrow{D}_{w/1}$ and completes the induction base. 

For the induction step, suppose that $\ell(v)>0$ and that the primed operators are determined as claimed for fixed $w$ and all Coxeter group elements of length $<\ell(v)$. By the generalized braided Leibniz rule for the primed operators, the induction hypothesis and Theorem~\ref{thm:leibniz}, we immediately find that
\[
\sum_{u\colon\ell(v)\leq\ell(u)}((\phi)\overleftarrow{D}_u)((\psi)\overleftarrow{D}_{w/u}'u)=\sum_{u\colon\ell(v)\leq\ell(u)}((\phi)\overleftarrow{D}_u)((\psi)\overleftarrow{D}_{w/u}u)
\]
for all $\phi,\psi\in\mathscr{B}_W$. If we apply this equation to $\phi=x_{v^{-1}}$, we find for $\mathbb{Z}_{\geq 0}$-degree reasons (cf.\ remark loc.\ cit.) that only those summands on each side are nonzero which are parametrized by $u\in W$ such that $\ell(u)=\ell(v)$. But for $u\in W$ such that $\ell(u)=\ell(v)$ we know that
$
(\phi)\overleftarrow{D}_u=\epsilon((\phi)\overleftarrow{D}_u)=\left<\phi,x_u\right>
$
by remark loc.\ cit.\ and Equation~\eqref{eq:pairing-derivative}. Proposition~\ref{prop:sat-chain} therefore shows that $(\phi)\overleftarrow{D}_u=1$ if $u=v$ and $=0$ otherwise. Consequently, the previous displayed equation reduces to $(\psi)\overleftarrow{D}_{w/v}'v=(\psi)\overleftarrow{D}_{w/v}v$ for all $\psi\in\mathscr{B}_W$ or equivalently to $\overleftarrow{D}_{w/v}'=\overleftarrow{D}_{w/v}$. This completes the induction step and hence the proof of the proposition.
\end{proof}

\begin{cor}

Suppose for all pairs $v,w\in W$, there are $\mathbb{C}$-linear endomorphisms $\overrightarrow{D}_{w/v}^{\circ\,\prime}$ of $\mathscr{B}_W$ given such that 
\[
\overrightarrow{D}_w^\circ(\phi\psi)=\sum_{v\in W}(\overrightarrow{D}_v^\circ(\phi))(v\overrightarrow{D}_{w/v}^{\circ\,\prime}(\psi))
\]
for all $\phi,\psi\in\mathscr{B}_W$. Then we have $\overrightarrow{D}_{w/v}^{\circ\,\prime}=\overrightarrow{D}_{w/v}^\circ$ for all $v,w\in W$.

\end{cor}

\begin{proof}

The generalized braided Leibniz rule for the primed operators, i.e.\ the assumption on $\overrightarrow{D}_{w/v}^{\circ\,\prime}$, applied to $w^{-1}$ for some $w\in W$ and with summation over $v^{-1}$ instead of $v$ can be written in view of the definition of the opposite braided partial derivatives as
\[
(\phi\psi)\overleftarrow{D}_w=\sum_{v\in W}((\phi)\overleftarrow{D}_v)((\overrightarrow{D}_{w^{-1}/v^{-1}}^{\circ\,\prime}(\psi))v)
\]
for all $\phi,\psi\in\mathscr{B}_W$. Proposition~\ref{prop:uniqueness-skew} now implies that $\overrightarrow{D}_{w^{-1}/v^{-1}}^{\circ\,\prime}(\psi)=(\psi)\overleftarrow{D}_{w/v}$ for all $v,w\in W$ and all $\psi\in\mathscr{B}_W$. The definition of the braided skew partial derivatives and Proposition~\ref{prop:rhoands}\eqref{item:involution} imply that $\overrightarrow{D}_{w^{-1}/v^{-1}}^{\circ\,\prime}=\overrightarrow{D}_{w^{-1}/v^{-1}}^\circ$ and thus $\overrightarrow{D}_{w/v}^{\circ\,\prime}=\overrightarrow{D}_{w/v}^\circ$ for all $v,w\in W$.
\end{proof}

\section{Positivity \except{toc}{\texorpdfstring{(\cite[\textsc{Section~3}]{liu})}{([\ref{bib-liu},~Section~3])}}\for{toc}{(\cite[Section~3]{liu})}}
\label{sec:positivity}

This section is devoted to the generalization of positivity after Liu \cite{liu}. For all $v,w\in W$ such that $v\leq w$, we prove that $x_{w/v}$ is a positive element of $\mathscr{B}_W$ (cf.\ Theorem~\ref{thm:positivity}) in the sense of the definition directly below. We give two formulas to produce a positive expression for $x_{w/v}$, a recursive formula (Corollary~\ref{cor:recursive}) and one in terms of reduced expressions (Corollary~\ref{cor:xwv-pos}).
We closely follow the references. The proofs are mostly routine.

\begin{defn}

We define a subset $\mathscr{B}_W^{\geq 0}$ of $\mathscr{B}_W$ as follows:
\[
\mathscr{B}_W^{\geq 0}=\operatorname{span}_{\mathbb{Z}_{\geq 0}}\bigcup_{m=0}^\infty\{x_{\alpha_1}\cdots x_{\alpha_m}\mid\alpha_1,\ldots,\alpha_m\in R^+\}\,.
\]
It is clear that $\mathscr{B}_W^{\geq 0}$ is a monoid with respect to addition and multiplication inherited from $\mathscr{B}_W$. 
We call the monoid $\mathscr{B}_W^{\geq 0}$ the positive cone of $\mathscr{B}_W$. We call an element of $\mathscr{B}_W$ positive if and only if it lies in $\mathscr{B}_W^{\geq 0}$. We further define $\mathscr{B}_W^{>0}=\mathscr{B}_W^{\geq 0}\setminus\{0\}$. We call the set $\mathscr{B}_W^{>0}$ the strictly positive cone of $\mathscr{B}_W$. We call an element of $\mathscr{B}_W$ strictly positive if and only if it lies in $\mathscr{B}_W^{>0}$, i.e.\ if and only if it is positive and nonzero. Moreover, it is useful to define a subset $\mathscr{B}_W^{\geq 0}\otimes\mathscr{B}_W^{\geq 0}$ of $\mathscr{B}_W\otimes\mathscr{B}_W$ as follows:
\[
\mathscr{B}_W^{\geq 0}\otimes\mathscr{B}_W^{\geq 0}=\operatorname{span}_{\mathbb{Z}_{\geq 0}}\left\{x\otimes y\;\middle|\; x,y\in\mathscr{B}_W^{\geq 0}\right\}\,.
\]
It is clear that $\mathscr{B}_W^{\geq 0}\otimes\mathscr{B}_W^{\geq 0}$ is a monoid with respect to addition inherited from $\mathscr{B}_W\otimes\mathscr{B}_W$. We further set $\mathscr{B}_W^{>0}\otimes\mathscr{B}_W^{>0}=\left(\mathscr{B}_W^{\geq 0}\otimes\mathscr{B}_W^{\geq 0}\right)\setminus\{0\}$. From their definition, it is clear that all the introduced sets are invariant under dilation with elements of $\mathbb{Z}_{\geq 0}$ or $\mathbb{Z}_{>0}$ as their superscripts indicate in each case.


\end{defn}

\begin{rem}[{\cite[Proposition~3.3]{liu}}]
\label{rem:positive-image}


By Proposition~\ref{prop:rhoands}\eqref{item:identity},\eqref{item:rhoantialgebra},\eqref{item:involution} and definition, it is immediately clear that $\rho(\mathscr{B}_W^{\geq 0})=\mathscr{B}_W^{\geq 0}$ and that $\rho(\mathscr{B}_W^{>0})=\mathscr{B}_W^{>0}$. On the other hand, let $w\in W$. Let $s_{\beta_1}\cdots s_{\beta_\ell}$ be a reduced expression of $w$. Let $\alpha_i=s_{\beta_1}\cdots s_{\beta_{i-1}}(\beta_i)$ and let $y_i=x_{\alpha_i}$ for all $1\leq i\leq\ell$. By \cite[Proposition~5.7]{humphreys-coxeter}, we know well that $\alpha_i$ is a positive root for all $1\leq i\leq\ell$. Thus it follows by Remark~\ref{rem:comp-sbar} that $\bar{\mathscr{S}}(x_w)=y_1\cdots y_{\ell}\in\mathscr{B}_W^{\geq 0}$. Since $x_w\neq 0$, it follows from Proposition~\ref{prop:rhoands}\eqref{item:involution} even that $\bar{\mathscr{S}}(x_w)\in\mathscr{B}_W^{>0}$. From now on, we will use these facts literally or an obvious modification of them with reference to this remark. 

\end{rem}

\begin{prop}[{\cite[Proposition~3.1]{liu}}]
\label{prop:simple-left-derivative}

Let $v,w\in W$ and let $\alpha\in R^+$. Then we have $\overrightarrow{D}_\alpha(x_w)=x_{s_\alpha w}$ if $s_\alpha w\lessdot w$ and $=0$ otherwise. Moreover, we have $\overrightarrow{D}_v(x_w)=x_{vw}$ if $\ell(w)=\ell(v^{-1})+\ell(vw)$ and $=0$ otherwise.

\end{prop}

\begin{proof}

The second claim in the statement of the proposition follows directly from the first. We prove the first. Let $w$ and $\alpha$ be as in the statement. Let $s_{\beta_1}\cdots s_{\beta_\ell}$ be a reduced expression of $w$. By definition of the comultiplication as an algebra homomorphism $\mathscr{B}_W\to\mathscr{B}_W\underline{\otimes}\mathscr{B}_W$ in the Yetter-Drinfeld category over $W$, the $(\mathbb{Z}_{\geq 0}\times\mathbb{Z}_{\geq 0})$-degree $(1,\ell-1)$-component of $\Delta(x_w)$ is given by
\[
\sum_{i=1}^\ell x_{s_{\beta_1}\cdots s_{\beta_{i-1}}(\beta_i)}\otimes x_{\beta_1}\cdots\hat{x}_{\beta_i}\cdots x_{\beta_\ell}=
\sum_{\gamma\in R^+\colon s_\gamma w\lessdot w}x_\gamma\otimes x_{s_\gamma w}
\]
where the equality follows by the nilCoxeter relations since the $i$th summand vanishes for all $1\leq i\leq\ell$ such that $s_{\beta_1}\cdots\hat{s}_{\beta_i}\cdots s_{\beta_\ell}$ has length $<\ell-1$. Furthermore, we used the characterization of the Bruhat order in terms of subexpressions (\cite[Theorem~5.10]{humphreys-coxeter}). The first claim is now immediate from the definition of $\overrightarrow{D}_\alpha$ and the definition of the restriction of the Hopf duality pairing between $\mathscr{B}_W$ and itself to $V_W\otimes V_W$. Note that we used Remark~\ref{rem:vanishing}. 
%
\end{proof}

\todo[inline,color=green]{The same remark applies to this corollary as for the corollary after Theorem~\ref{thm:first-step}.

{\bf Corollary.} Let $v,w\in W$. Then the element $(\bar{\mathscr{S}}(x_w))\protect\overleftarrow{D}_v$ of $\mathscr{B}_W$ is positive.

[In the beginning, it seemed that I wanted to follow this from Proposition~\ref{prop:simple-left-derivative} and Remark~\ref{rem:positive-image},~\ref{rem:liu2.10(d)} (maybe also Proposition~\ref{prop:liu2.10(d)}).]}

\begin{prop}[{\cite[Proposition~3.2]{liu}}]
\label{prop:sat-chain}

Let $w\in W$ be an element of length $\ell$. Let $\alpha_1,\ldots,\alpha_\ell\in R^+$ be such that $s_{\alpha_\ell}\cdots s_{\alpha_1}=w$. Then we have
$
\left<x_{\alpha_1}\cdots x_{\alpha_\ell},x_w\right>=1
$
if
$$
1\lessdot s_{\alpha_1}\lessdot s_{\alpha_2}s_{\alpha_1}\lessdot\cdots\lessdot s_{\alpha_{\ell-1}}\cdots s_{\alpha_1}\lessdot s_{\alpha_\ell}\cdots s_{\alpha_1}=w
$$
is a saturated chain in the Bruhat order and $=0$ otherwise. In particular, for $v\in W$, we have $\left<x_v,x_w\right>=1$ if $w=v^{-1}$ and $=0$ otherwise.

\end{prop}

\begin{proof}

The second statement in the proposition follows immediately from the first. The first follows from Proposition~\ref{prop:simple-left-derivative} and Equation~\eqref{eq:pairing-derivative}.
\end{proof}

\begin{prop}[{\cite[Proposition~3.4]{liu}}]
\label{prop:reflection-ordering}

The equation $\bar{\mathscr{S}}(x_{w_o})=x_{w_o}$ holds in $\mathscr{B}_W$.

\end{prop}

\begin{proof}

By the nondegeneracy of the Hopf duality pairing between $\mathscr{B}_W$ and itself, it suffices to show that $\left<\phi,\bar{\mathscr{S}}(x_{w_o})\right>=\left<\phi,x_{w_o}\right>$ for all $\phi\in\mathscr{B}_W$. Let $\phi\in\mathscr{B}_W$ be fixed but arbitrary. Let $\ell=\ell(w_o)$. To show the desired equality, we clearly can assume that $\phi$ is homogeneous of $\mathbb{Z}_{\geq 0}$-degree $\ell$ and of $W$-degree $w_o$ since otherwise both sides of the desired equality are equal to zero. (To see this, note that $w_o$ is an involution and consider Proposition~\ref{prop:rhoands}\eqref{item:morphop}. We also take Remark~\ref{rem:vanishing} into account and that the Hopf duality pairing between $\mathscr{B}_W$ and itself is a morphism in the Yetter-Drinfeld category over $W$.) By linearity, we even may assume that $\phi=x_{\alpha_1}\cdots x_{\alpha_\ell}$ where $\alpha_1,\ldots,\alpha_\ell\in R^+$ are such that $s_{\alpha_\ell}\cdots s_{\alpha_1}=w_o$.

Let us now prove as an intermediate step that we have a saturated chain 
$$
1\lessdot s_{\alpha_1}\lessdot s_{\alpha_2}s_{\alpha_1}\lessdot\cdots\lessdot s_{\alpha_{\ell-1}}\cdots s_{\alpha_1}\lessdot s_{\alpha_\ell}\cdots s_{\alpha_1}=w_o
$$
if and only if we have a saturated chain
$$
1\lessdot s_{\alpha_\ell}\lessdot s_{\alpha_{\ell-1}}s_{\alpha_\ell}\lessdot\cdots\lessdot s_{\alpha_2}\cdots s_{\alpha_\ell}\lessdot s_{\alpha_1}\cdots s_{\alpha_\ell}=w_o\,.
$$
To see this, just note that the operation of multiplying each member of a saturated chain from $1$ to $w_o$ from the right by $w_o$ takes a saturated chain from $1$ to $w_o$ to a new such chain (cf.\ \cite[Exercise~5.9, Example~5.9.3]{humphreys-coxeter}). Furthermore, this operation on saturated chains is involutive. In this way, we get from the first displayed chain to the second and vice versa using the described operation and the decompositions $w_o=s_{\alpha_1}\cdots s_{\alpha_\ell}$ and $w_o=s_{\alpha_\ell}\cdots s_{\alpha_1}$ respectively. This proves the equivalence stated as an intermediate step.

In view of Proposition~\ref{prop:sat-chain}, the intermediate step tells us that 
\[
\left<x_{\alpha_1}\cdots x_{\alpha_\ell},x_{w_o}\right>=\left<x_{\alpha_\ell}\cdots x_{\alpha_1},x_{w_o}\right>
\] 
which can be rewritten as $\left<\phi,x_{w_o}\right>=\left<\rho(\phi),x_{w_o}\right>$. In view of Proposition~\ref{prop:adjoint}, the last equality is equivalent to $\left<\phi,\bar{\mathscr{S}}(x_{w_o})\right>=\left<\phi,x_{w_o}\right>$. This completes the proof.
\end{proof}

\begin{rem}
\label{rem:quad-rel}

In type $\mathsf{A}$, relations of $\mathbb{Z}_{\geq 0}$-degree $\leq 2$ suffice to prove Proposition~\ref{prop:reflection-ordering} (cf. Example~\ref{ex:relations} and the proof of \cite[Proposition~3.4]{liu}). In general, it is unclear whether relations of $\mathbb{Z}_{\geq 0}$-degree $\leq 2$ will suffice to prove Proposition~\ref{prop:reflection-ordering}. Fortunately, we do not need such a sharper result. For our purposes, it is sufficient to use all the relations in the Woronowicz ideal of $V_W$, i.e.\ to use the nondegeneracy of the Hopf duality pairing between $\mathscr{B}_W$ and itself.

\end{rem}

\begin{rem}

Let $\ell=|R^+|$. According to \cite[Definition~2.1,~2.2, Proposition~2.13]{dyer-reflection-ordering}, a total order $\alpha_1\prec\cdots\prec\alpha_\ell$ on $R^+$ is called a reflection ordering if and only if there exists a reduced expression $s_{\beta_1}\cdots s_{\beta_\ell}$ of $w_o$ such that $\alpha_i=s_{\beta_1}\cdots s_{\beta_{i-1}}(\beta_i)$ for all $1\leq i\leq\ell$. Using this terminology and Remark~\ref{rem:comp-sbar}, Proposition~\ref{prop:reflection-ordering} can be stated as follows: We have $x_{w_o}=x_{\alpha_1}\cdots x_{\alpha_\ell}$ in $\mathscr{B}_W$ for any reflection ordering $\alpha_1\prec\cdots\prec\alpha_\ell$. If we apply $\rho$ to the previous equation, we can equivalently say that $x_{w_o}=x_{\alpha_\ell}\cdots x_{\alpha_1}$ in $\mathscr{B}_W$ for any reflection ordering $\alpha_1\prec\cdots\prec\alpha_\ell$.

\end{rem}

\begin{thm}
\label{thm:first-step}

Let $w\in W$. Let $s_{\beta_1}\cdots s_{\beta_\ell}$ be a reduced expression of $w$. Let $\alpha_i=s_{\beta_1}\cdots$ $s_{\beta_{i-1}}(\beta_i)$ and let $y_i=x_{\alpha_i}$ for all $1\leq i\leq\ell$. Then we have
\[
\Delta(\bar{\mathscr{S}}(x_w))=\sum_{J\subseteq\{1,\ldots,\ell\}}\bar{\mathscr{S}}\bigg(\prod_{j\in\bar{J}}x_{\beta_j}\bigg)\otimes\prod_{j\in J}y_j
\]
where $\bar{J}$ denotes the complement of a subset $J$ in $\{1,\ldots,\ell\}$. In particular, it follows that $\Delta(\bar{\mathscr{S}}(x_w))\in\mathscr{B}_W^{>0}\otimes\mathscr{B}_W^{>0}$.

\end{thm}

\begin{proof}

Let the notation be as in the statement. From Remark~\ref{rem:positive-image} and from the nilCoxeter relations, it is clear that $\Delta(\bar{\mathscr{S}}(x_w))\in\mathscr{B}_W^{\geq 0}\otimes\mathscr{B}_W^{\geq 0}$ once the formula in the statement of the theorem is established. In any Hopf algebra $H$, it is clear from the counit axiom that $\Delta(h)=0$ if and only if $h=0$ where $h\in H$. Hence, we have $\Delta(\bar{\mathscr{S}}(x_w))\neq 0$ if and only if $\bar{\mathscr{S}}(x_w)\neq 0$ if and only if $x_w\neq 0$ where the last equivalence follows from Proposition~\ref{prop:rhoands}\eqref{item:involution}. Since $x_w$ is a basis vector of $\tilde{\mathscr{N}}_W$, we know that $x_w\neq 0$. This shows that $\Delta(\bar{\mathscr{S}}(x_w))\in\mathscr{B}_W^{>0}\otimes\mathscr{B}_W^{>0}$ once the formula in the statement of the theorem is established. 

Let us now prove this very formula by induction on the length of $w$. If $w=1$, then both sides of the claimed formula equal $1\otimes 1$ (cf. Proposition~\ref{prop:rhoands}\eqref{item:identity}) and there is nothing to prove. For the induction step, suppose that $\ell>0$ and that the formula is proved for all Coxeter group elements of length $<\ell$. Let $\alpha_i'=s_{\beta_2}\cdots s_{\beta_{i-1}}(\beta_i)$ and $y_i'=x_{\alpha_i'}$ for all $2\leq i\leq\ell$. Let $\beta=\beta_1$ and $w'=s_\beta w$ for brevity. Note that $s_\beta y_i'=y_i$ for all $2\leq i\leq\ell$. If we now apply the induction hypothesis to $w'$ and to the reduced expression $s_{\beta_2}\cdots s_{\beta_\ell}$ of $w'$, we find by definition of the comultiplication and by Proposition~\ref{prop:rhoands}\eqref{item:identity},\eqref{item:sbaralebgraprop} that 
\[
\Delta(\bar{\mathscr{S}}(x_w))=\Delta(x_\beta(s_\beta\bar{\mathscr{S}}(x_{w'})))=(x_\beta\otimes 1+1\otimes x_\beta)\cdot s_\beta\sum_{J'\subseteq\{2,\ldots,\ell\}}\bar{\mathscr{S}}\bigg(\prod_{j\in\bar{J}'}x_{\beta_j}\bigg)\otimes\prod_{j\in J'}y_j'
\]
where $\cdot$ denotes the twisted multiplication in $\mathscr{B}_W\underline{\otimes}\mathscr{B}_W$ and $\bar{J}'$ denotes the complement of a subset $J'$ in $\{2,\ldots,\ell\}$. If we multiply out, the first summand becomes in view of proposition loc.\ cit.\
\[
\sum_{J\in\{1,\ldots,\ell\}\colon 1\notin J}\bar{\mathscr{S}}\bigg(\prod_{j\in\bar{J}}x_{\beta_j}\bigg)\otimes\prod_{j\in J}y_j
\]
and the second summand becomes
\[
\sum_{J\in\{1,\ldots,\ell\}\colon 1\in J}\bar{\mathscr{S}}\bigg(\prod_{j\in\bar{J}}x_{\beta_j}\bigg)\otimes\prod_{j\in J}y_j
\]
where in both of the two previous displayed equations $\bar{J}$ denotes the complement of a subset $J$ in $\{1,\ldots,\ell\}$ as in the statement of the theorem. If we combine the two summands, we indeed find the desired expression for $\Delta(\bar{\mathscr{S}}(x_w))$ as it was claimed. This completes the induction step and hence the proof of the theorem.
\end{proof}

\todo[inline,color=green]{The following corollary (cf. the commented corollary below and also the commented corollary after Proposition~\ref{prop:simple-left-derivative}) becomes superfluous when I prove the stronger results in the subsection on \enquote{Strict positivity}. Yet, these results (statements and proofs) still have to be included. I can mention the particular case $x_{w/1}=x_w$ in the general situation for $x_{w/v}$. But it does not make sense to refer to the original references as they are too specific. And then it is questionable if I mention the particular case anyway.

{\bf Corollary.} Let $v,w\in W$. Then the element $\protect\overrightarrow{D}_v(\bar{\mathscr{S}}(x_w))$ of $\mathscr{B}_W$ is positive.

[The corollaries are now trivial special cases which follow from Theorem~\ref{thm:positivity} and Corollary~\ref{cor:formal-kirillov}. I don't even mention them.]}

\begin{cor}
\label{cor:sxwv}

Let $v,w\in W$ be such that $v\leq w$. Let $s_{\beta_1}\cdots s_{\beta_\ell}$ be a reduced expression of $w$. Let $\alpha_i=s_{\beta_1}\cdots s_{\beta_{i-1}}(\beta_i)$ and let $y_i=x_{\alpha_i}$ for all $1\leq i\leq\ell$. Then we have
\[
\bar{\mathscr{S}}(x_{w/v})=\sum_J\prod_{j\in J}y_j
\]
where the sum ranges over all subsets $J$ of $\{1,\ldots,\ell\}$ such that the product $\prod_{j\in \bar{J}}s_{\beta_j}$ is a reduced expression of $v$. Here, as in Theorem~\ref{thm:first-step}, the set $\bar{J}$ denotes the complement of a subset $J$ in $\{1,\ldots,\ell\}$. In particular, it follows that $\bar{\mathscr{S}}(x_{w/v})$ is a positive element of $\mathscr{B}_W$.

\end{cor}

\begin{proof}

Let the notation be as in the statement. Once the claimed formula in the statement of the corollary is proved, it is clear that $\bar{\mathscr{S}}(x_{w/v})$ is a positive element of $\mathscr{B}_W$. Let us now prove this formula. By Proposition~\ref{prop:rhoands}\eqref{item:sbaranticolagebra}, it follows that we have
\[
\Delta(\bar{\mathscr{S}}(x_w))=\sum_{u\leq w}\bar{\mathscr{S}}(x_u)\otimes\bar{\mathscr{S}}(x_{w/u})\,.
\]
On the other hand, in view of the characterization of the Bruhat order in terms of reduced subexpressions (cf. \cite[Corollary~5.8(b), Theorem~5.10]{humphreys-coxeter}, note that $\prod_{j\in\bar{J}}x_{\beta_j}=0$ where $\bar{J}\subseteq\{1,\ldots,\ell\}$ unless $u=\prod_{j\in\bar{J}}s_{\beta_j}$ is a reduced expression because of the nilCoxeter relations) we can write the formula in the statement of Theorem~\ref{thm:first-step} as 
\[
\Delta(\bar{\mathscr{S}}(x_w))=\sum_{u\leq w}\bar{\mathscr{S}}(x_u)\otimes\sum_J\prod_{j\in J}y_j
\]
where the sum in each summand of the above expression depends on $u\leq w$ and ranges over all subsets $J$ of $\{1,\ldots,\ell\}$ such that the product $\prod_{j\in\bar{J}}s_{\beta_j}$ is a reduced expression of $u$. As in the statement of the corollary, the set $\bar{J}$ denotes the complement of a subset $J$ in $\{1,\ldots,\ell\}$. Now, we know that the vectors $x_u$ where $u\leq w$ form a linearly independent system and so do the vectors $\bar{\mathscr{S}}(x_u)$ where $u\leq w$ by Proposition~\ref{prop:rhoands}\eqref{item:involution}. By comparison of the summands corresponding to $u=v$ in the two previous displayed equations, the formula in the statement of the corollary follows from the observation in the last sentence. 
\end{proof}

\begin{rem}

We note that Corollary~\ref{cor:sxwv} makes even sense in the case that $v,w\in W$ are such that $v\not\leq w$. Both sides of the formula in Corollary~\ref{cor:sxwv} are then equal to zero (cf.\ \cite[Theorem~5.10]{humphreys-coxeter}). In the proof of Proposition~\ref{prop:recursive-prel}, we make use of the said formula in this case too. 

\end{rem}

\begin{thm}[{Positivity \cite[Theorem~3.5]{liu}}]
\label{thm:positivity}

Let $v,w\in W$ be such that $v\leq w$. Then we have $x_{w/v}=\bar{\mathscr{S}}(x_{vw_o/ww_o})$. In particular, it follows that $x_{w/v}$ is a positive element of $\mathscr{B}_W$.

\end{thm}

\begin{proof}

Let $v,w$ be as in the statement. Once the claimed formula for $x_{w/v}$ is established, the positivity of $x_{w/v}$ is evident from Corollary~\ref{cor:sxwv}. (Recall that $v\leq w$ if and only if $ww_o\leq vw_o$ by \cite[Exercise~5.9, Example~5.9.3]{humphreys-coxeter}.) We now prove the claimed formula for $x_{w/v}$. Since $\ell(w_o)=\ell(w_ow^{-1})+\ell(w)$, we know by Proposition~\ref{prop:simple-left-derivative} that $\overrightarrow{D}_{ww_o}(x_{w_o})=x_w$. By the definition of the braided right partial derivatives and Proposition~\ref{prop:sat-chain}, we also have $(x_w)\overleftarrow{D}_{v^{-1}}=x_{w/v}$. (For a more general result separated from the details of this proof, compare the previous equation with Proposition~\ref{prop:reflection-skew}.) In view of the two previous equations, Proposition~\ref{prop:liu2.10(d)}, \ref{prop:reflection-ordering} and Remark~\ref{rem:left-right-simult}, \ref{rem:liu2.10(d)}, we can write
\[
x_{w/v}=\overrightarrow{D}_{ww_o}(x_{w_o})\overleftarrow{D}_{v^{-1}}=\overrightarrow{D}_{ww_o}(\bar{\mathscr{S}}(x_{w_o}))\overleftarrow{D}_{v^{-1}}=\big(\overrightarrow{D}_{ww_o}\circ\bar{\mathscr{S}}\circ\overrightarrow{D}_v\big)(x_{w_o})=\overrightarrow{D}_{ww_o}(\bar{\mathscr{S}}(x_{vw_o}))
\]
where we used in the last equality again Proposition~\ref{prop:simple-left-derivative} which is possible because $\ell(w_o)=\ell(v^{-1})+\ell(vw_o)$. By Proposition~\ref{prop:rhoands}\eqref{item:sbaranticolagebra}, we know that
\[
\Delta(\bar{\mathscr{S}}(x_{vw_o}))=\sum_{u\leq vw_o}\bar{\mathscr{S}}(x_u)\otimes\bar{\mathscr{S}}(x_{vw_o/u})\,.
\]
If we plug this expression into the definition of the braided left partial derivative, we find in view of Proposition~\ref{prop:adjoint}, \ref{prop:sat-chain} that
\[
\overrightarrow{D}_{ww_o}(\bar{\mathscr{S}}(x_{vw_o}))=\sum_{u\leq vw_o}\left<x_{ww_o},\bar{\mathscr{S}}(x_u)\right>\bar{\mathscr{S}}(x_{vw_o/u})=\bar{\mathscr{S}}(x_{vw_o/ww_o})\,.
\]
If we combine the previous and the third previous displayed formula, we get the claimed formula for $x_{w/v}$. This completes the proof.
\end{proof}

\begin{cor}[{\cite[Corollary~3.6]{liu}}]
\label{cor:xwv-pos}

Let $v,w\in W$ be such that $v\leq w$. Let $s_{\beta_1}\ldots s_{\beta_\ell}$ be a reduced expression of $vw_o$. Let $\alpha_i=s_{\beta_1}\cdots s_{\beta_{i-1}}(\beta_i)$ and $y_i=x_{\alpha_i}$ for all $1\leq i\leq\ell$. Then we have
\[
x_{w/v}=\sum_J\prod_{j\in J}y_j
\]
where the sum ranges over all subsets $J$ of $\{1,\ldots,\ell\}$ such that the product $\prod_{j\in\bar{J}}s_{\beta_j}$ is a reduced expression of $ww_o$. Here, the set $\bar{J}$ denotes the complement of a subset $J$ in $\{1,\ldots,\ell\}$.

\end{cor}

\begin{proof}

This corollary follows immediately from Theorem~\ref{thm:positivity} and Corollary~\ref{cor:sxwv}.
\end{proof}

\begin{cor}
\label{cor:xwv-circ-pos}

Let $v,w\in W$ be such that $v\leq w$. Let $s_{\beta_1}\cdots s_{\beta_\ell}$ be a reduced expression of $w_ov$. Let $\alpha_i=s_{\beta_\ell}\cdots s_{\beta_{i+1}}(\beta_i)$ and $y_i=x_{\alpha_i}$ for all $1\leq i\leq\ell$. Then we have
\[
x_{w/v}^\circ=\sum_J\prod_{j\in J}y_j
\]
where the sum ranges over all subsets $J$ of $\{1,\ldots,\ell\}$ such that the product $\prod_{j\in\bar{J}}s_{\beta_j}$ is a reduced expression of $w_ow$. Here, the set $\bar{J}$ denotes the complement of a subset $J$ in $\{1,\ldots,\ell\}$.

\end{cor}

\begin{proof}

This corollary follows immediately from Corollary~\ref{cor:xwv-pos} applied to $v^{-1},w^{-1}$, \cite[Exercise~5.9]{humphreys-coxeter} and from the definition of $\rho$ and $x_{w/v}^\circ$.
\end{proof}

\todo[inline,color=green]{Circle analogue of the previous corollary is supposed to come. [It is there: Corollary~\ref{cor:xwv-circ-pos}.] Same holds for monomial properties of order one and two (for the latter case all in one statement [cf.\ Corollary~\ref{cor:monomial-one}, Theorem~\ref{thm:monomial-two}]). Also, circle analogue of the recursive formula [cf.\ Corollary~\ref{cor:recursive-circ}]. Note, for Theorem~\ref{thm:positivity} there is no circle analogue because the map $\bar{\mathscr{S}}$ changes.}

\begin{prop}
\label{prop:recursive-prel}

Let $v,w\in W$ be such that $v\leq w$. Let $\beta\in\Delta$ be such that $ws_\beta<w$. Let $v'=vs_\beta$, $w'=ws_\beta$, $\alpha=-w(\beta)$ for brevity. Then we have $\bar{\mathscr{S}}(x_{w/v})=\bar{\mathscr{S}}(x_{w'/v})x_\alpha$ if $v<v'$ and $=\bar{\mathscr{S}}(x_{w'/v})x_\alpha+\bar{\mathscr{S}}(x_{w'/v'})$ if $v>v'$.

\end{prop}

\begin{proof}

Let the notation be as in the statement. It is easy to see from the strong exchange condition \cite[Theorem~5.8]{humphreys-coxeter} that there exists a reduced expression $s_{\beta_1}\cdots s_{\beta_\ell}$ of $w$ such that $\beta=\beta_\ell$. We fix such a reduced expression and define $\alpha_i=s_{\beta_1}\cdots s_{\beta_{i-1}}(\beta_i)$ and $y_i=x_{\alpha_i}$ for all $1\leq i\leq\ell$ as in the statement of Corollary~\ref{cor:sxwv}. If we know apply Corollary~\ref{cor:sxwv} to $v$ and to the fixed reduced expression of $w$, we can write
\[
\bar{\mathscr{S}}(x_{w/v})=\sum_{J\colon\ell\in J}\prod_{j\in J}y_j+\sum_{J\colon\ell\notin J}\prod_{j\in J}y_j
\]
where the two sums range over all subsets $J$ of $\{1,\ldots,\ell\}$ such that the product $\prod_{j\in\bar{J}}s_{\beta_j}$ is a reduced expression of $v$ plus an additional condition which is written in the subscripts of the sums. Here, the set $\bar{J}$ denotes the complement of a subset $J$ in $\{1,\ldots,\ell\}$. 
The first sum in the previous displayed formula can clearly be written as
$
\sum_{J'}\prod_{j\in J'}y_j x_\alpha
$
where the sum ranges over all subsets $J'$ of $\{1,\ldots,\ell-1\}$ such that the product $\prod_{j\in\bar{J}'}s_{\beta_j}$ is a reduced expression of $v$. Here, the set $\bar{J}'$ denotes the complement of a subset $J'$ in $\{1,\ldots,\ell-1\}$. By Corollary~\ref{cor:sxwv} applied to $v$ and to the reduced expression $s_{\beta_1}\cdots s_{\beta_{\ell-1}}$ of $w'$, we see that
\[
\bar{\mathscr{S}}(x_{w/v})=\bar{\mathscr{S}}(x_{w'/v})x_\alpha+\sum_{J\colon \ell\notin J}\prod_{j\in J}y_j
\]
where sum ranges over all $J$ as in the first displayed formula of this proof.

We now distinguish the two cases as in the statement of the proposition. Suppose first that $v<v'$. Then there exists no reduced expression of $v$ which ends with $s_\beta$. Consequently, the sum in the previous displayed formula must be zero. This proves the claimed formula if $v<v'$. Suppose finally that $v>v'$. In this case, the sum in the previous displayed formula can clearly be written as $\sum_{J'}\prod_{j\in J'}y_j$ where the sum ranges over all subsets $J'$ of $\{1,\ldots,\ell-1\}$ such that the product $\prod_{j\in\bar{J}'}s_{\beta_j}$ is a reduced expression of $v'$. Here, the set $\bar{J}'$ denotes the complement of a subset $J'$ in $\{1,\ldots,\ell-1\}$. If we now apply Corollary~\ref{cor:sxwv} to $v'$ and to the reduced expression $s_{\beta_1}\cdots s_{\beta_{\ell-1}}$ of $w'$, we find that the sum in the previous displayed formula must be equal to $\bar{\mathscr{S}}(x_{w'/v'})$. This proves the claimed formula if $v>v'$ and completes the proof.
\end{proof}

\begin{cor}[{\cite[Corollary~3.8]{liu}}]
\label{cor:recursive}

Let $v,w\in W$ be such that $v\leq w$. Let $\beta\in\Delta$ be such that $v<vs_\beta$. Let $v'=vs_\beta$, $w'=ws_\beta$, $\alpha=v(\beta)$ for brevity. Then we have $x_{w/v}=x_{w/v'}x_\alpha$ if $w>w'$ and $=x_{w/v'}x_\alpha+x_{w'/v'}$ if $w<w'$.

\end{cor}

\begin{proof}

Let the notation be as in the statement. Let $\beta'=-w_o(\beta)$. By \cite[Section~1.8]{humphreys-coxeter}, we know that $\beta'$ is a simple root. Let $v^*=vw_o$, $w^*=ww_o$, $v^{*\,\prime}=v^*s_{\beta'}$, $w^{*\,\prime}=w^*s_{\beta'}$ for brevity. Note that by assumption and \cite[Exercise~5.9, Example~5.9.3]{humphreys-coxeter}, it is clear that $w^*\leq v^*$ and that $v^{*\,\prime}<v^*$. If we now apply Theorem~\ref{thm:positivity} and Proposition~\ref{prop:recursive-prel} to $w^*,v^*,\beta'$ we find that 
\[
x_{w/v}=\bar{\mathscr{S}}(x_{v^*/w^*})=\bar{\mathscr{S}}(x_{v^{*\,\prime}/w^*})x_\alpha=x_{w/v'}x_\alpha
\]
if $w^*<w^{*\,\prime}$ and similarly $=x_{w/v'}x_\alpha+x_{w'/v'}$ if $w^*>w^{*\,\prime}$. (For this step, note that $-v^*(\beta')=\alpha$ and that $v^{*\,\prime}w_o=v'$.) But by \cite[loc.~cit.]{humphreys-coxeter} the conditions $w^*\lessgtr w^{*\,\prime}$ are equivalent to $w\gtrless w'$ which gives the desired result.
\end{proof}

\begin{cor}
\label{cor:recursive-circ}

Let $v,w\in W$ be such that $v\leq w$. Let $\beta\in\Delta$ be such that $v<s_\beta v$. Let $v'=s_\beta v$, $w'=s_\beta w$, $\alpha=v^{-1}(\beta)$ for brevity. Then we have $x_{w/v}^\circ=x_\alpha x_{w/v'}^\circ$ if $w>w'$ and $=x_\alpha x_{w/v'}^\circ+x_{w'/v'}^\circ$ if $w<w'$.

\end{cor}

\begin{proof}

This follows immediately from Proposition~\ref{prop:rhoands}\eqref{item:identity},\eqref{item:rhoantialgebra}, Corollary~\ref{cor:recursive} applied to $v^{-1}$, $w^{-1}$, $\beta$ and \cite[loc.~cit.]{humphreys-coxeter}.
\end{proof}



\section{Monomial properties}
\label{sec:monomial}

In this section, we show monomial properties for the elements $x_{w/v}$ where $v\leq w$, i.e.\ properties which guarantee that $x_{w/v}$ can be written as a positive monomial in $\mathscr{B}_W$ of the form $x_{\alpha_1}\cdots x_{\alpha_\ell}$ where $\alpha_1,\ldots,\alpha_\ell\in R^+$. While the monomial property of order one (Theorem~\ref{thm:monomial-one}) is a trivial and general consequence of results from the previous section, the monomial property of order two (Theorem~\ref{thm:monomial-two}) relies on the combinatorics developed in Appendix~\ref{appendix:shuffle}.

\begin{thm}[{Monomial property of order one \cite[Chapter~2, Example~3]{macdonald1991notes}}]
\label{thm:monomial-one}

Let $v,w\in W$ be such that $v\lessdot w$. Let $\alpha\in R^+$ be uniquely determined by the equation $v=s_\alpha w$. Then we have $x_{w/v}=x_\alpha$.

\end{thm}

\begin{proof}

Let $v,w,\alpha$ be as in the statement. By assumption, it is clear that $x_{w/v}$ is a homogeneous element of $\mathbb{Z}_{\geq 0}$-degree $1$ and of $W$-degree $s_\alpha$ (cf. Remark~\ref{rem:degree-xwv}). By definition of $V_W$, this means that $x_{w/v}=\lambda x_\alpha$ for some $\lambda\in\mathbb{C}$. By definition of the restriction of the Hopf duality pairing between $\mathscr{B}_W$ and itself to $V_W\otimes V_W$, we know that $\lambda=\left<x_\alpha,x_{w/v}\right>$. By Proposition~\ref{prop:simple-left-derivative}, we have $\overrightarrow{D}_\alpha(x_w)=x_v$. On the other hand, the very definition of $\overrightarrow{D}_\alpha$ and of $x_{w/v}$ tells us in comparison with the above expression that $\overrightarrow{D}_\alpha(x_w)=\left<x_\alpha,x_{w/v}\right>x_v$. Hence, we have $\lambda=1$ and $x_{w/v}=x_\alpha$ as claimed.
%
%
\end{proof}

\todo[inline,color=green]{The other direction of the previous theorem is actually also true (because of the Remark~\ref{rem:degree-xwv}). Calls for a remark at this point which I want to cite in the proof of Proposition~\ref{prop:reflection-skew}. In fact, I also want to mention that the propositions in the section on strict positivity are generalizations of what we already did in the section on positivity (uses maybe the example on $x_{w/1}$, i.e.\ Example~\ref{ex:top-low-degree-xw}). [I cite Remark~\ref{rem:converse-monom-one} once in Proposition~\ref{prop:reflection-skew}. For Proposition~\ref{prop:reflection-skew},~\ref{prop:sat-chain-general} and Theorem~\ref{thm:one-property}, I say either in the header or in a remark afterwards what they generalize. That the \enquote{header-explanation} uses Example~\ref{ex:top-low-degree-xw} is implicit.]}

\begin{cor}
\label{cor:monomial-one}

Let $v,w\in W$ be such that $v\lessdot w$. Let $\alpha\in R^+$ be uniquely determined by the equation $v=w s_\alpha$. Then we have $x_{w/v}^\circ=x_\alpha$.

\end{cor}

\begin{proof}

This follows immediately from Theorem~\ref{thm:monomial-one} applied to $v^{-1}\lessdot w^{-1}$.
\end{proof}

\begin{rem}
\label{rem:converse-monom-one}

The converse of Theorem~\ref{thm:monomial-one} and Corollary~\ref{cor:monomial-one} is clearly also true. Let $v,w\in W$ be such that $x_{w/v}=x_\alpha$ for some $\alpha\in R^+$. Then, $v\lessdot w$ and $v=s_\alpha w$. Similarly, let $v,w\in W$ be such that $x_{w/v}^\circ=x_\alpha$ for some $\alpha\in R^+$. Then, $v\lessdot w$ and $v=ws_\alpha$. This follows directly from Remark~\ref{rem:degree-xwv}.

\end{rem}

\begin{lem}
\label{lem:unique}

Let $\alpha,\alpha',\gamma,\gamma'\in R^+$ be such that $\alpha\neq\gamma,\alpha'\neq\gamma'$ and such that $x_\alpha x_\gamma=x_{\alpha'}x_{\gamma'}$ in $\mathscr{B}_W$. Then we have $\{\alpha,\gamma\}=\{\alpha',\gamma'\}$.

\end{lem}

\begin{proof}

Let $\alpha,\alpha',\gamma,\gamma'$ be as in the statement. Using the canonical homogeneous basis of $T(V_W)$ induced by the canonical homogeneous basis of $V_W$, it is easy to see that $x_{\alpha}\otimes x_{\gamma}-x_{\alpha'}\otimes x_{\gamma'}\in V_W^{\otimes 2}$ lies in the kernel of $[2]!=1+\Psi$ only if $\{\alpha,\gamma\}=\{\alpha',\gamma'\}$.
\end{proof}

\begin{thm}[Monomial property of order two]
\label{thm:monomial-two}

Let $W$ be a simply laced Weyl group.
Let $v,w\in W$ be such that $v\leq w$ and such that $\ell(v,w)=2$. Then, there exist positive roots $\alpha,\gamma,\alpha^\circ,\gamma^\circ$ such that $x_{w/v}=x_\alpha x_\gamma$ and such that $x_{w/v}^\circ=x_{\alpha^\circ}x_{\gamma^\circ}$. Each choice of such positive roots $\alpha,\gamma,\alpha^\circ,\gamma^\circ$ satisfies $\alpha\neq\gamma$, $\alpha^\circ\neq\gamma^\circ$. Moreover, the sets $\{\alpha,\gamma\}$, $\{\alpha^\circ,\gamma^\circ\}$ are uniquely determined by $v,w$.

\end{thm}

\todo[inline,color=green]{It seems that I have to repeat this footnote once more in the introduction. [I decided later to leave out all footnotes and all other references to Subsubsection~\ref{subsubsec:simply-laced} in case I use the words \enquote{simply laced Weyl group} or \enquote{simply laced root system}, because I feel it is well-known terminology. This is also conform with the other papers I wrote.]}

\begin{proof}

Let $v$ and $w$ be as in the statement. Since $v\leq w$ if and only if $v^{-1}\leq w^{-1}$, it is clear that the statements for $x_{w/v}$ are equivalent to those for $x_{w/v}^\circ$. Therefore, it suffices to prove those for $x_{w/v}^\circ$. Once the existence of $\alpha^\circ$ and $\gamma^\circ$ is established, it follows from Theorem~\ref{thm:strict-positivity} and Example~\ref{ex:relations} that $\alpha^\circ\neq\gamma^\circ$ (since otherwise $x_{w/v}^\circ=0$). Moreover, the uniqueness of the set $\{\alpha^\circ,\gamma^\circ\}$ is evident from Lemma~\ref{lem:unique}. Hence, it suffices to show the existence of $\alpha^\circ$ and $\gamma^\circ$.

Suppose there exists $\beta\in\Delta$ such that $v<s_\beta v$ and such that $s_\beta v\not\leq w$. By \cite[Proposition~5.9]{humphreys-coxeter} applied to $v\leq w$, we then have automatically $s_\beta v\leq s_\beta w$ and thus $w<s_\beta w$. If we apply Corollary~\ref{cor:recursive-circ} to $v,w,\beta$, we see that $x_{w/v}^\circ=x_{s_\beta w/s_\beta v}^\circ$. Since the length of all Weyl group elements is bounded by $\ell(w_o)$, we can by repeating the previous procedure finitely many times and by replacing each time $v$ and $w$ with $s_\beta v$ and $s_\beta w$ respectively assume that for all $\beta\in\Delta$ such that $v<s_\beta v$ the inequality $s_\beta v\leq w$ holds. Since $W$ is assumed to be a simply laced Weyl group, this is precisely the situation where Theorem~\ref{thm:main-interval2} applies. Thus, there exists $\beta\in\Delta$ such that $v<s_\beta v\lessdot w$ and such that $s_\beta w<w$. With $\beta$ as in the last sentence, we define $\alpha^\circ=v^{-1}(\beta)$. Moreover, we uniquely define $\gamma^\circ\in R^+$ such that $s_\beta v=ws_{\gamma^\circ}$. By Corollary~\ref{cor:recursive-circ},~\ref{cor:monomial-one}, we then have $x_{w/v}^\circ=x_{\alpha^\circ}x_{w/s_\beta v}^\circ=x_{\alpha^\circ}x_{\gamma^\circ}$ -- as required.
\end{proof}

\begin{ex}
\label{ex:monom-two}

We give an example of a non simply laced Weyl group for which Theorem~\ref{thm:monomial-two} fails. Indeed, for this example, let $W$ be the Weyl group of rank two with simple roots $\beta$ and $\beta'$ such that $m(\beta,\beta')=4$. Similarly as in Example~\ref{ex:shuffle1}, let $v=s_{\beta'}$ and $w=s_{\beta'}s_\beta s_{\beta'}$. Then, $v\leq w$ such that $\ell(v,w)=2$ but $x_{w/v}$ cannot be written as $\lambda x_\alpha x_\gamma$ for some $\alpha,\gamma\in R^+$ and some $\lambda\in\mathbb{C}$. 
Indeed, in view of the reduced expression $s_{\beta'}s_\beta s_{\beta'}s_\beta$ of $w_o$, we find $vw_o=s_\beta s_{\beta'}s_\beta$ and $ww_o=s_\beta$. Corollary~\ref{cor:xwv-pos} applied to this reduced expression of $vw_o$ gives $x_{w/v}=x_{\alpha_2}x_{\alpha_3}+x_{\alpha_1}x_{\alpha_2}$ where $\alpha_1=\beta$, $\alpha_2=s_\beta(\beta')$, $\alpha_3=s_\beta s_{\beta'}(\beta)=s_{\beta'}(\beta)$. A direct computation shows that
\[
x_{\alpha_2}\otimes x_{\alpha_3}+x_{\alpha_1}\otimes x_{\alpha_2}-\lambda x_\alpha\otimes x_\gamma\in V_W^{\otimes 2}
\]
where $\alpha,\gamma\in R^+$ and where $\lambda\in\mathbb{C}$ never lies in the kernel of $[2]!=1+\Psi$.

\end{ex}

\begin{lem}
\label{lem:nec}

Let $v,w\in W$ be such that $v\leq w$. Let $\ell=\ell(v,w)$ for brevity. Suppose that $x_{w/v}$ can be written as $\lambda x_{\alpha_1}\cdots x_{\alpha_\ell}$ for some $\alpha_1,\ldots,\alpha_\ell\in R^+$ and some $\lambda\in\mathbb{C}$. Then we have
\[
\{\alpha\in R^+\mid v\lessdot s_\alpha v\leq w\}\subseteq\{\alpha_1,\ldots,\alpha_\ell\}\,,
\]
and in particular
\[
\#\{\alpha\in R^+\mid v\lessdot s_\alpha v\leq w\}\leq\ell\,.
\]

\end{lem}

\begin{proof}

Let the notation be as in the statement. It suffices to show the claimed inclusion as the inequality follows by taking cardinalities. Let $\alpha\in R^+$ be such that $v\lessdot s_\alpha v\leq w$. By Theorem~\ref{thm:strict-positivity} and Proposition~\ref{prop:reflection-skew}, we know that $(x_{w/v})\overleftarrow{D}_\alpha=x_{w/s_\alpha v}\neq 0$. On the other hand, the braided Leibniz rule for $\overleftarrow{D}_\alpha$ (Remark~\ref{rem:braided-leibniz}) repeatedly applied to the expression $\lambda x_{\alpha_1}\cdots x_{\alpha_\ell}$ shows that $(x_{w/v})\overleftarrow{D}_\alpha\neq 0$ can only happen if $\alpha\in\{\alpha_1,\ldots,\alpha_\ell\}$. This completes the proof.
\end{proof}

\begin{rem}

Let $v,w\in W$ be such that $v\leq w$. Let $\ell=\ell(v,w)$ for brevity. The inequality in Lemma~\ref{lem:nec} is only necessary for $x_{w/v}$ to be writable as a monomial $\lambda x_{\alpha_1}\cdots x_{\alpha_\ell}$ for some $\alpha_1,\ldots,\alpha_\ell\in R^+$ and some $\lambda\in\mathbb{C}$ but not sufficient. Indeed, with the setup as in Example~\ref{ex:monom-two}, the set 
\[
\{\alpha\in R^+\mid v\lessdot s_\alpha v\leq w\}=\{\alpha_1,\alpha_3\}
\]
has cardinality equal to $\ell(v,w)=2$.

\end{rem}

\begin{rem-prob}

Let $v,w\in W$ be such that $v\leq w$. Let $\ell=\ell(v,w)$ for brevity. For a simply laced Weyl group $W$, except for the case $\ell\leq 2$, it is unclear to the author whether or not the inequality in Lemma~\ref{lem:nec} is sufficient for $x_{w/v}$ to be writable as a monomial $\lambda x_{\alpha_1}\cdots x_{\alpha_\ell}$ for some $\alpha_1,\ldots,\alpha_\ell\in R^+$ and some $\lambda\in\mathbb{C}$.

\end{rem-prob}

\begin{ex}

We give an example which shows that Theorem~\ref{thm:monomial-two} fails for Bruhat intervals of length three. For this example, let $R$ be of type $\mathsf{A}_3$. Let $\beta_1,\beta_2,\beta_3$ be the simple roots with the labeling as in \cite[Plate~I]{bourbaki_roots}. Let $v=s_{\beta_1}s_{\beta_3}$ and $w=s_{\beta_1}s_{\beta_2}s_{\beta_3}s_{\beta_2}s_{\beta_1}$. Then, we have $v\leq w$ and $\ell(v,w)=3$ but $x_{w/v}$ cannot be written as $\lambda x_\alpha x_\gamma x_\delta$ for some $\alpha,\gamma,\delta\in R^+$ and some $\lambda\in\mathbb{C}$. Indeed, the cardinality of the set
\[
\{\alpha\in R^+\mid v\lessdot s_\alpha v\leq w\}=\{\beta_2,\beta_1+\beta_2,\beta_2+\beta_3,\beta_1+\beta_2+\beta_3\}
\]
is equal to $4>\ell(v,w)=3$. Thus, the claim follows from Lemma~\ref{lem:nec}

\end{ex}

\section{Strict positivity}
\label{sec:strict-positivity}

In this section, we generalize Proposition~\ref{prop:simple-left-derivative} and \ref{prop:sat-chain} from $x_w$ to the relative setting $x_{w/v}$ where $v,w\in W$ such that $v\leq w$. We draw several consequences which lead for example to the sharpening of positivity (Theorem~\ref{thm:positivity}) to strict positivity (Theorem~\ref{thm:strict-positivity}). We also prepare the combinatorial consequences of the next section using the invariants $c_{w/v,w'/v'}$ where $v,w,v',w'\in W$ (cf.\ Definition~\ref{def:cwv}).

\begin{prop}[Generalization of Proposition~\ref{prop:simple-left-derivative}]
\label{prop:reflection-skew}

Let $v,w\in W$ and let $\alpha\in R^+$. Then we have $\overrightarrow{D}_\alpha(x_{w/v})=x_{s_\alpha w/v}$ if $s_\alpha w\lessdot w$ and $=0$ otherwise. Moreover, we have $(x_{w/v})\overleftarrow{D}_\alpha=x_{w/s_\alpha v}$ if $v\lessdot s_\alpha v$ and $=0$ otherwise.

\end{prop}

\begin{proof}

Let $v,w,\alpha$ be as in the statement. Let us first prove the first statement of the proposition. If $v\not\leq w$ then both sides of the claimed equality are $=0$ and there is nothing to prove. Therefore, we may assume that $v\leq w$. In this case the claimed equality follows immediately from the definition of $\overrightarrow{D}_\alpha$, Remark~\ref{rem:coproduct-xwv},~\ref{rem:converse-monom-one} and Theorem~\ref{thm:monomial-one}. The second statement of the proposition can be proved completely analogously to the first just with the definition of $\overrightarrow{D}_\alpha$ replaced by the definition of $\overleftarrow{D}_\alpha$.
\end{proof}

\todo[inline,color=green]{{\bf Remark.} Let $V_W^\diamond$ be the $\mathbb{C}$-vector subspace of $\mathscr{B}_W$ spanned by $x_{w/v}$ where $v,w\in W$. It is well-known or easy to see that for each $\alpha\in R^+$ there exist $v,w\in W$ such that $v\lessdot w$ and such that $v=s_\alpha w$. (This follows basically from the fact that each reflection in $W$ has a palindrome reduced expression.) 
Hence, Theorem~\ref{thm:monomial-one} implies that $V_W$ is a vector subspace of $V_W^\diamond$. With this notation, Proposition~\ref{prop:reflection-skew} implies that $\mathscr{B}_W$ acts from the left and the right on $V_W^\diamond$ via left and right braided partial derivatives.

$\bar{\mathscr{S}}(V_W^\diamond)=V_W^\diamond$, comultiplication on $V_W^\diamond$ inherited from $\mathscr{B}_W$, so that $V_W^\diamond$ becomes a coalgebra (cf. Remark~\ref{rem:coproduct-xwv}), unaware if the expected basis of $V_W^\diamond$ is actually a basis, while finite dimensionality is clear [a cheap proof that the expected basis is actually a basis is not possible, cf.\ Example~\ref{ex:comb}]. For the last point, I actually need strict positivity to make a sensible statement, so that I maybe place this remark somewhere else. Also, the order of the statements in the remark should make \enquote{sense}.

{\bf Problem.} The problem with this remark is that $V_W^\diamond$ has no sensible $W$-action, and does not belong to the world of Yetter-Drinfeld $W$-modules. The action $ux_{w/v}=x_{uw/uv}$ where $u,v,w\in W$ does not make much sense because we do not necessarily have $uv\leq uw$ if $v\leq w$ (hence a nonzero element could be send to zero). Further, $V_W^\diamond$ is not stable under the $W$-action inherited from $\mathscr{B}_W$, otherwise $\mathbb{C}x_{w_o}$ would be $W$-stable as the $\ell(w_o)$-component of $V_W^{\diamond}$ which is clearly not the case.}

\begin{cor}
\label{cor:formal-kirillov}

Let $v,w\in W$ and let $y,\xi\in\mathscr{B}_W^{\geq 0}$. Then $\overrightarrow{D}_\xi(x_{w/v})\overleftarrow{D}_y$ and $\bar{\mathscr{S}}\big(\overrightarrow{D}_\xi(x_{w/v})\overleftarrow{D}_y\big)$ are both positive elements of $\mathscr{B}_W$. 

\end{cor}

\begin{proof}

Let $v,w,y,\xi$ be as in the statement. Since
\[
\bar{\mathscr{S}}\big(\overrightarrow{D}_\xi(x_{w/v})\overleftarrow{D}_y\big)=\overrightarrow{D}_{\rho(y)}(\bar{\mathscr{S}}(x_{w/v}))\overleftarrow{D}_{\rho(\xi)}=\overrightarrow{D}_{\rho(y)}(x_{vw_o/ww_o})\overleftarrow{D}_{\rho(\xi)}
\]
by Theorem~\ref{thm:positivity}, Proposition~\ref{prop:rhoands}\eqref{item:involution},~\ref{prop:liu2.10(d)} and Remark~\ref{rem:left-right-simult},~\ref{rem:liu2.10(d)}, the positivity of the second expression in the statement of the corollary follows in in view of Remark~\ref{rem:positive-image} from the positivity of the first expression. But the positivity of the first expression is obvious because of Theorem~\ref{thm:positivity} and Proposition~\ref{prop:reflection-skew}.
\end{proof}

\begin{rem}

Let $u,v,w\in W$. Corollary~\ref{cor:formal-kirillov} shows because of Theorem~\ref{thm:positivity}, Remark~\ref{rem:positive-image} and Example~\ref{ex:top-low-degree-xw} in particular that $(x_u)\overleftarrow{D}_{w/v}$ and $\overrightarrow{D}_{w/v}^\circ(x_u)$ are both positive elements of $\mathscr{B}_W$. This result can be regarded as a \enquote{formal analogue} of \cite[Conjecture~1]{kirillov}.


\end{rem}

\todo[inline,color=green]{I need here to say that $c_{w/v,w'/v'}$ is always $\geq 0$ as a consequence of the above corollary and I refer to this remark in the introduction to the combinatorics section where I give a enumerative and manifestly positive description of this invariant. [Done in Remark~\ref{rem:cwv} and remark after Theorem~\ref{thm:comb}.]

I should mention that the previous corollary is a formal analogue of Kirillov's conjecture for the special case that\ldots maybe also relevant for the motivational section in the introduction. [Done in the remark after Corollary~\ref{cor:formal-kirillov}.]

What we also need to say in the introduction is that all this work by Liu seemed to be inspired by the conjecture of Kirillov. [Done in the introduction before Subsection~\ref{subsec:coxeter}.]}

\begin{prop}[Generalization of Proposition~\ref{prop:sat-chain}]
\label{prop:sat-chain-general}

Let $v,w\in W$ be such that $v\leq w$. Let $\ell=\ell(v,w)$. Let $\alpha_1,\ldots,\alpha_\ell\in R^+$ be such that $s_{\alpha_\ell}\cdots s_{\alpha_1}=wv^{-1}$. Then we have $\left<x_{\alpha_1}\cdots x_{\alpha_\ell},x_{w/v}\right>=1$ if
\[
v\lessdot s_{\alpha_1}v\lessdot s_{\alpha_2}s_{\alpha_1}v\lessdot\cdots\lessdot s_{\alpha_{\ell-1}}\cdots s_{\alpha_1}v\lessdot s_{\alpha_\ell}\cdots s_{\alpha_1}v=w
\]
or equivalently if
\[
v=s_{\alpha_1}\cdots s_{\alpha_\ell}w\lessdot s_{\alpha_2}\cdots s_{\alpha_\ell}w\lessdot\cdots\lessdot s_{\alpha_{\ell-1}}s_{\alpha_\ell}w\lessdot s_{\alpha_\ell}w\lessdot w
\]
is a saturated chain in the Bruhat order and $=0$ otherwise.

\end{prop}

\begin{proof}

The two displayed conditions in the statement of the proposition are clearly equivalent. The claim now follows from Proposition~\ref{prop:reflection-skew} and Equation~\eqref{eq:pairing-derivative}.
\end{proof}

\begin{prop}
\label{prop:equal}

Let $v,w,v',w'\in W$ be such that $v\leq w$ and $v'\leq w'$. Then we have $x_{w/v}=x_{w'/v'}$ if and only if the following conditions hold:

\begin{enumerate}

\item 
\label{item:zdegree}

$\ell=\ell(v,w)=\ell(v',w')$.

\item
\label{item:wdegree}

$wv^{-1}=w'v'^{-1}$.

\item
\label{item:prop}

For all $\alpha_1,\ldots,\alpha_\ell\in R^+$ such that $s_{\alpha_\ell}\cdots s_{\alpha_1}=wv^{-1}=w'v'^{-1}$ we have
\begin{gather*}
v\lessdot s_{\alpha_1}v\lessdot s_{\alpha_2}s_{\alpha_1}v\lessdot\cdots\lessdot s_{\alpha_{\ell-1}}\cdots s_{\alpha_1}v\lessdot s_{\alpha_\ell}\cdots s_{\alpha_1}v=w\\
\text{if and only if}\\
v'\lessdot s_{\alpha_1}v'\lessdot s_{\alpha_2}s_{\alpha_1}v'\lessdot\cdots\lessdot s_{\alpha_{\ell-1}}\cdots s_{\alpha_1}v'\lessdot s_{\alpha_\ell}\cdots s_{\alpha_1}v'=w'
\end{gather*}
or equivalently we have
\begin{gather*}
v=s_{\alpha_1}\cdots s_{\alpha_\ell}w\lessdot s_{\alpha_2}\cdots s_{\alpha_\ell}w\lessdot\cdots\lessdot s_{\alpha_{\ell-1}}s_{\alpha_\ell}w\lessdot s_{\alpha_\ell}w\lessdot w\\
\text{if and only if}\\
v'=s_{\alpha_1}\cdots s_{\alpha_\ell}w'\lessdot s_{\alpha_2}\cdots s_{\alpha_\ell}w'\lessdot\cdots\lessdot s_{\alpha_{\ell-1}}s_{\alpha_\ell}w'\lessdot s_{\alpha_\ell}w'\lessdot w'\,.
\end{gather*}

\end{enumerate}

\end{prop}

\begin{proof}

Let the notation be as in the statement and as in the items. Suppose that $x_{w/v}=x_{w'/v'}$. Remark~\ref{rem:degree-xwv} immediately implies Item~\eqref{item:zdegree},\eqref{item:wdegree}. The two conditions in Item~\eqref{item:prop} are clearly equivalent. The first as well as the second condition in Item~\eqref{item:prop} follow from the equation $\left<x_{\alpha_1}\cdots x_{\alpha_\ell},x_{w/v}\right>=\left<x_{\alpha_1}\cdots x_{\alpha_\ell},x_{w'/v'}\right>$ and Proposition~\ref{prop:sat-chain-general}.

Conversely, suppose that the three conditions in Item~\eqref{item:zdegree},\eqref{item:wdegree},\eqref{item:prop} are satisfied. By the nondegeneracy of the Hopf duality pairing between $\mathscr{B}_W$ and itself, it suffices to show that $\left<\phi,x_{w/v}\right>=\left<\phi,x_{w'/v'}\right>$ for all $\phi\in\mathscr{B}_W$. To show this equality, by Item~\eqref{item:zdegree},\eqref{item:wdegree}, we clearly can assume that $\phi$ is homogeneous of $\mathbb{Z}_{\geq 0}$-degree $\ell$ and of $W$-degree $vw^{-1}=v'w'^{-1}$ since otherwise both sides of the desired equality are equal to zero. (To see this, we take Remark~\ref{rem:vanishing} into account and that the Hopf duality pairing between $\mathscr{B}_W$ and itself is a morphism in the Yetter-Drinfeld category over $W$.) By linearity, we even may assume that $\phi=x_{\alpha_1}\cdots x_{\alpha_\ell}$ where $\alpha_1,\ldots,\alpha_\ell\in R^+$ are such that $s_{\alpha_\ell}\cdots s_{\alpha_1}=wv^{-1}=w'v'^{-1}$. Under this reduction, the desired equality follows immediately from Proposition~\ref{prop:sat-chain-general} and Item~\eqref{item:prop}.
\end{proof}

\begin{thm}[Strict positivity]
\label{thm:strict-positivity}

Let $v,w\in W$ be such that $v\leq w$. Then $x_{w/v}$ is a strictly positive element of $\mathscr{B}_W$.

\end{thm}

\begin{proof}

Let $v,w$ be as in the statement. By Theorem~\ref{thm:positivity} it is sufficient to prove that $x_{w/v}$ is nonzero. To this end, let $\ell=\ell(v,w)$. By \cite[Proposition~5.11]{humphreys-coxeter}, there exist $\alpha_1,\ldots,\alpha_\ell\in R^+$ such that
\[
v\lessdot s_{\alpha_1}v\lessdot s_{\alpha_2}s_{\alpha_1}v\lessdot\cdots\lessdot s_{\alpha_{\ell-1}}\cdots s_{\alpha_1}v\lessdot s_{\alpha_\ell}\cdots s_{\alpha_1}v=w\,.
\]
Proposition~\ref{prop:sat-chain-general} now implies that $\left<x_{\alpha_1}\cdots x_{\alpha_\ell},x_{w/v}\right>=1$ which in turn implies that $x_{w/v}$ is nonzero.
\end{proof}

\begin{defn}
\label{def:cwv}

Let $v,w,v',w'\in W$. Then we define an invariant depending on $v,w,v',w'$ by the equation $c_{w/v,w'/v'}=\left<x_{w/v},\bar{\mathscr{S}}(x_{w'/v'})\right>$. Moreover, we set $c_{w/v}=c_{w/v,w/v}$ for brevity.

\end{defn}

\begin{rem}
\label{rem:cwv}

Let $v,w,v',w'\in W$. By definition, Proposition~\ref{prop:adjoint} and Equation~\eqref{eq:pairing-derivative} we know that
\begin{equation}
\label{eq:cwv}
c_{w/v,w'/v'}=\left<\rho(x_{w/v}),x_{w'/v'}\right>=\epsilon\big(\overrightarrow{D}_{w^{-1}/v^{-1}}^\circ(x_{w'/v'})\big)\,.
\end{equation}
Hence, it follows from Corollary~\ref{cor:formal-kirillov} because of Theorem~\ref{thm:positivity} and Remark~\ref{rem:positive-image} that $c_{w/v,w'/v'}\in\mathbb{Z}_{\geq 0}$ and in particular that $c_{w/v}\in\mathbb{Z}_{\geq 0}$. Alternatively, we can similarly see the same result by directly using Corollary~\ref{cor:sxwv},~\ref{cor:formal-kirillov} and Equation~\eqref{eq:pairing-derivative}.
We also remark that $c_{w/v,w'/v'}$ is nonzero (i.e.\ an element of $\mathbb{Z}_{>0}$ in view of the previous sentence) only if $v\leq w$, $v'\leq w'$, $wv^{-1}=w'v'^{-1}$, $\ell(v,w)=\ell(v',w')$. Indeed, this follows immediately because the Hopf duality pairing between $\mathscr{B}_W$ and itself is a morphism in the Yetter-Drinfeld category over $W$ and by considering Proposition~\ref{prop:rhoands}\eqref{item:morphop} and Remark~\ref{rem:vanishing},~\ref{rem:degree-xwv}.

\end{rem}

\begin{rem}
\label{rem:evidence-01}

Let $w,v,w',v'\in W$. If $x_{w/v}$ is a positive monomial, i.e.\ it can be written as $x_{\alpha_1}\cdots x_{\alpha_\ell}$ for some positive roots $\alpha_1,\ldots,\alpha_\ell$, then it follows from Proposition~\ref{prop:sat-chain-general}, Remark~\ref{rem:positive-image} and Equation~\eqref{eq:cwv} that $c_{w/v,w'/v'}\in\{0,1\}$. The assumption that $x_{w/v}$ is a positive monomial is satisfied, for instance, in the following situations:
\begin{itemize}
\item
If $\ell(v,w)=0$ or if $v=1$ by Example~\ref{ex:top-low-degree-xw}.
\item
If $\ell(v,w)=1$ by Theorem~\ref{thm:monomial-one}.
\item
If $W$ is a simply laced Weyl group and if $\ell(v,w)=2$ by Theorem~\ref{thm:monomial-two}.
\end{itemize}

\end{rem}

\todo[inline,color=green]{The "proof" of this (see the commented passage below, with \enquote{this} I mean $c_{w/v,w'/v'}\in\{0,1\}$) I worked out in the math diary from the 9th January 2018 turned out to be false. I still believe the result might be true but can't prove it and will make a remark on this open problem later (in the section on combinatorial consequences). I have not verified anything with a computer. The previous remark indeed gives some evidence. Evidence is supplied by Theorem~\ref{thm:one-property} and Remark~\ref{rem:evidence-01}. [I made a Remark* after the proof of Theorem~\ref{thm:comb}.]}

\begin{prop}
\label{prop:recursive-cwv}

Let $v,w,v',w'\in W$ be such that $v\leq w$, $v'\leq w'$, $wv^{-1}=w'v'^{-1}$, $\ell(v,w)=\ell(v',w')$. Let $\beta\in\Delta$ be such that $v<vs_\beta$. Let $\alpha=v(\beta)$ for brevity. Then we have
\[
c_{w/v,w'/v'}=\begin{cases}
c_{w/vs_\beta,w'/s_\alpha v'}&\text{if }w>ws_\beta\text{ and }v'\mspace{1mu}\lessdot\mspace{1mu}s_\alpha v'\,,\\
0&\text{if }w>ws_\beta\text{ and }v'\centernot\lessdot s_\alpha v'\,,\\
c_{w/vs_\beta,w'/s_\alpha v'}+c_{ws_\beta/vs_\beta,w'/v'}&\text{if }w<ws_\beta\text{ and }v'\mspace{1mu}\lessdot\mspace{1mu}s_\alpha v'\,,\\
c_{ws_\beta/vs_\beta,w'/v'}&\text{if }w<ws_\beta\text{ and }v'\centernot\lessdot s_\alpha v'\,.
\end{cases}
\]

\end{prop}

\begin{proof}

Before we prove the statement of the proposition, we make a preliminary observation. Let $v,w\in W$ be such that $v\leq w$. Let $\phi,\psi\in\mathscr{B}_W$. Then we have
\[
\left<\phi\psi,\bar{\mathscr{S}}(x_{w/v})\right>=\sum_{v\leq u\leq w}\left<\phi,\bar{\mathscr{S}}(x_{w/u})\right>\left<\psi,\bar{\mathscr{S}}(x_{u/v})\right>\,.
\]
Indeed, by Proposition~\ref{prop:rhoands}\eqref{item:sbaranticolagebra} and Remark~\ref{rem:coproduct-xwv}, we know that
\[
\Delta(\bar{\mathscr{S}}(x_{w/v}))=\sum_{v\leq u\leq w}\bar{\mathscr{S}}(x_{u/v})\otimes\bar{\mathscr{S}}(x_{w/u})\,.
\]
The second previous equation now follows from the last in view of the definition of the Hopf duality pairing between $\mathscr{B}_W$ and itself. We now prove the claimed formula in the statement of the proposition. Let $v,w,v',w',\beta,\alpha$ be as in the statement of the proposition. Suppose first that $w>ws_\beta$. By Corollary~\ref{cor:recursive} applied to $x_{w/v}$ and by the preliminary observation in the beginning of the proof, we see that 
\[
c_{w/v,w'/v'}=\left<x_{w/vs_\beta}x_\alpha,\bar{\mathscr{S}}(x_{w'/v'})\right>=
\sum_{v'\leq u'\leq w'}c_{w/vs_\beta,w'/u'}\left<x_\alpha,\bar{\mathscr{S}}(x_{u'/v'})\right>\,.
\]
For each $v'\leq u'\leq w'$, by Proposition~\ref{prop:rhoands}\eqref{item:identity},~\ref{prop:adjoint},~\ref{prop:sat-chain-general}, we now have $\left<x_\alpha,\bar{\mathscr{S}}(x_{u'/v'})\right>=\left<x_\alpha,x_{u'/v'}\right>=1$ if $v'\lessdot s_\alpha v'=u'$ and $=0$ otherwise. In view of the last displayed equation, this covers the first two cases of the four cases in the formula in the statement of the proposition. Finally, we may assume that $w<ws_\beta$. Again, by Corollary~\ref{cor:recursive} applied to $x_{w/v}$,
we see that
\[
c_{w/v,w'/v'}=\left<x_{w/vs_\beta}x_\alpha,\bar{\mathscr{S}}(x_{w'/v'})\right>+c_{ws_\beta/vs_\beta,w'/v'}\,.
\]
But by identically the same arguments as for the case $w>ws_\beta$, we see in this case that the first summand in the previous displayed equation is $=c_{w/vs_\beta,w'/s_\alpha v'}$ if $v'\lessdot s_\alpha v'$ and $=0$ otherwise. This fact and the previous displayed equation now cover the last two cases of the four cases in the formula in the statement of the proposition.
\end{proof}

\begin{thm}
\label{thm:one-property}

Let $v,w\in W$ be such that $v\leq w$. Then we have $c_{w/v}=1$.

\end{thm}

\begin{proof}

Let $\beta\in\Delta$ be such that $v<vs_\beta$. Let $v'=vs_\beta$, $w'=ws_\beta$, $\alpha=v(\beta)$, $\alpha'=w(\beta)$ for brevity. Then it is clear that $s_\alpha v=vs_\beta=v'>v$. Assume first that $w<w'$. Then, $v'\leq w'$ by \cite[Proposition~5.9]{humphreys-coxeter}. If we now apply Proposition~\ref{prop:recursive-prel} to $v',w',\beta$, we find that $\bar{\mathscr{S}}(x_{w'/v'})=\bar{\mathscr{S}}(x_{w/v'})x_{\alpha'}+\bar{\mathscr{S}}(x_{w/v})$ because $-w'(\beta)=\alpha'$. Thus, we also have 
\[
c_{w'/v',w/v}=c_{w'/v'}-\left<x_{w'/v'},\bar{\mathscr{S}}(x_{w/v'})x_{\alpha'}\right>\,.
\]
by applying $\left<x_{w'/v'},-\right>$ to the equation in the last sentence and by rearranging. If we use Proposition~\ref{prop:sat-chain-general}, Remark~\ref{rem:coproduct-xwv} and the definition of the Hopf duality pairing between $\mathscr{B}_W$ and itself, we can evaluate the very last term in the last equation as
\[
\sum_{v'\leq u'\leq w'}\left<x_{u'/v'},\bar{\mathscr{S}}(x_{w/v'})\right>\left<x_{w'/u'},x_{\alpha'}\right>=c_{w/v'}
\]
because $s_{\alpha'}w'=w<w'$. If we put the last two displayed equations together, we see from the third case of Proposition~\ref{prop:recursive-cwv} applied to $c_{w/v}$ that
\[
c_{w/v}=c_{w/v'}+c_{w'/v',w/v}=c_{w'/v'}\,.
\]
By applying the first case of Proposition~\ref{prop:recursive-cwv} to $c_{w/v}$
and taking into account the third case treated in the last equation above, we see in general that $c_{w/v}=c_{w/v'}$ if $w>w'$ and $=c_{w'/v'}$ if $w<w'$.
Note that $v'\leq w$ if $w>w'$ and that $v'\leq w'$ if $w<w'$ by \cite[Proposition~5.9]{humphreys-coxeter} applied to $v\leq w$; hence, in both cases, we are in the situation we started with.
Since the length of each Coxeter group element is bounded by $\ell(w_o)$, we see by applying the two cases in the equation in the second last sentence finitely many times (until there exists no $\beta$ anymore as we chose it in the beginning) that $c_{w/v}=c_{w_o/w_o}$. In view of Proposition~\ref{prop:rhoands}\eqref{item:identity} and Example~\ref{ex:top-low-degree-xw}, it obviously follows that $c_{w/v}=1$ -- as claimed.
\end{proof}

\begin{rem}

Indirectly, we see once more from Theorem~\ref{thm:one-property} that $x_{w/v}\neq 0$ for all $v,w\in W$ such that $v\leq w$ as it was already shown in Theorem~\ref{thm:strict-positivity}.

\end{rem}

\begin{rem}

The second statement of Proposition~\ref{prop:sat-chain} says that $\left<\rho(x_w),x_w\right>=1$ for all $w\in W$. In view of Example~\ref{ex:top-low-degree-xw} and Equation~\eqref{eq:cwv}, Theorem~\ref{thm:one-property} can therefore be seen as a generalization of the second statement of Proposition~\ref{prop:sat-chain}.

\end{rem}

\todo[inline,color=green]{The previous result can be seen as a generalization of $\left<x_{w^{-1}},x_w\right>=\left<\rho(x_w),x_w\right>=1$. Reference to example (Example~\ref{ex:top-low-degree-xw}), to equation with $\rho$ (Equation~\eqref{eq:cwv}), and to the original result (Proposition~\ref{prop:sat-chain}).}

\section{Combinatorial consequences}
\label{sec:comb}

In this section, we give an enumerative description of the invariants $c_{w/v,w'/v'}$ where $v$, $w$, $v'$, $w'\in W$ (cf.\ Theorem~\ref{thm:comb}). This leads to a sharpening of the well-known existence of saturated chains in the Bruhat order (\cite[Proposition~5.11]{humphreys-coxeter}): As a consequence of Theorem~\ref{thm:one-property}, we construct in  a controlled and unique way for an interval $v\leq w$ where $v,w\in W$ a saturated chain from $v$ to $w$ depending on a reduced expression of $w$ (cf.\ Corollary~\ref{cor:comb}).

\begin{thm}
\label{thm:comb}

Let $v,w,v',w'\in W$ be such that $v\leq w$, $v'\leq w'$, $wv^{-1}=w'v'^{-1}$, $m=\ell(v,w)=\ell(v',w')$. Let $s_{\beta_1}\cdots s_{\beta_\ell}$ be a reduced expression of $w'$. Let $\alpha_i=s_{\beta_1}\cdots s_{\beta_{i-1}}(\beta_i)$ and let $\alpha_i'=s_{\beta_{\ell}}\cdots s_{\beta_{i+1}}(\beta_i)$ for all $1\leq i\leq\ell$. Then we have
\begin{alignat*}{3}
c_{w/v,w'/v'}=\#\Bigg\{&
\centermathcell{
\begin{gathered}
\textstyle{J=\{1\leq j_1<\cdots<j_m\leq\ell\}\colon v'=\prod_{j\in\bar{J}}s_{\beta_j}\,,}\\
\textstyle{v\lessdot s_{\alpha_{j_1}}v\lessdot\cdots\lessdot s_{\alpha_{j_{m-1}}}\cdots s_{\alpha_{j_1}}v\lessdot s_{\alpha_{j_m}}\cdots s_{\alpha_{j_1}}v=w}
\end{gathered}}
&&\Bigg\}\\
\intertext{and equivalently we also have}
c_{w^{-1}/v^{-1},w'^{-1}/v'^{-1}}=\#\Bigg\{&
\centermathcell{
\begin{gathered}
\textstyle{J=\{1\leq j_1<\cdots<j_m\leq\ell\}\colon v'=\prod_{j\in\bar{J}}s_{\beta_j}\,,}\\
\textstyle{v\lessdot v s_{\alpha_{j_m}'}\lessdot\cdots\lessdot v s_{\alpha_{j_m}'}\cdots s_{\alpha_{j_2}'}\lessdot v s_{\alpha_{j_m}'}\cdots s_{\alpha_{j_1}'}=w}
\end{gathered}}
&&\Bigg\}
\end{alignat*}
where $\bar{J}$ denotes the complement of a subset $J$ in $\{1,\ldots,\ell\}$. In particular, the cardinalities on the right side of the displayed equations are independent of the choice of the reduced expression of $w'$.

\end{thm}

\begin{proof}

Let the notation be as in the statement. Once the claimed equalities are proved, it is clear that the cardinalities on the right side of the displayed equations are independent of the choice of the reduced expression of $w'$. The second equation in the statement of the proposition follows by applying the first to $v^{-1},w^{-1},v'^{-1},w'^{-1}$ and to the reduced expression $s_{\beta_\ell}\cdots s_{\beta_1}$ of $w'^{-1}$ and vice versa. (For this step, we also take \cite[Exercise~5.9]{humphreys-coxeter} into account.) Therefore, it suffices to prove the first equation in the statement of the proposition. To this end, let $y_i=x_{\alpha_i}$ for all $1\leq i\leq\ell$. By Corollary~\ref{cor:sxwv} applied to $\bar{\mathscr{S}}(x_{w'/v'})$, we then have
\[
c_{w/v,w'/v'}=\sum_J\Big<x_{w/v},\prod_{j\in J}y_j\Big>
\]
where the sum ranges over all subsets $J$ of $\{1,\ldots,\ell\}$ such that the product $\prod_{j\in \bar{J}}s_{\beta_j}$ is a reduced expression of $v'$. Here, the set $\bar{J}$ denotes the complement of a subset $J$ in $\{1,\ldots,\ell\}$. The claim now follows from Proposition~\ref{prop:sat-chain-general} and the previous displayed equation.
\end{proof}

\begin{rem}

We see once more from Theorem~\ref{thm:comb} that $c_{w/v,w'/v'}\in\mathbb{Z}_{\geq 0}$ for all $v,w,v',w'\in W$ as it was already remarked in Remark~\ref{rem:cwv}.

\end{rem}

\begin{rem-prob}

Theorem~\ref{thm:one-property} and Remark~\ref{rem:evidence-01} discuss instances of $v,w,v',w'\in W$ such that $c_{w/v,w'/v'}\in\{0,1\}$. The author was not able to find an example of $v,w,v',w'\in W$ such that $c_{w/v,w'/v'}>1$. It might be true in general that $c_{w/v,w'/v'}\in\{0,1\}$ for all $v,w,v',w'\in W$.

\end{rem-prob}

\begin{ex}
\label{ex:comb}

We give an example of $v,w,v',w'\in W$ such that $c_{w/v,w'/v'}=1$ but $x_{w/v}\neq x_{w'/v'}$. Indeed, for this example, let $R$ be of type $\mathsf{A}_3$. Let $\beta_1,\beta_2,\beta_3$ be the simple roots with the labeling as in \cite[Plate~I]{bourbaki_roots}. Let $v=s_{\beta_2}s_{\beta_1}$, $w=s_{\beta_1}s_{\beta_2}s_{\beta_3}s_{\beta_2}s_{\beta_1}$, $v'=s_{\beta_2}s_{\beta_1}s_{\beta_2}$, $w'=s_{\beta_1}s_{\beta_2}s_{\beta_3}s_{\beta_2}s_{\beta_1}s_{\beta_2}$. We see that $v\leq w$, $v'\leq w'$, $wv^{-1}=w'v'^{-1}=s_{\beta_1}s_{\beta_2}s_{\beta_3}$, $\ell(v,w)=\ell(v',w')=3$. Theorem~\ref{thm:comb} shows that $c_{w/v,w'/v'}=1$ because there exists precisely one set $J$, namely $J=\{1,3,4\}$, meeting the prescribed conditions in the first displayed equality of Theorem~\ref{thm:comb} applied to the reduced expression of $w'$ written in the second last sentence. On the other hand, let $\alpha=\beta_1$. Then, $v\lessdot s_\alpha v\leq w$ but $v'\centernot\lessdot s_\alpha v'$. Proposition~\ref{prop:equal} therefore shows that $x_{w/v}\neq x_{w'/v'}$.

\end{ex}

\begin{cor}
\label{cor:comb}

Let $v,w\in W$ be such that $v\leq w$. Let $s_{\beta_1}\cdots s_{\beta_\ell}$ be a reduced expression of $w$. Let $m=\ell(v,w)$ for brevity. Then there exists a unique sequence $1\leq j_1<\cdots<j_m\leq\ell$ such that
\begin{align*}
v=s_{\beta_1}\cdots\hat{s}_{\beta_{j_1}}\cdots\hat{s}_{\beta_{j_m}}\cdots s_{\beta_\ell}&\lessdot s_{\beta_1}\cdots\hat{s}_{\beta_{j_2}}\cdots\hat{s}_{\beta_{j_m}}\cdots s_{\beta_\ell}\lessdot\cdots\\
\cdots &\lessdot s_{\beta_1}\cdots\hat{s}_{\beta_{j_m}}\cdots s_{\beta_\ell}\lessdot s_{\beta_1}\cdots s_{\beta_\ell}=w\\
\intertext{where in the $i$th expression in the last displayed equation the sequence $j_i<\ldots<j_m$ is removed from the reduced expression of $w$ and where $i$ runs through $1,2,\ldots,m+1$. Similarly, there exists a unique sequence $1\leq k_1<\cdots<k_m\leq\ell$ such that}
v=s_{\beta_1}\cdots\hat{s}_{\beta_{k_1}}\cdots\hat{s}_{\beta_{k_m}}\cdots s_{\beta_\ell}&\lessdot s_{\beta_1}\cdots\hat{s}_{\beta_{k_1}}\cdots\hat{s}_{\beta_{k_{m-1}}}\cdots s_{\beta_\ell}\lessdot\cdots\\
\cdots &\lessdot s_{\beta_1}\cdots\hat{s}_{\beta_{k_1}}\cdots s_{\beta_\ell}\lessdot s_{\beta_1}\cdots s_{\beta_\ell}=w
\end{align*}
where in the $i$th expression in the last displayed equation the sequence $k_1<\ldots<k_{m-i+1}$ is removed from the reduced expression of $w$ and where $i$ runs through $1,2,\ldots,m+1$. 

\end{cor}

\begin{proof}

This follows directly from Theorem~\ref{thm:one-property},~\ref{thm:comb}.
\end{proof}

\todo[inline,color=green]{The sequence $j_i$ may certainly be different from $k_i$ (e.g.\ in $w=s_2 s_1 s_2$ and $v=s_2$ in $\mathsf{A}_2$). Remark on $r_{w/v}\colon R(w)\to R(v)$.}

\begin{ex}

Let $v,w\in W$ be such that $v\leq w$. Let $s_{\beta_1}\cdots s_{\beta_\ell}$ be a reduced expression of $w$. Let $m=\ell(v,w)$ for brevity. The sequences $1\leq j_1<\cdots<j_m\leq\ell$ and $1\leq k_1<\cdots<k_m\leq\ell$ associated to this situation as in the statement of Corollary~\ref{cor:comb} will in general not coincide. Indeed, we give an examples for this behavior. For this example, let $R$ be of type $\mathsf{A}_2$. Let $\beta_1$ and $\beta_2$ be the simple roots of $R$. Let $v=s_{\beta_2}$ and $w=s_{\beta_2}s_{\beta_1}s_{\beta_2}$. We then have $\ell=3$ and $m=2$. For the specific reduced expression of $w$ written in the second last sentence, the sequences $1\leq j_1<j_2\leq 3$ and $1\leq k_1<k_2\leq 3$ as in Corollary~\ref{cor:comb} are easily determined to be $(2,3)$ and $(1,2)$. In particular, these two sequences are essentially different.

\end{ex}

\begin{notation}

Let $w\in W$. We denote by $R(w)$ the set of all reduced expressions of $w$. We denote by $r(w)$ the cardinality of $R(w)$, i.e.\ the number of reduced expressions of $w$.

\end{notation}

\begin{defn}

Let $v,w\in W$ be such that $v\leq w$. Then we define a map $r_{w/v}\colon R(w)\to R(v)$ which associates to a reduced expression $s_{\beta_1}\cdots s_{\beta_\ell}$ of $w$ the reduced expression 
\[
s_{\beta_1}\cdots\hat{s}_{\beta_{j_1}}\cdots\hat{s}_{\beta_{j_m}}\cdots s_{\beta_\ell}
\] 
of $v$ where $1\leq j_1<\cdots<j_m\leq\ell$ is the unique sequence which satisfies the condition in the first displayed equation of Corollary~\ref{cor:comb}. Moreover, we define a map $r_{w/v}^\circ\colon R(w)\to R(v)$ which associates to a reduced expression $s_{\beta_1}\cdots s_{\beta_\ell}$ of $w$ the reduced expression 
\[
s_{\beta_1}\cdots\hat{s}_{\beta_{k_1}}\cdots\hat{s}_{\beta_{k_m}}\cdots s_{\beta_\ell}
\] 
of $v$ where $1\leq k_1<\cdots<k_m\leq\ell$ is the unique sequence which satisfies the condition in the second displayed equation of Corollary~\ref{cor:comb}.

\end{defn}

\todo[inline,color=green]{Here comes a commutative diagram which links $r_{w/v}$ to $r_{w^{-1}/v^{-1}}^\circ$ via the involutive inversions/bijections $R(w)\cong R(w^{-1})$. In particular, $r(w)=r(w^{-1})$. Just as an explanatory step for the next example. By the way, it is obvious that $r(w)\leq r(w_o)$ for all $w\in W$.}

\begin{rem}
\label{rem:dia}

Let $v,w\in W$ be such that $v\leq w$. Then we have a commutative diagram of sets
\[
\begin{tikzcd} 
R(w)\arrow{r}{r_{w/v}}\arrow{d}[verticalsouth]{\sim} & R(v)\arrow{d}[verticalnorth]{\sim}\\ 
R(w^{-1})\arrow{r}{r_{w^{-1}/v^{-1}}^\circ} & R(v^{-1}) 
\end{tikzcd}
\]
where the vertical arrows indicate the involutive bijections which send a reduced expression $s_{\beta_1}\cdots s_{\beta_\ell}$ to the reduced expression $s_{\beta_\ell}\cdots s_{\beta_1}$ of its inverse.

\end{rem}

\begin{ex}

For this example, let $R$ be of type $\mathsf{A}_5$. Let $\beta_1,\beta_2,\beta_3,\beta_4,\beta_5$ be the simple roots with the labeling as in \cite[Plate~I]{bourbaki_roots}. Let $v=s_{\beta_1}s_{\beta_3}$ and let $w=s_{\beta_1}s_{\beta_2}s_{\beta_3}s_{\beta_5}$. Then, we have $v\leq w$, $r(v)<r(w)$, and the map $r_{w/v}$ is neither injective nor surjective. Consequently, by Remark~\ref{rem:dia}, the map $r_{w^{-1}/v^{-1}}^\circ$ is also neither injective nor surjective.

\end{ex}

Similarly as above, one can also give examples of $v,w\in W$ such that $v\leq w$, $r(v)>r(w)$, and such that $r_{w/v}$ is neither injective nor surjective.

\begin{question}

If $v,w\in W$ such that $v\leq w$ and such that $r(v)=r(w)$, does this imply that $r_{w/v}$ is a bijection?

\end{question}

\todo[inline,color=green]{{\bf A bit of an unmotivated and unrelated problem:} If $v,w\in W$ such that $v\leq w$ and $r(v)=r(w)$, does this imply that $r_{w/v}\colon R(w)\to R(v)$ is a bijection? I cannot find a counterexample for the moment.}

\section{Interpretation of the results in classical terms\except{toc}{:
passage from braided right partial derivatives to divided difference operators}}
\label{sec:passage}

In this section, we relate the results in $\mathscr{B}_W$ obtained thus far to the classical situation presented in the introduction. To this end, we use the embedding of $S_W$ into $\mathscr{B}_W$ as in \cite[Section~5]{bazlov1}. To make the desired link, it then suffices to restrict the opposite braided left partial derivatives and left skew partial derivatives to the image of $S_W$ in $\mathscr{B}_W$.


By \cite[Section~5]{bazlov1}, there exists an injective morphism $S_W\hookrightarrow\mathscr{B}_W$ of $\mathbb{Z}_{\geq 0}$-graded algebras in the category of $W$-modules which is uniquely determined by the assignment
\[
x\mapsto 2\sum_{\alpha\in R^+}B(x,\alpha)x_\alpha
\]
where $x\in\mathfrak{h}$, and extended linearly and multiplicatively (cf.\ Remark~\ref{rem:coinv-ident}). We denote the image of this 
morphism by $\tilde{S}_W$. Note that $\tilde{S}_W$ is in general not a $W$-graded vector subspace of $\mathscr{B}_W$. Consequently, $\tilde{S}_W$ is in general not a Yetter-Drinfeld $W$-submodule of $\mathscr{B}_W$, and $S_W$ has in general no $W$-grading and no structure of Yetter-Drinfeld $W$-module inherited from $\mathscr{B}_W$ via $\tilde{S}_W$. To simplify notation, we will from now on identify $S_W$ with $\tilde{S}_W$ via the morphism as explained above.

By \cite[Lemma~5.9]{bazlov1}, the left action of the algebra $\mathscr{B}_W$ on $\mathscr{B}_W$ via opposite braided left partial derivatives restricts to a left action of the algebra $\mathscr{B}_W$ on $S_W$ which is given on generators by sending an element $x_\alpha$ where $\alpha\in R^+$ to the $\mathbb{C}$-linear endomorphism
\[
\overrightarrow{D}_\alpha^\circ\big|_{S_W}=\partial_\alpha
\]
of $S_W$, and extended linearly and multiplicatively. Consequently, it follows immediately that $\partial_\beta$ where $\beta\in\Delta$ satisfy the nilCoxeter relations (because $x_\beta$ where $\beta\in\Delta$ do), that $\partial_w$ where $w\in W$ is well-defined, and that 
\[
\overrightarrow{D}_w^\circ\big|_{S_W}=\partial_w
\]
for all $w\in W$. By Corollary~\ref{cor:independent-Dwvcirc}, we also see that Definition~\ref{def:partialwv} is well-defined, and that 
\[
\overrightarrow{D}_{w/v}^\circ\big|_{S_W}=\partial_{w/v}
\]
for all $v,w\in W$. (To see this, just note that $\overrightarrow{\varphi}^{\bm{\beta}}_J\big|_{S_W}=\partial^{\bm{\beta}}_J$ for each sequence of simple roots $\bm{\beta}$ corresponding to a fixed reduced expression of $w$ and for each subset $J$ of $\{1,\ldots,\ell\}$ where $\ell=\ell(w)$.)

All claims made in Subsection~\ref{subsec:partial} and~\ref{subsec:summary} now follow from the observations in the previous paragraph. Indeed, one simply has to restrict the action induced by the statements written in brackets after the claims in the introductory subsections from $\mathscr{B}_W$ to $S_W$, and has to use the three displayed formulas worked out above.

\appendix

\section{Shuffle elements and Bruhat intervals of length two}
\label{appendix:shuffle}

In this appendix, we investigate the combinatorics of shuffle elements and their relation to Bruhat intervals of length two. We call an element of a Coxeter group a shuffle element if there exists a unique ascent in the weak left Bruhat order (cf. Definition~\ref{def:shuffle}). Our main theorem on shuffle elements (Theorem~\ref{thm:shuffle}) says something general about some second successor of a shuffle element in the Bruhat order. For us, this notion is useful because under certain conditions on a Bruhat interval of length two its starting point tends to be a shuffle element. The conditions we impose as well as the final formulation of the main theorem on Bruhat intervals of length two (Theorem~\ref{thm:main-interval2}) are motivated by the application in the proof of the monomial property of order two (Theorem~\ref{thm:monomial-two}) in the main corpus of this text.

The techniques we use in our proofs in this appendix are very often restricted to simply laced Weyl groups. For example, the most nontrivial input, the generalized lifting property (\cite[Theorem~6.3]{GLP}) holds for a Coxeter system if and only if it is finite and simply laced (\cite[Theorem~6.9]{GLP}). Indeed, we illustrate in several examples that the main theorems of this appendix 
fail for rather trivial reasons for non simply laced Weyl groups. 
Finally, it should be said that one can get a rather cheap proof of Theorem~\ref{thm:shuffle},~\ref{thm:main-interval2}
restricted to the symmetric group. However, for the exceptional types, all presented methods are needed.

In Subsection~\ref{subsec:glp},~\ref{subsec:weyl}, we recall basic notions around the generalized lifting property
and Weyl groups. A reader familiar with this setup can directly jump to Subsection~\ref{subsec:prelim-roots} bearing in mind that we are always treating onwards simply laced Weyl groups unless otherwise specified (cf. Convention~\ref{conv:simply-laced}).

\todo[inline,color=green]{Already in the introduction, I should remark that everything becomes trivial in type $\mathsf{A}$ (most of our considerations in Section~3,~4,~5?), and that [outdated:] a reader only interested in the applications on skew divided difference operators can go directly to the sections concerning those (make a reference to the remark after the main theorem on shuffle elements which explains why this is trivial [no reference, no such remark on the symmetric group]).}

\todo[inline,color=green]{In this introduction I want to explain why I assume $W$ to be a simply laced Weyl group and why the results are not working in greater generality (GLP + reference to counterexample [no reference]).} 

\todo[inline,color=green]{The appendix can be read by mathematicians only interested in the theory of Coxeter groups and in combinatorics of root systems independently from the rest of the text.}

\todo[inline,color=green]{We actually prove sharper statements than our conjecture for the simply laced case [outdated].}

\subsection{Notation around the generalized lifting property \texorpdfstring{{\normalfont(\cite[Section~2]{GLP})}}{([\ref{bib-GLP},~Section~2])}}
\label{subsec:glp}

Let $(W,S)$ be a Coxeter system with the notation attached to this Situation as in Subsection~\ref{subsec:coxeter}. As in \cite[Section~2]{GLP}, for $w\in W$, we introduce the following notation related to the generalized lifting property:
\begin{alignat*}{6}
A_\ell(w)&=\{\beta\in\;\Delta &&\mid w<s_\beta w\}=
\{\beta\in\;\Delta &&\mid w^{-1}(\beta)>0\}\,,\\
A(w)&=\{\alpha\in R^+&&\mid w<s_\alpha w\}=
\{\alpha\in R^+&&\mid w^{-1}(\alpha)>0\}\,,\\
D(w)&=\{\alpha\in R^+&&\mid w>s_\alpha w\}=
\{\alpha\in R^+&&\mid w^{-1}(\alpha)<0\}\,.
\end{alignat*}
The second equation in each line of the previous displayed equation follows from \cite[Proposition~5.7]{humphreys-coxeter}. Furthermore, for each pair $v,w\in W$, we define $AD(v,w)=A(v)\cap D(w)$.

\subsection{Weyl groups}
\label{subsec:weyl}

In this subsection, we fix our notation and conventions for Weyl groups. From now on and for the rest of the appendix,
all notation and conventions from this subsection will be in force.
 
Let $R$ be a reduced, irreducible and crystallographic root system in a real vector space $\mathfrak{h}_{\mathbb{R}}$.
Let $W$ be the subgroup of $\operatorname{Aut}_{\mathbb{R}}(\mathfrak{h}_{\mathbb{R}})$ generated by the reflections $s_\alpha$ where $\alpha$ runs through $R$. The group $W$ is necessarily finite because $R$ is finite. It is called the Weyl group associated to $R$. We fix once and for all a $W$-invariant scalar product $(-,-)$ on $\mathfrak{h}_{\mathbb{R}}$. It is well-known that such a scalar product is unique up to scalar multiple in $\mathbb{R}_{>0}$. Thanks to this scalar product and the axioms of $R$, for each pair of roots $\alpha$ and $\gamma$, we define an integer $\left<\gamma,\alpha\right>$ by the expression $\frac{2(\gamma,\alpha)}{(\alpha,\alpha)}$. This last expression is manifestly independent of the choice of the scalar product $(-,-)$. For each root $\alpha$, the reflection $s_\alpha$ now acts on $\gamma\in R$ via the familiar formula $s_\alpha(\gamma)=\gamma-\left<\gamma,\alpha\right>\alpha$.
We choose once and for all a base $\Delta$ of $R$ consisting of simple roots. The choice of a base $\Delta$ induces a partial order on $R$ which is denoted by \enquote{$\leq$}. The set of positive elements of $R$ with respect to this partial order is denoted by $R^+$. The elements of $R^+$ are called positive roots.

Let $S=\{s_\beta\mid\beta\in\Delta\}$. Then $(W,S)$ becomes an irreducible Coxeter system such that $W$ is a finite Coxeter group. We now make $\Delta$ correspond to $S$ via the bijection $\Delta\overset{\sim}{\to}S,\beta\mapsto s_\beta$ and the real vector space $\mathfrak{h}_{\mathbb{R}}$ from this subsection correspond to the identically named in Subsection~\ref{subsec:coxeter}. In this way, we can well speak about the Coxeter integers $m(s,s')$, $m(\beta,\beta')$ and the scalar product $B$ on $\mathfrak{h}_{\mathbb{R}}$ (where $s,s'\in S$ and $\beta,\beta'\in\Delta$).

\begin{rem}
\label{rem:weyl-coxeter}

We are now in the situation that we have two actions of $W$ on $\mathfrak{h}_{\mathbb{R}}$, one via the inclusion $W\subseteq\operatorname{Aut}_{\mathbb{R}}(\mathfrak{h}_{\mathbb{R}})$, which we call the natural action of the Weyl group $W$, and one given by the geometric representation of $W$ considered as a Coxeter group. In general, these two actions are essentially different. Consequently, the scalar product $B$ will be in general not invariant under the natural $W$-action but only under the action induced by the geometric representation of $W$, and $B$ will be in general not a scalar multiple in $\mathbb{R}_{>0}$ of $(-,-)$. Furthermore, the root system and its positive roots introduced in Subsection~\ref{subsec:coxeter} will be in general different from the root system and its positive roots introduced in this subsection.

\end{rem}

\begin{conv}
\label{conv:weyl-coxeter}

To avoid confusion we make clear: From now on, whenever we work with a Weyl group $W$, we exclusively consider the natural action of $W$ on $\mathfrak{h}_{\mathbb{R}}$ in the sense of Remark~\ref{rem:weyl-coxeter} as well as the root system and its positive roots as introduced in this subsection. Since we will be in this appendix mostly concerned with the case where the natural action coincides with the one induced by the geometric representation and where consequently all other notions coincide, i.e.\ with simply laced Weyl groups,
we do not have to bother much about the difference between the two perspectives (-- the Weyl group and the Coxeter group perspective).

\end{conv}

For two simple roots $\beta$ and $\beta'$, the integer $\left<\beta,\beta'\right>$ is called a Cartan integer. We recall the following relation between the Cartan and the Coxeter integers. By \cite[Chapter~VI,~\S~1,~n\textsuperscript{o}~5,~Equation~(11)]{bourbaki_roots}, we have
\begin{equation}
\label{eq:cartan-coxeter}
\left<\beta,\beta'\right>=2\,\frac{\left\|\beta\right\|}{\left\|\beta'\right\|}B(\beta,\beta')
\end{equation}
where $\beta,\beta'\in\Delta$ and where $\|{-}\|$ denotes the norm induced by the scalar product $(-,-)$.

\subsubsection{Notation for roots}
\label{subsubsec:roots}

Let $\alpha$ be a positive root and $\beta$ be a simple root. We denote by $n_\beta(\alpha)$ the coefficient of $\beta$ in the unique expression of $\alpha$ as a linear combination of simple roots. In other words, we have $\alpha=\sum_{\beta\in\Delta}n_\beta(\alpha)\beta$ where $n_\beta(\alpha)\in\mathbb{Z}_{\geq 0}$ for all $\beta\in\Delta$.

In the following, we will adopt the notation from \cite[Plate~I--IX]{bourbaki_roots} to depict positive roots in terms of their coefficients at simple roots. We briefly recall how this works. To determine a positive root $\alpha$, we print the Dynkin diagram of $R$ as an array where each node corresponding to a simple root $\beta$ is labeled with the nonnegative integer $n_\beta(\alpha)$. In explicit examples, whenever this matters, we also adopt from \cite[Plate~I--IX]{bourbaki_roots} the labeling of the simple roots $\Delta=\{\beta_1,\ldots,\beta_n\}$ where $n$ is the rank of $R$ and where $\beta_i$ corresponds to the $i$th node of the Dynkin diagram of $R$ (for all $1\leq i\leq n$). To explicitly make clear to which type of $R$ a positive root $\alpha$ belongs, we even illustrate nodes of the Dynkin diagram of $R$ corresponding to simple roots $\beta$ such that $n_\beta(\alpha)=0$.

For example, if we write that the positive root $\alpha$ is given by
\[
\begin{array}{cccc}
0 & 1 & 1 & 1
\end{array}
\]
then we mean that $\alpha=\beta_2+\beta_3+\beta_4$ where $R$ is of type $\mathsf{A}_4$ and where $\beta_1,\beta_2,\beta_3,\beta_4$ are the simple roots in $\Delta$ with the labeling as in \cite[Plate~I]{bourbaki_roots}. Whereas, if we write that the positive root $\alpha$ is given by
\[
\begin{array}{ccc}
1 & 1 & 1
\end{array}
\]
then we mean that $\alpha=\beta_1+\beta_2+\beta_3$ where $R$ is of type $\mathsf{A}_3$ and where $\beta_1,\beta_2,\beta_3$ are the simple roots in $\Delta$ with the labeling as in \cite[Plate~I]{bourbaki_roots}. This notation would of course give rise to ambiguity because abstracted from the double edges the arrays of the Dynkin diagrams, for example, of type $\mathsf{A}_3$ and $\mathsf{B}_3$ look identically. Yet, since we will only use this notation for simply laced root systems, 
all positive roots will be uniquely determined.

\begin{notation}

Let $\alpha\in R^+$. We denote by $\Delta(\alpha)$ the support of $\alpha$, i.e.\ the subset of $\Delta$ given by the equation
\[
\Delta(\alpha)=\{\beta\in\Delta\mid\beta\leq\alpha\}=\{\beta\in\Delta\mid n_\beta(\alpha)>0\}\,.
\]

\end{notation}

\todo[inline,color=green]{Base, positive roots, Cartan integers. The notation $\mathfrak{h}_{\mathbb{R}}$ is suggestive because this vector space really coincides with the one from the subsection on Coxeter groups. Other things (this concerns $R$ and $R^+$, and consequently \enquote{$\leq$}) will also coincide: $\Delta,R,R^+,\ldots$, in case of simply laced Weyl groups, but not in general, and if the choices are made correctly for the $B$-setting. $\Delta$ will always coincide.}

\todo[inline,color=green]{Note: $\left<-,-\right>$ becomes a $W$-invariant scalar product for simply laced Weyl groups. For Weyl group $W$, $(W,S)$ is even irreducible -- to be mentioned as adjective.}

\todo[inline,color=green]{Since the the \enquote{natural representation} and the geometric representation of a non simply laced Weyl group differ (they coincide if and only if the Weyl group is simply laced), the notion of positivity differs too, i.e.\ I cannot naively use the result: $w(\alpha)>0$ if and only if $ws_\alpha>w$, in general. I have to recheck the examples for non simply laced Weyl group. Do I somewhere use this result from \cite{humphreys-coxeter}?

{\bf Answer.} This problem can be partially solved since the analogous result also holds for reflection groups. Note that \cite[Proposition~5.7]{humphreys-coxeter} is formally derived from \cite[Theorem~5.4]{humphreys-coxeter}, and \cite[Lemma~1.6]{humphreys-coxeter} holds. Except this, I do not use the result \enquote{$w(\alpha)>0$ if and only if $ws_\alpha>w$} anywhere for non simply laced Weyl groups and the natural representation. The only situation where I use a non simply laced Weyl group is in Example~\ref{ex:shuffle1}, and there, I don't need anything. I can without doubt globally refer to the Coxeter group result in the introduction.}

\todo[inline,color=green]{Also, I better change the name from Coxeter number to Coxeter integer to avoid a conflict in the terminology (in the spirit of Cartan integer).}

\subsubsection{Simply laced Weyl groups}
\label{subsubsec:simply-laced}

Let $R$ be a simply laced root system, i.e.\ a root system such that the Dynkin diagram of $R$ has only simple edges. Let $W$ be the Weyl group associated to $R$. The group $W$ is called a simply laced Weyl group. All Coxeter integers of the Coxeter system $(W,S)$ where $S$ is associated to $\Delta$ as in Subsection~\ref{subsec:weyl} are $\leq 3$. With the terminology of \cite[Section~2]{GLP}, the Coxeter system $(W,S)$ therefore becomes an irreducible simply laced Coxeter system, and every irreducible simply laced Coxeter system arises in this way from a simply laced Weyl group.
Since all roots in $R$ have equal length, Equation~\eqref{eq:cartan-coxeter} becomes 
\begin{equation}
\label{eq:simply-laced-scalar}
\left<\beta,\beta'\right>=2B(\beta,\beta')
\end{equation}
for all $\beta,\beta'\in\Delta$. Hence, as anticipated in Convention~\ref{conv:weyl-coxeter}, the natural $W$-action on $\mathfrak{h}_{\mathbb{R}}$ and the one induced by the geometric representation of $W$ coincide. Consequently, all other notions introduced in Subsection~\ref{subsec:coxeter} and Subsection~\ref{subsec:weyl} correspond to each other.

\begin{rem}

Because of Equation~\eqref{eq:simply-laced-scalar}, we clearly see that for simply laced Weyl groups the form $\left<-,-\right>$ becomes a $W$-invariant scalar product. This is very specific to this situation. If there are two root length, i.e.\ if $R$ is non simply laced, then the form $\left<-,-\right>$ will be neither bilinear nor symmetric (but still $W$-invariant with respect to the natural $W$-action).

\end{rem}

\begin{conv}
\label{conv:simply-laced}

From now on and for the rest of the appendix, unless otherwise stated in some examples, $R$ will be always a simply laced root system and $W$ a simply laced Weyl group associated to it. We keep all the notation and conventions attached to this situation and introduced thus far.

\end{conv}

The following lemma is fundamental and will be often tacitly in use without referencing to it.

\begin{lem}
\label{lem:01}

Let $\alpha$ and $\gamma$ be non-proportional roots (i.e.\ $\gamma\neq\pm\alpha$). Then we have $\left<\gamma,\alpha\right>\in\{-1,0,1\}$.

\end{lem}

\begin{proof}

Since in our situation all roots are deemed to be long, the lemma is a trivial consequence of \cite[Lemma~7.14]{quasi-homogeneity}.
\end{proof}

\todo[inline,color=green]{$\leq,\lessdot,\Delta,R^+,W,s_\alpha,A_\ell,A,D,AD,\ell,\ell(-,-),R$,

$m(\beta,\beta')=$Coxeter integer of $\beta$ and $\beta'$ (cf. Example~\ref{ex:shuffle1}, here we have $\beta,\beta'\in\Delta$, maybe not standard terminology but we use it anyway),

order on roots, $\leq_\ell$.}

\todo[inline,color=green]{Notation not needed anywhere and hence not introduced: $R^-$, $\alpha^\vee$, $A_r$, $D_r$, $D_\ell$.}

\todo[inline,color=green]{Notation and terminology already introduced: $(-,-)$, $\left<-,-\right>$, $n_\beta(\alpha)$, $\Delta(\alpha)$, simply laced Weyl group, simply laced Coxeter group, simply laced root system.}

\todo[inline,color=green]{We often use \cite[Proposition~5.7,~5.2]{humphreys-coxeter} without always referring to it. (This applies to the whole paper and not only to the appendix. Hence, I should note it directly in the beginning, i.e.\ in Subsection~\ref{subsec:coxeter}.) $W$ always denotes a simply laced Weyl group. I will mention before Lemma~\ref{lem:01} that $\left<-,-\right>$ is commutative/symmetric and even bilinear (scalar product). I also use Lemma~\ref{lem:01} tacitly throughout the text. Mention this. Not all of the notation, $A_r$ for example, is needed. Reduce it to a minimum.}

\todo[inline,color=green]{It is probably useful to put this notational section before the introduction of Appendix~\ref{appendix:shuffle}, so that I can use already there freely the notation. For example, one can put this notation as a subsection in the main body of the text where I introduce my notation on arbitrary Coxeter groups. 

It seems that I use the more specific notation for ascent and descent $AD$ only in this appendix and that it is not needed in the main body of the text. So, I should only introduce it (this more specific notation) here in the beginning of the appendix.}

\subsection{Preliminaries on roots}
\label{subsec:prelim-roots}

In this subsection, we provide the necessary preliminaries on roots. 
Basically, we have to exclude the existence of a root which satisfies all of the properties listed in the items in Proposition~\ref{prop:absurd}. Unfortunately, we have no elegant proof of this. Our considerations in this subsection are based on the classification and explicit realization of root systems as in \cite{bourbaki_roots}. They rely on divers routine checks. We only indicate the most relevant ideas of those. An impatient reader can skip this subsection in good faith and take the non-existence of the before mentioned root as granted.  

\todo[inline,color=green]{\color{black} I should explain my notation for roots in the beginning of this section. We indicate always zeros to show in which type we are. In most cases, we have no good proofs. They rely on the classification of root systems and divers checks. We only indicate the most important ideas.

This seems to be a comment on the out-commented and outdated version of this subsection which became redundant in the final formulation of the proceeding:
I should comment on the nature of the two classification remarks. They involve extensive checking with the root tables and are not important for us. In fact, we chose a formulation of our proofs which exactly avoids these checks. We only mention them for the sake of completeness. Also, the reader can skim through this section since it is boring.}

\begin{defn}

Let $\alpha\in R^+$ and $\beta\in\Delta(\alpha)$. We introduce three sets of simple roots depending on $\alpha$ and $\beta$ which are relevant for our investigations:
\begin{align*}
N^+(\alpha,\beta)&=\{\beta'\in\Delta(\alpha)\mid(\beta,\beta')<0,(\alpha,\beta')>0\}\,,\\
N^-(\alpha,\beta)&=\{\beta'\in\Delta(\alpha)\mid(\beta,\beta')<0,(\alpha,\beta')\leq 0\}\,,\\
N(\alpha,\beta)&=\{\beta'\in\Delta(\alpha)\mid(\beta,\beta')<0\}\,.
\end{align*}
We clearly have $N(\alpha,\beta)=N^+(\alpha,\beta)\sqcup N^-(\alpha,\beta)$.

\end{defn}

\begin{lem}
\label{lem:equaltozero}

Let $\alpha\in R^+$. Let $\beta\in\Delta(\alpha)$ such that $(\alpha,\beta)<0$. Then we have for all $\beta'\in N^-(\alpha,\beta)$ that $(\alpha,\beta')=0$.

\end{lem}

\begin{proof}

Let $\alpha$ and $\beta$ be as in the statement. Let $\beta'\in N^-(\alpha,\beta)$. It is clear that $\beta+\beta'$ is a positive root. Hence, we have $\left<\alpha,\beta+\beta'\right>=\left<\alpha,\beta\right>=-1$. It follows that $(\alpha,\beta')=0$ as claimed.
\end{proof}

\begin{lem}
\label{lem:n<2n}

Let $\alpha\in R^+$ and $\beta\in\Delta(\alpha)$. Let $\beta'\in N^+(\alpha,\beta)$. Then we have $n_\beta(\alpha)<2n_{\beta'}(\alpha)$.

\end{lem}
 
\begin{proof}

Let $\alpha,\beta,\beta'$ be as in the statement. Then we have
$
0<\left<\alpha,\beta'\right>\leq 2n_{\beta'}(\alpha)-n_\beta(\alpha)
$
and hence $n_\beta(\alpha)<2n_{\beta'}(\alpha)$ as claimed.
\end{proof} 
 
\begin{prop}
\label{prop:existenceofalpha'}

Let $\alpha\in R^+$. Let $\beta\in\Delta(\alpha)$ such that $(\alpha,\beta)<0$. Suppose that $n_\beta(\alpha)=1$. Then there exists $\alpha'\in R^+$ such that $\alpha'\leq\alpha$, $(\alpha',\beta)>0$, $(\alpha,\alpha')=0$.

\end{prop}

\begin{proof}

Let $\alpha$ and $\beta$ be as in the statement. It is clear that there exists $\beta'\in\Delta(\alpha)$ such that $(\alpha,\beta')>0$ since otherwise we have $(\alpha,\alpha)<0$ which is contrary to the fact that $(-,-)$ is a scalar product. Hence, we can choose a path $\beta_1,\ldots,\beta_\ell$ of minimal length $\ell$ in the connected subdiagram of the Dynkin diagram with nodes $\Delta(\alpha)$ such that $\beta_1=\beta$ and such that $(\alpha,\beta_\ell)>0$. Since $(\alpha,\beta)<0$ by assumption, it is clear that $\ell>1$. By the minimality of $\ell$, it is moreover clear that $(\alpha,\beta_i)\leq 0$ for all $1\leq i<\ell$. By an argument very similar to the one in the proof of Lemma~\ref{lem:equaltozero}, we even know that $(\alpha,\beta_i)=0$ for all $1<i<\ell$. We can now define $\alpha'=\sum_{i=1}^\ell\beta_i$. It is clear that $\alpha'$ is a positive root. By what we said above, it is easy to see that $\alpha'$ satisfies all the requirements in the statement of the proposition.
\end{proof}

\begin{rem}

Let $\alpha$, $\beta$, $\alpha'$ be as in the statement of Proposition~\ref{prop:existenceofalpha'}. Then it is clear that $\alpha'$ even satisfies $\alpha'<\alpha$.

\end{rem}

\begin{fact}
\label{fact:classification-final}

Suppose that $R$ is of type $\mathsf{E}_6$ or $\mathsf{E}_7$. Let $\alpha\in R^+$. Let $\beta\in\Delta(\alpha)$ be such that $(\alpha,\beta)<0$. Suppose that $n_\beta(\alpha)>1$. Then $\beta$ is uniquely determined by $\alpha$. Moreover, the root $\alpha$ occurs in the following list:
\begin{align*}
&\mathsf{E}_6\colon\,&&\begin{array}{ccccc}
1 & {\color{green} 2} & {\color{blue} 2} & {\color{green} 2} & 1 \\
  &   & {\color{red} 1} &   &
\end{array}\,,\\
&\mathsf{E}_7\colon\,&&\begin{array}{cccccc}
1 & {\color{green} 2} & {\color{blue} 2} & {\color{green} 2} & 1 & 0 \\
  &   & {\color{red} 1} &   &   &
\end{array}\,,\,
\begin{array}{cccccc}
1 & {\color{green} 2} & {\color{blue} 2} & {\color{green} 2} & 1 & 1 \\
  &   & {\color{red} 1} &   &   &
\end{array}\,,\,
\begin{array}{cccccc}
1 & {\color{green} 2} & {\color{blue} 2} & {\color{red} 2} & 2 & 1 \\
  &   & {\color{red} 1} &   &   &
\end{array}\,,\\
&&&\begin{array}{cccccc}
1 & 2 & {\color{green} 3} & {\color{blue} 2} & {\color{green} 2} & 1 \\
  &   & 1 &   &   &
\end{array}\,,\,
\begin{array}{cccccc}
1 & 2 & {\color{red} 3} & {\color{blue} 2} & {\color{green} 2} & 1 \\
  &   & 2 &   &   &
\end{array}\,,\,
\begin{array}{cccccc}
1 & {\color{red} 2} & {\color{blue} 3} & {\color{green} 3} & 2 & 1 \\
  &   & {\color{green} 2} &   &   &
\end{array}\,,\\
&&&\begin{array}{cccccc}
{\color{red} 1} & {\color{blue} 2} & {\color{green} 4} & 3 & 2 & 1 \\
  &   & 2 &   &   &
\end{array}\,.
\end{align*}
In this picture, the blue colored numbers correspond each time to the simple root $\beta$. The green colored numbers correspond each time to simple roots in $N^+(\alpha,\beta)$. The red colored numbers correspond each time to simple roots in $N^-(\alpha,\beta)$.

\end{fact}

\begin{proof}

The proof of the fact works by going through the list of all roots in type $\mathsf{E}_6$ and $\mathsf{E}_7$ which have a coefficient $>1$ (cf. \cite[Plate~V~and~VI]{bourbaki_roots}). By checking for each simple root $\beta$ having a coefficient $>1$ in $\alpha$ whether the pairing $(\alpha,\beta)$ is negative, we can create the table as above. A posteriori, we see from this table that $\beta$ is uniquely determined by $\alpha$.
\end{proof}

\begin{rem}

The statement in Fact~\ref{fact:classification-final} that $\beta$ is uniquely determined by $\alpha$ fails for type $\mathsf{E}_8$. In general, for a given $\alpha\in R^+$, there might be several $\beta\in\Delta(\alpha)$ such that $(\alpha,\beta)<0$ and such that $n_\beta(\alpha)>1$. This occurs for precisely four roots in type $\mathsf{E}_8$ for which there exist in each case precisely two $\beta$'s.

\end{rem}

\begin{cor}
\label{cor:classification1}

Suppose that $R$ is of type $\mathsf{E}_6$ or $\mathsf{E}_7$. Let $\alpha\in R^+$. Let $\beta\in\Delta(\alpha)$ be such that $(\alpha,\beta)<0$. Suppose that $n_\beta(\alpha)>1$. Then there exists $\beta'\in N^+(\alpha,\beta)$ such that $n_\beta(\alpha)\leq n_{\beta'}(\alpha)$.

\end{cor}

\begin{proof}

This follows directly from the list of roots we have created in Fact~\ref{fact:classification-final}. Graphically, we see that for each blue number in the list there exists a neighboring green number which is greater or equal than the blue number.
\end{proof}

\begin{cor}
\label{cor:classification2}

Suppose that $R$ is of type $\mathsf{E}_6$ or $\mathsf{E}_7$. Let $\alpha\in R^+$. Let $\beta\in\Delta(\alpha)$ be such that $(\alpha,\beta)<0$. Suppose that $n_\beta(\alpha)>1$. Then there exists $\alpha'\in R^+$ such that $n_\beta(\alpha)\alpha'\leq\alpha$, $(\alpha',\beta)>0$, $(\alpha,\alpha')=0$.

\end{cor}

\begin{proof}

Let $\alpha$ and $\beta$ be as in the statement. By Corollary~\ref{cor:classification1}, there exists $\beta'\in N^+(\alpha,\beta)$ such that $n_\beta(\alpha)\leq n_{\beta'}(\alpha)$. Let $\alpha'=\beta+\beta'$. It is clear that $\alpha'$ is a positive root. It is easy to see that $\alpha'$ satisfies all the requirements in the statement of the corollary.
\end{proof}

\begin{rem}

Let $\alpha$, $\beta$, $\alpha'$ be as in the statement of Corollary~\ref{cor:classification2}. Then it is clear that $\alpha'$ even satisfies $n_\beta(\alpha)\alpha'<\alpha$.

\end{rem}

\begin{rem}

Let $\alpha$, $\beta$, $\beta'$ be as in the statement of Corollary~\ref{cor:classification1}. Then we see from the list of roots in Fact~\ref{fact:classification-final} that $\beta'$ is not uniquely determined by $\alpha$ and $\beta$ (or just by $\alpha$). 
Indeed, graphically, there are blue numbers in the list such that there exists more than one neighboring green number which is greater or equal than the blue number. Consequently, by the proof of Corollary~\ref{cor:classification2}, the root $\alpha'$ in the statement of Corollary~\ref{cor:classification2} is also not uniquely determined by $\alpha$ and $\beta$ (or just by $\alpha$).

\end{rem}

\begin{rem}
\label{rem:one-exception}

Corollary~\ref{cor:classification2} (and consequently Corollary~\ref{cor:classification1}) actually fails in type $\mathsf{E}_8$. More precisely, there exists a unique exception to Corollary~\ref{cor:classification2}. We can say the following.
Let $\alpha\in R^+$. Let $\beta\in\Delta(\alpha)$ be such that $(\alpha,\beta)<0$. Suppose that $n_\beta(\alpha)>1$. Suppose further that there exists no $\alpha'\in R^+$ such that $n_\beta(\alpha)\alpha'\leq\alpha$, $(\alpha',\beta)>0$, $(\alpha,\alpha')=0$ (or equivalently no $\alpha'\in R^+$ such that $n_\beta(\alpha)\alpha'<\alpha$, $(\alpha',\beta)>0$, $(\alpha,\alpha')=0$). Then $\alpha$ and $\beta$ are uniquely determined. Explicitly, the root $\alpha$ is given by
$$
\begin{array}{ccccccc}
2 & 4 & {\color{blue} 5} & 4 & 3 & 2 & 1\\
  &   & 3 &   &   &   &
\end{array}
$$
where the blue colored number corresponds to the simple root $\beta$. 


\end{rem}

\subsection{Shuffle elements}

In this subsection, we introduce shuffle elements and prove the main theorem on them (Definition~\ref{def:shuffle}, Theorem~\ref{thm:shuffle}). Although we need the generalized lifting property (\cite[Theorem~6.3]{GLP}) only once in the proof of Theorem~\ref{thm:shuffle}, it is the main motivation for us to study minimal elements of $AD(v,w)$ where $v$ is a shuffle element and $w$ is some second successor of $v$ in the Bruhat order.

\todo[inline,color=green]{We don't really need the results of \cite{GLP}. [We need them once in the proof of Theorem~\ref{thm:shuffle}.] But our motivations for looking for minimal elements of $AD$ in the simply laced case goes back to them, so cite it with explanation [see introduction to this subsection above].} 

\todo[inline,color=green]{After the main theorem on intervals (next section), mention the two examples (motivation for one example: skew divided difference operators of length difference three are not monomials) and the triviality of type $\mathsf{A}$. (Note that for our applications this trivial statement actually suffices.) Or better even the triviality after the main theorem on shuffle elements. As a conjecture, once the result of Liu is generalized to all simply laced Fomin-Kirillov algebras, one can state that the skew difference operators of length difference two are either zero or again monomials (prospective research projects), basically because the theorem on Bruhat intervals of length two holds for all simply laced Weyl groups. [Outdated/realized. I mention the triviality in type $\mathsf{A}$ directly in the introduction to Appendix~\ref{appendix:shuffle}. I give ample examples (for length difference $>2$ and non simply laced) after both shuffle and interval theorem.]}

\todo[inline,color=green]{Later on, I should try if the antipode maps carry skew divided difference operators to skew divided difference operators of the same degree (in the ordinary $\mathbb{Z}$-sense first, but maybe even in the group sense, $W$-grading). [This is indeed true: The bar antipode $\bar{\mathscr{S}}$ carries $x_{w/v}$ to $x_{vw_o/ww_o}$ (cf.\ Theorem~\ref{thm:positivity}); $\mathbb{Z}_{\geq 0}$-degree preserved, $W$-degree inverted.]}

\todo[inline,color=green]{See notes in the math diary on the fifth of December 2017: I actually believe that the roots which occur in the monomial are uniquely determined (up to ordering) by the interval of length $\leq 2$.

Since the generalized lifting property holds precisely for simply laced Weyl groups or finite simply laced Coxeter groups, it is impossible to generalize the results -- and this is our main motivation to study only those groups in this appendix at least.}

\begin{defn}
\label{def:shuffle}


Let $v\in W$. We call $v$ a shuffle element if the set $A_\ell(v)$ consists of a unique element. If $v\in W$ is a shuffle element, we call the unique element of $A_\ell(v)$ the pivot element of $v$.

\end{defn}

\begin{rem}

Definition~\ref{def:shuffle} makes of course sense for an arbitrary Coxeter group although we assume everywhere $W$ to be a simply laced Weyl group since the rest of our theory only works in this situation.

\end{rem}

\begin{rem}


The terminology of Definition~\ref{def:shuffle} comes from the intuition for the case of the symmetric group. Indeed, for this remark, let $R$ be of type $\mathsf{A}_n$ where $n\in\mathbb{Z}_{>0}$. As usual, we identify the group $W$ with permutations of the set $\{1,\ldots,m\}$ where $m=n+1$. Then $v\in W$ is a shuffle element if and only if $v\neq w_o$ and if $m,m-1,\ldots,\mu+1$ and $\mu,\mu-1,\ldots,1$ are subsequences of $v(1),\ldots,v(m)$ for some $\mu\in\{1,\ldots,n\}$. In that case, the integer $\mu$ is uniquely determined by $v$ and the pivot element of $v$ is given by $\beta_\mu$ where $\beta_\mu$ is the simple root corresponding to the $\mu$th node of the Dynkin diagram of $R$ with labeling of the nodes from left to right as in \cite[Plate~I]{bourbaki_roots}.

\end{rem}

\begin{thm}

Let $v\in W$ be a shuffle element with pivot element $\beta$. Let $\alpha$ be a minimal element of $A(v)\setminus\{\beta\}$. Then we have $(\alpha,\beta)>0$.

\end{thm}

\begin{proof}

Let $v,\beta,\alpha$ be as in the statement. Suppose for a contradiction that $(\alpha,\beta)\leq 0$. Since $(\alpha,\alpha)>0$, there must exist $\beta'\in\Delta\setminus\{\beta\}$ such that $(\alpha,\beta')>0$. Since $\alpha\notin\Delta$, we clearly have $\left<\alpha,\beta'\right>=1$. 
Hence, $s_{\beta'}(\alpha)=\alpha-\beta'\in R^+$. Since $\alpha\in A(v)$ and $\beta'\in D(v)$, it follows that $v^{-1}(\alpha-\beta')>0$ and thus $\alpha-\beta'\in A(v)$. By the minimality of $\alpha$, we must have $\alpha=\beta+\beta'$. But this implies that $\left<\alpha,\beta\right>=2-1=1$. In particular, $(\alpha,\beta)>0$ -- a contradiction.
\end{proof}

\begin{prop}
\label{prop:absurd}

Let $v\in W$ be a shuffle element with pivot element $\beta$. Let $\alpha$ be a minimal element of $AD(v,s_\alpha s_\beta v)$. Then we have:

\begin{enumerate}

\item
\label{item:shuffle1}

$\beta\in\Delta(\alpha)$.

\item 
\label{item:shuffle2}

$(\alpha,\beta)<0$.

\item
\label{item:shuffle3}

$n_\beta(\alpha)>\sum_{\beta'\in N^+(\alpha,\beta)}n_{\beta'}(\alpha)$.

\item
\label{item:shuffle4}

There exists no $\alpha'\in R^+$ such that $n_\beta(\alpha)\alpha'\leq\alpha$, $(\alpha',\beta)>0$, $(\alpha,\alpha')=0$.

\item
\label{item:shuffle5}

$n_\beta(\alpha)>1$.

\end{enumerate}

\end{prop}

\begin{proof}

Let $v,\beta,\alpha$ be as in the statement.
Ad~Item~\eqref{item:shuffle1}. Suppose that $\beta\notin\Delta(\alpha)$. Since $v$ is a shuffle element, it follows that $v^{-1}(\alpha)<0$ -- contrary to the assumption on $\alpha$ that $\alpha\in A(v)$.

Ad~Item~\eqref{item:shuffle2}. We first show that $\alpha\neq\beta$. Indeed, suppose that $\alpha=\beta$. Then we have $AD(v,s_\alpha s_\beta v)=AD(v,v)=\emptyset$ which is contrary to the existence of $\alpha$. Since $\alpha\neq\beta$, we clearly have $\left<\alpha,\beta\right>\in\{-1,0,1\}$. Suppose now that $(\alpha,\beta)\geq 0$. A short computation shows that
$$
v^{-1}s_\beta s_\alpha(\beta)=(\left<\alpha,\beta\right>^2-1)v^{-1}(\beta)-\left<\beta,\alpha\right>v^{-1}(\alpha)\,.
$$
Since $v^{-1}(\alpha)>0$, $v^{-1}(\beta)>0$, $(\alpha,\beta)\geq 0$, $\left<\alpha,\beta\right>^2-1\leq 0$, it follows immediately from the previous line that $v^{-1}s_\beta s_\alpha(\beta)<0$. Hence, we have $\beta\in AD(v,s_\alpha s_\beta v)$. By the minimality of $\alpha$ and Item~\eqref{item:shuffle1}, it follows that $\alpha=\beta$. But we already excluded this in the beginning of the proof of the item. 

Ad~Item~\eqref{item:shuffle3}. Let $\beta'\in N^+(\alpha,\beta)$. We first show that we then have $v^{-1}(\beta')<-v^{-1}(\beta)$. Indeed, let $\gamma=s_{\beta'}(\alpha)=\alpha-\beta'\in R^+$. We clearly have $v^{-1}(\gamma)>0$. By the minimality of $\alpha$, it consequently follows that $v^{-1}s_\beta s_\alpha(\gamma)>0$. But a short computation shows that $v^{-1}s_\beta s_\alpha(\gamma)=-v^{-1}(\beta)-v^{-1}(\beta')$. Hence, it indeed follows that $v^{-1}(\beta')<-v^{-1}(\beta)$. We now can conclude by considering $v^{-1}(\alpha)>0$. We have
\begin{align*}
0&<v^{-1}(\alpha)\\
&=n_\beta(\alpha)v^{-1}(\beta)+\sum_{\beta'\in N^+(\alpha,\beta)}n_{\beta'}(\alpha)v^{-1}(\beta')+\sum_{\mu\in\Delta(\alpha)\setminus(\{\beta\}\cup N^+(\alpha,\beta))}n_\mu(\alpha)v^{-1}(\mu)\\
&\leq n_\beta(\alpha)v^{-1}(\beta)+\sum_{\beta'\in N^+(\alpha,\beta)}n_{\beta'}(\alpha)v^{-1}(\beta')\\
&\leq\left(n_\beta(\alpha)-\sum_{\beta'\in N^+(\alpha,\beta)}n_{\beta'}(\alpha)\right)v^{-1}(\beta)\,.
\end{align*}
Since $v^{-1}(\beta)>0$, it is now clear that we must have $n_\beta(\alpha)>\sum_{\beta'\in N^+(\alpha,\beta)}n_{\beta'}(\alpha)$ -- as claimed. 

Ad~Item~\eqref{item:shuffle4}. Suppose there exists $\alpha'\in R^+$ such that $n_\beta(\alpha)\alpha'\leq\alpha$, $(\alpha',\beta)>0$, $(\alpha,\alpha')=0$. First note that the inequality $n_\beta(\alpha)\alpha'\leq\alpha$ makes clear that $n_\beta(\alpha')=1$. Next, we show that we even have $n_\beta(\alpha)\alpha'<\alpha$. Indeed, suppose that $n_\beta(\alpha)\alpha'=\alpha$. This implies that $n_\beta(\alpha)=1$ and $\alpha=\alpha'$. The inequality $(\alpha',\beta)>0$ then means $(\alpha,\beta)>0$ which is contrary to Item~\eqref{item:shuffle2}. The inequality $n_\beta(\alpha)\alpha'<\alpha$ means in particular that $\alpha'<\alpha$. By the minimality of $\alpha$, this means that $\alpha'\notin AD(v,s_\alpha s_\beta v)$. We now show the contrary which leads to the desired contradiction. Indeed, we have $0<v^{-1}(\alpha)<n_\beta(\alpha)v^{-1}(\alpha')$ since $\alpha$ and $n_\beta(\alpha)\alpha'$ have the same coefficient at $\beta$. It follows that $v^{-1}(\alpha')>0$. On the other hand, in view of the assumptions on $\alpha'$, we compute that $v^{-1}s_\beta s_\alpha(\alpha')=v^{-1}(\alpha'-\beta)$. Since $n_\beta(\alpha')=1$, we know that $\alpha'-\beta$ has zero coefficient at $\beta$. It follows that $v^{-1}(\alpha'-\beta)<0$ and thus $v^{-1}s_\beta s_\alpha(\alpha')<0$ and thus $\alpha'\in AD(v,s_\alpha s_\beta v)$ -- a contradiction.

Ad~Item~\eqref{item:shuffle5}. This item follows from Item~\eqref{item:shuffle1},\eqref{item:shuffle2},\eqref{item:shuffle4} and Proposition~\ref{prop:existenceofalpha'}.
\end{proof}

\begin{rem}

Let $v,\beta,\alpha$ be as in the statement of Proposition~\ref{prop:absurd}. Then we note that Item~\eqref{item:shuffle5} implies Item~\eqref{item:shuffle1}. Moreover, if $N^+(\alpha,\beta)=\emptyset$, then Item~\eqref{item:shuffle1} and \eqref{item:shuffle3} are equivalent. Also, Item~\eqref{item:shuffle4} can be equivalently formulated by saying that there exists no $\alpha'\in R^+$ such that $n_\beta(\alpha)\alpha'<\alpha$, $(\alpha',\beta)>0$, $(\alpha,\alpha')=0$. This last sentence follows from Item~\eqref{item:shuffle2}.

\end{rem}

\begin{thm}
\label{thm:absurd}

Let $v\in W$ be a shuffle element with pivot element $\beta$. Then there exists no $\alpha\in R^+$ such that $\alpha$ is a minimal element of $AD(v,s_\alpha s_\beta v)$.

\end{thm}

\begin{proof}


Let $v$ and $\beta$ be as in the statement. We argue by contradiction and distinguish the types of $R$. Assume that there exists $\alpha\in R^+$ such that $\alpha$ is a minimal element of $AD(v,s_\alpha s_\beta v)$. By Proposition~\ref{prop:absurd}\eqref{item:shuffle2},\eqref{item:shuffle5}, it is clear that $\alpha+\beta$ is a positive root which has a coefficient $>2$ at $\beta$. Since all roots in type $\mathsf{A}$ and $\mathsf{D}$ have coefficients $\leq 2$ at all simple roots, those two types are ruled out directly from the beginning. Assume now that $R$ is of type $\mathsf{E}_6$ or $\mathsf{E}_7$. Then Corollary~\ref{cor:classification1} contradicts Proposition~\ref{prop:absurd}\eqref{item:shuffle3} and Corollary~\ref{cor:classification2} contradicts Proposition~\ref{prop:absurd}\eqref{item:shuffle4} (both corollaries apply because of Proposition~\ref{prop:absurd}\eqref{item:shuffle2},\eqref{item:shuffle5}). Assume finally that $R$ is of type $\mathsf{E}_8$. In view of Proposition~\ref{prop:absurd}\eqref{item:shuffle2},\eqref{item:shuffle4},\eqref{item:shuffle5}, Remark~\ref{rem:one-exception} implies that $\alpha$ and $\beta$ are uniquely determined and given as in the statement of Remark~\ref{rem:one-exception}. From the explicit description of $\alpha$ and $\beta$, we see that $n_\beta(\alpha)=5$ and that $\sum_{\beta'\in N^+(\alpha,\beta)}n_{\beta'}(\alpha)=3+4=7$. This contradicts eventually Proposition~\ref{prop:absurd}\eqref{item:shuffle3}.
%
%
%
\end{proof}

\begin{lem}
\label{lem:preparation-shuffle}

Let $v\in W$. Let $\beta\in A_\ell(v)$. Let $\alpha\in A(s_\beta v)$. Suppose that $(\alpha,\beta)\geq 0$. Then we have $\alpha\in A(v)$ and $\beta\in D(s_\alpha s_\beta v)$.

\end{lem}

\begin{proof}

Let $v,\beta,\alpha$ be as in the statement. By assumption, we have $v^{-1}s_\beta(\alpha)>0$. Expanding the expression, we find $v^{-1}(\alpha)>\left<\alpha,\beta\right>v^{-1}(\beta)\geq 0$ and hence $\alpha\in A(v)$. On the other hand, a computation similar to the one in the proof of Proposition~\ref{prop:absurd}\eqref{item:shuffle2} shows that
$$
v^{-1}s_\beta s_\alpha(\beta)=(\left<\alpha,\beta\right>^2-1)v^{-1}(\beta)-\left<\beta,\alpha\right>v^{-1}(\alpha)\,.
$$
Since $\alpha\in A(s_\beta v)$ by assumption, we clearly have $\alpha\neq\beta$ and thus $\left<\alpha,\beta\right>\in\{0,1\}$. The first statement of the lemma and the assumptions read as $v^{-1}(\alpha)>0, v^{-1}(\beta)>0, (\alpha,\beta)\geq 0, \left<\alpha,\beta\right>^2-1\leq 0$. It follow immediately from this and the previous displayed equation that $v^{-1}s_\beta s_\alpha(\beta)<0$ and thus $\beta\in D(s_\alpha s_\beta v)$.
\end{proof}

\begin{thm}[Main theorem on shuffle elements]
\label{thm:shuffle}

Let $v\in W$ be a shuffle element with pivot element $\beta$. Let $w\in W$ be such that $s_\beta v\lessdot w$. Then we have $s_\beta w<w$.

\end{thm}

\begin{proof}

Let $v,\beta,w$ as in the statement. Let $\alpha\in R^+$ be such that $w=s_\alpha s_\beta v$. Suppose for a contradiction that $s_\beta w>w$. Then Lemma~\ref{lem:preparation-shuffle} implies that $(\alpha,\beta)<0$. Since $\alpha\in A(s_\beta v)$ by definition, it is clear that $\alpha\neq\beta$ and thus $\left<\alpha,\beta\right>=-1$. By our assumption $s_\beta w>w$, we have $w^{-1}(\beta)>0$. But a short computation using $\left<\alpha,\beta\right>=-1$ shows that $w^{-1}(\beta)=v^{-1}(\alpha)$. Hence, we have $\alpha\in A(v)$. In total, this means $\alpha\in AD(v,w)$. By Theorem~\ref{thm:absurd}, we know that $\alpha$ cannot be a minimal element of $AD(v,w)$. Consequently, there exists a root $\alpha'\in R^+$ such that $\alpha'<\alpha$ and such that $\alpha'$ is a minimal element of $AD(v,w)$. Let $\alpha'$ be such a positive root. By our choice of $\alpha'$, we know that $v^{-1}(\alpha')>0$ and $w^{-1}(\alpha')<0$. In view of $\left<\alpha,\beta\right>=-1$, a short computation shows that 
$$
w^{-1}(\alpha')=v^{-1}(\alpha')-\left<\alpha',\beta\right>v^{-1}(\beta)-\left<\alpha',\alpha\right>(v^{-1}(\alpha)+v^{-1}(\beta))
$$
Since $v^{-1}(\beta)>0, v^{-1}(\alpha)>0, v^{-1}(\alpha')>0, w^{-1}(\alpha')<0$, it follows that either $(\alpha',\beta)>0$ or $(\alpha,\alpha')>0$. 

Since $\alpha'$ is a minimal element of $AD(v,w)$, we can apply the generalized lifting property to the Bruhat interval $v<w$ (cf. \cite[Theorem~6.3]{GLP}). This theorem tells us in view of the length difference $\ell(v,w)=2$ that $v\lessdot s_{\alpha'}v\lessdot w$. Hence, there exists a unique root $\gamma\in R^+$ such that $w=s_\gamma s_{\alpha'}v$. This means in particular that $s_\alpha s_\beta=s_\gamma s_{\alpha'}$.

We now treat the two cases $(\alpha',\beta)>0$ and $(\alpha,\alpha')>0$ separately. Assume first that $(\alpha',\beta)>0$. Since $s_\beta w>w$ by assumption and since $s_{\alpha'}w<w$, we see that $\alpha'\neq\beta$. Thus, we have $\left<\alpha',\beta\right>=1$. 
Now, let $\delta=s_{\alpha'}(\beta)=\beta-\alpha'$. It is clear that $\delta$ is a negative root. If we evaluate the equation $s_\alpha s_\beta=s_\gamma s_{\alpha'}$ at $\beta$, we find in view of $\left<\alpha,\beta\right>=-1$ that
$$
-\alpha-\beta=\delta-\left<\delta,\gamma\right>\gamma=\beta-\alpha'-\left<\delta,\gamma\right>\gamma\,.
$$
If we rearrange this equation, it becomes
$$
2\beta=\alpha'-\alpha+\left<\delta,\gamma\right>\gamma\,.
$$
As $\alpha'-\alpha<0$, we see that $\left<\delta,\gamma\right>>0$. We now prove that $\gamma\neq\pm\delta$. 
Indeed, suppose that $\gamma=\pm\delta$. Then we have $s_\alpha s_\beta=s_\gamma s_{\alpha'}=s_{s_{\alpha'}(\beta)} s_{\alpha'}=s_{\alpha'}s_\beta$. It follows that $s_\alpha=s_{\alpha'}$ and thus $\alpha=\alpha'$ -- a contradiction. Since $\gamma\neq\pm\delta$, we know that $\left<\delta,\gamma\right>=1$. If we plug this value into the previous displayed equality and rearrange, we find that $\gamma=2\beta+\alpha-\alpha'$. Thus we find the pairing $\left<\gamma,\beta\right>=4-1-1=2$ which means that $\beta=\gamma$. But then we have $s_\gamma w=s_{\alpha'}v<w<s_\beta w=s_\gamma w$ -- a contradiction. Equally well, one can plug the equality $\beta=\gamma$ into the equation $\gamma=2\beta+\alpha-\alpha'$ which leads to $\alpha'=\alpha+\beta$ -- a contradiction to $\alpha'<\alpha$.

All in all, this shows that we can assume $(\alpha',\beta)\leq 0$ and $(\alpha,\alpha')>0$. Since $\alpha\neq\alpha'$, we clearly have $\left<\alpha,\alpha'\right>=1$. We can define $\alpha''=s_{\alpha'}(\alpha)=\alpha-\alpha'\in R^+$. It is rather clear that $\left<\alpha',\alpha''\right>=-1$ and that $\alpha=s_{\alpha'}(\alpha'')$. We find that 
$$
s_{\alpha'}s_{\alpha''}s_{\alpha'}s_\beta=s_{s_{\alpha'}(\alpha'')}s_\beta=s_\alpha s_\beta=s_\gamma s_{\alpha'}
$$
and thus $s_{\alpha''}s_{\alpha'}s_\beta=s_{\alpha'}s_\gamma s_{\alpha'}$. Since $\alpha'\in D(w)$, we can use this equation to see that the root
$$
w^{-1}(\alpha')=v^{-1}s_{\alpha'} s_\gamma(\alpha')=-v^{-1}s_{\alpha'}s_\gamma s_{\alpha'}(\alpha')=-v^{-1}s_\beta s_{\alpha'}s_{\alpha''}(\alpha')
$$
is negative. Thus, we have $v^{-1}s_\beta s_{\alpha'}s_{\alpha''}(\alpha')>0$ which means that $s_{\alpha''}s_{\alpha'}s_\beta v<s_{\alpha'}s_{\alpha''}s_{\alpha'}s_\beta v=s_\alpha s_\beta v=w$. On the other hand, since $\alpha'\in A(v)$ and since $(\alpha',\beta)\leq 0$, we also have $v^{-1}s_\beta(\alpha')=v^{-1}(\alpha')-\left<\alpha',\beta\right>v^{-1}(\beta)>0$ and thus $s_\beta v<s_{\alpha'}s_\beta v$. Finally, we compute that $v^{-1}s_\beta s_{\alpha'}(\alpha'')=v^{-1}s_\beta(\alpha)>0$ by the very definition of $\alpha$. This means that $s_{\alpha'}s_\beta v<s_{\alpha''}s_{\alpha'}s_\beta v$. If we put these findings together, we obtain the following chain 
$$
v<s_\beta v<s_{\alpha'}s_\beta v<s_{\alpha''}s_{\alpha'}s_\beta v<s_{\alpha'}s_{\alpha''}s_{\alpha'}s_\beta v=w\,.
$$
It is obvious that this contradicts the fact that $\ell(v,w)=2$. This completes the proof of the theorem.
\end{proof}

\begin{ex}
\label{ex:shuffle1}

We give examples of non simply laced Weyl groups for which Theorem~\ref{thm:shuffle} fails. Indeed, for this example, let $W$ be a Weyl group of rank two with simple roots $\beta$ and $\beta'$ whose Coxeter integer is $>3$. Let $v=s_{\beta'}$ and $w=s_{\beta'}s_\beta s_{\beta'}$. Then $v,\beta,w$ satisfy all of the assumptions of Theorem~\ref{thm:shuffle} except that $W$ is a non simply laced Weyl group. But the conclusion of Theorem~\ref{thm:shuffle} fails. In fact, we have $s_\beta w>w$. Note, since we did not fix the arrangement of the simple roots $\beta$ and $\beta'$ according to their different length, we actually gave for a fixed non simply laced Weyl group of rank two two examples for which Theorem~\ref{thm:shuffle} fails.

\end{ex}

\begin{ex}
\label{ex:shuffle2}

We give an example which shows that Theorem~\ref{thm:shuffle} fails for $v,\beta,w$ which satisfy all of the assumptions of Theorem~\ref{thm:shuffle} except that $s_\beta v\lessdot w$ is replaced by $s_\beta v<w$.
Indeed, for this example, let $R$ be of type $\mathsf{A}_4$. Let $\beta_1,\beta_2,\beta_3,\beta_4$ be the simple roots with the labeling as in \cite[Plate~I]{bourbaki_roots}. Let $v=s_{\beta_3} s_{\beta_4}s_{\beta_3}s_{\beta_1}$, $\beta=\beta_2$, $w=s_{\beta_1}s_{\beta_2}s_{\beta_3}s_{\beta_4}s_{\beta_3}s_{\beta_2}s_{\beta_1}$. Then $v,\beta,w$ satisfy all of the assumptions of Theorem~\ref{thm:shuffle} except that only $s_\beta v<w$ and not $s_\beta v\lessdot w$. But the conclusion of Theorem~\ref{thm:shuffle} fails. In fact, we have $s_\beta w>w$.

\end{ex}

\begin{ex}
\label{ex:shuffle3}

Let $v,w\in W$. Let $\beta\in A_\ell(v)$ be such that $s_\beta v\lessdot w$. We give an example of $v,\beta,w$ as stated such that $s_\beta w>w$. This shows that Theorem~\ref{thm:shuffle} fails if $v$ is not a shuffle element and if $\beta$ is an element of $A_\ell(v)$ which is not uniquely determined by $v$. Indeed, for this example, let $R$ be of type $\mathsf{A}_n$ where $n\geq 2$. Let $\beta,\beta'\in\Delta$ such that $m(\beta,\beta')=3$. Let $v=1$ and let $w=s_{\beta'}s_\beta$. Then $v,\beta,w$ are as stated above such that $s_\beta w>w$. If we set $n=2$, we see that $A_\ell(v)$ consists precisely of the elements $\beta$ and $\beta'$, and $v$ fails to be a shuffle element by one additional root in $A_\ell(v)$. 

\end{ex}

\subsection{Bruhat intervals of length two}

In this subsection, we prove the main theorem on Bruhat intervals of length two (Theorem~\ref{thm:main-interval2}). This theorem follows quite easily from Theorem~\ref{thm:shuffle}. It is only left to express the conditions imposed on a shuffle element in terms of properties of a Bruhat interval of length two. 

\todo[inline,color=green]{\color{black} I don't use interval notation like $[v,w]$ except maybe in the introduction or other less formal contexts. [I don't use it at all up to now and don't intend to use it in the future.]

Also, I should note in the section on roots that there are precisely four exceptions $\alpha$ such that there are several $\beta$'s in $\Delta(\alpha)$ such that $(\alpha,\beta)<0$. This should be a part of the corresponding remark after I did all necessary corrections and proof reading. [I mentioned this in a remark after Fact~\ref{fact:classification-final}.]

Another, slightly different description of the main theorem on shuffle elements. The assumptions are stated in terms of assumptions on the interval $[v,w]$ of length two rather than on the element $v$. Useful for our applications.}

\begin{lem}
\label{lem:simple}

Let $\beta,\beta'\in\Delta$ and $\alpha,\alpha'\in R^+$ be such that $\beta\neq\beta'$ and such that $s_\beta s_{\beta'}=s_{\alpha}s_{\alpha'}$. Then we have $\alpha\in\{\beta,\beta'\}$ or $\alpha'\in\{\beta,\beta'\}$.

\end{lem}

\begin{proof}

Let $\beta,\beta',\alpha,\alpha'$ be as in the statement. Since $\beta\neq\beta'$, it is easy to see that the ordinary length as well as well as the reflection length of $s_\beta s_{\beta'}$ is equal to two. In view of \cite[Theorem~1.4]{reflection-length} applied to the parabolic subgroup of $W$ generated by $s_\beta$ and $s_{\beta'}$, it follows that $\Delta(\alpha)\subseteq\{\beta,\beta'\}$ and $\Delta(\alpha')\subseteq\{\beta,\beta'\}$. Assume that $\alpha\notin\{\beta,\beta'\}$. Then it is clear that $\beta$ and $\beta'$ cannot be orthogonal and that $\alpha=\beta+\beta'$. In particular, we have $s_\alpha=s_\beta s_{\beta'} s_\beta$. From the equation $s_\beta s_{\beta'}=s_\alpha s_{\alpha'}$, it now follows that $s_{\alpha'}=s_\alpha s_\beta s_{\beta'}=s_\beta$ and thus $\alpha'=\beta$. This completes the proof of the lemma.
\end{proof}

\begin{rem}

It should be clear that one can also give a simple and more elementary proof of Lemma~\ref{lem:simple} without referring to \cite{reflection-length}. Since this is cumbersome, we preferred this solution.

\end{rem}

\begin{lem}
\label{lem:unique-ascent}

Let $v,w\in W$ be such that $\ell(v,w)=2,v\leq w, v\not\leq_\ell w$. Suppose that for all $\beta\in A_\ell(v)$ we have $s_\beta v\leq w$. Then $v$ is a shuffle element.

\end{lem}

\begin{proof}

Let $v$ and $w$ be as in the statement. By assumption, it is clear that $A_\ell(v)\neq\emptyset$. Let $\beta,\beta'\in A_\ell(v)$. It suffices to show that $\beta=\beta'$. Suppose for a contradiction that $\beta\neq\beta'$. By assumption, we have $s_\beta v\lessdot w$ and $s_{\beta'}v\lessdot w$. Hence, there exist unique roots $\alpha,\alpha'\in R^+$ such that $w=s_\alpha s_\beta v=s_{\alpha'}s_{\beta'}v$. This last equality means also that $s_\alpha s_\beta=s_{\alpha'}s_{\beta'}$ and that $s_\beta s_{\beta'}=s_\alpha s_{\alpha'}$. Lemma~\ref{lem:simple} now implies that $\alpha\in\{\beta,\beta'\}$ or that $\alpha'\in\{\beta,\beta'\}$. In both cases, it follows that $v\leq_\ell w$ which is contrary to our initial assumption on the interval $v<w$.
\end{proof}

\begin{thm}
\label{thm:interval2}

Let $v,w\in W$ be such that $\ell(v,w)=2, v\leq w, v\not\leq_\ell w$. Suppose that for all $\beta\in A_\ell(v)$ we have $s_\beta v\leq w$. Then we have for all $\beta\in A_\ell(v)$ that $s_\beta w<w$.

\end{thm}

\begin{proof}

Let $v$ and $w$ be as in the statement. By Lemma~\ref{lem:unique-ascent}, it follows that $v$ is a shuffle element. Let $\beta$ be the pivot element of $v$. By assumption, we have $s_\beta v\lessdot w$. Theorem~\ref{thm:shuffle} implies that $s_\beta w<w$. In view of $A_\ell(v)=\{\beta\}$, this completes the proof.
\end{proof}

\begin{rem-prob}
\label{rem*:two}

While Theorem~\ref{thm:shuffle} and Theorem~\ref{thm:main-interval2} clearly fail for non simply laced Weyl groups (cf.\ Remark~\ref{rem:two} and Example~\ref{ex:shuffle1}), we were not able to find a counterexample to Theorem~\ref{thm:interval2} for an arbitrary Weyl group. It might be true that under the same assumptions on $v$ and $w$ the conclusion of Theorem~\ref{thm:interval2} or the weaker conclusion that there exists $\beta\in A_\ell(v)$ such that $s_\beta w<w$ hold for more general Coxeter groups, e.g.\ arbitrary Weyl groups.

\end{rem-prob}

\begin{rem}

Example~\ref{ex:shuffle2} shows that the conclusion of Theorem~\ref{thm:interval2} (or even the weaker conclusion that there exists $\beta\in A_\ell(v)$ such that $s_\beta w<w$) fails for elements $v,w\in W$ which satisfy all of the assumptions of Theorem~\ref{thm:interval2} except that $\ell(v,w)>2$. Indeed, the elements $v,w\in W$ as in Example~\ref{ex:shuffle2} satisfy in addition to the properties discussed there $v\not\leq_\ell w$ and $\ell(v,w)=3$.

\end{rem}

\begin{thm}[Main theorem on Bruhat intervals of length two]
\label{thm:main-interval2}

Let $v,w\in W$ be such that $\ell(v,w)=2$ and $v\leq w$. Suppose that for all $\beta\in A_\ell(v)$ we have $s_\beta v\leq w$. Then there exists $\beta\in A_\ell(v)$ such that $s_\beta w<w$.

\end{thm}

\begin{proof}

Let $v$ and $w$ be as in the statement. If $v\not\leq_\ell w$, then the result follows clearly from Theorem~\ref{thm:interval2}. Hence, we may assume that $v\leq_\ell w$. Let $\beta,\beta'\in\Delta$ such that $w=s_{\beta'}s_\beta v$. If $s_\beta w <w$, the theorem is proved. Hence, we may assume that $s_\beta w>w$. This means in particular that $\beta$ and $\beta'$ are not orthogonal. Under these assumptions, we find that $w^{-1}(\beta)=v^{-1}(\beta')>0$. Hence, it suffices to prove that $w^{-1}(\beta')<0$. But this last inequality is clear because we have $w^{-1}(\beta')=-v^{-1}(\beta)-v^{-1}(\beta')$ and $\beta,\beta'\in A_\ell(v)$.
\end{proof}

\begin{rem}

In Lemma~\ref{lem:unique-ascent}, Theorem~\ref{thm:interval2}, Theorem~\ref{thm:main-interval2}, the assumption $v\leq w$ is actually each time superfluous because it follows from the other assumptions on $v,w\in W$. We included it just for simplicity.

\end{rem}

\begin{rem}
\label{rem:two}

Example~\ref{ex:shuffle1} and Example~\ref{ex:shuffle2} show again that Theorem~\ref{thm:main-interval2} fails for non simply laced Weyl groups and for elements $v,w\in W$ which satisfy all of the assumptions of Theorem~\ref{thm:main-interval2} except that $\ell(v,w)>2$.

\end{rem}

\begin{ex}

Let $v,w\in W$ as in the statement of Theorem~\ref{thm:main-interval2}. We give an example to show that the stronger conclusion \enquote{For all $\beta\in A_\ell(v)$, we have that $s_\beta w<w$.} (as in Theorem~\ref{thm:interval2}) does not hold. Indeed, for this example, analogously to Example~\ref{ex:shuffle3}, let $R$ be of type $\mathsf{A}_2$ with simple roots $\beta$ and $\beta'$. Let $v=1$ and $w=s_{\beta'}s_\beta$. Then $v$ and $w$ satisfy all of the assumptions of Theorem~\ref{thm:main-interval2}. But we have $\beta'\in A_\ell(v)$ and $s_{\beta'}w>w$.

\end{ex}

\begin{ex}

Let $v,w\in W$ as in the statement of Theorem~\ref{thm:main-interval2}. We give an example which shows that the $\beta\in A_\ell(v)$ such that $s_\beta w<w$ is not necessarily unique (as it is the case in the conclusion of Theorem~\ref{thm:interval2} as we see from its proof). Indeed, for this example, let $R$ be of type $\mathsf{A}_3$. Let $\beta_1,\beta_2,\beta_3$ be the simple roots with the labeling as in \cite[Plate~I]{bourbaki_roots}. Let $v=s_{\beta_2}$ and $w=s_{\beta_1}s_{\beta_3}s_{\beta_2}$. Then $v$ and $w$ satisfy all of the assumptions of Theorem~\ref{thm:main-interval2}. But we have $A_\ell(v)=\{\beta_1,\beta_3\}$, $s_{\beta_1}w<w$, $s_{\beta_3}w<w$.

\end{ex}

\bibliographystyle{aomplain}
\bibliography{lib}

\providecommand{\bysame}{\leavevmode\hbox to3em{\hrulefill}\thinspace}
\providecommand{\noopsort}[1]{}
\providecommand{\mr}[1]{\href{http://www.ams.org/mathscinet-getitem?mr=#1}{MR~#1}}
\providecommand{\zbl}[1]{\href{http://www.zentralblatt-math.org/zmath/en/search/?q=an:#1}{Zbl~#1}}
\providecommand{\jfm}[1]{\href{http://www.emis.de/cgi-bin/JFM-item?#1}{JFM~#1}}
\providecommand{\arxiv}[1]{\href{http://www.arxiv.org/abs/#1}{arXiv~#1}}
\providecommand{\doi}[1]{\url{http://dx.doi.org/#1}}
\providecommand{\MR}{\relax\ifhmode\unskip\space\fi MR }
\providecommand{\MRhref}[2]{%
  \href{http://www.ams.org/mathscinet-getitem?mr=#1}{#2}
}
\providecommand{\href}[2]{#2}
\begin{thebibliography}{10}

\bibitem{andrus}
\bgroup\scshape{}N.~Andruskiewitsch\egroup{} and \bgroup\scshape{}H.-J.
  Schneider\egroup{}, Pointed {H}opf algebras,  in \emph{New directions in
  {H}opf algebras}, \emph{Math. Sci. Res. Inst. Publ.} \textbf{43}, Cambridge
  Univ. Press, Cambridge, 2002, pp.~1--68. \mr{1913436}.  Available at
  \url{https://doi.org/10.2977/prims/1199403805}.

\bibitem{quasi-homogeneity}
\bgroup\scshape{}C.~{B{\"a}rligea}\egroup{}, {Quasi-homogeneity of the moduli
  space of stable maps to homogeneous spaces},  \emph{ArXiv e-prints} (2017).
  Available at \url{http://adsabs.harvard.edu/abs/2017arXiv170606452B}.

\bibitem{reflection-length}
\bgroup\scshape{}B.~Baumeister\egroup{}, \bgroup\scshape{}M.~Dyer\egroup{},
  \bgroup\scshape{}C.~Stump\egroup{}, and \bgroup\scshape{}P.~Wegener\egroup{},
  A note on the transitive {H}urwitz action on decompositions of parabolic
  {C}oxeter elements,  \emph{Proc. Amer. Math. Soc. Ser. B} \textbf{1} (2014),
  149--154. \mr{3294251}.  Available at
  \url{https://doi.org/10.1090/S2330-1511-2014-00017-1}.

\bibitem{bazlov1}
\bgroup\scshape{}Y.~Bazlov\egroup{}, Nichols-{W}oronowicz algebra model for
  {S}chubert calculus on {C}oxeter groups,  \emph{J. Algebra} \textbf{297}
  (2006), 372--399. \mr{2209265}.  Available at
  \url{https://doi.org/10.1016/j.jalgebra.2006.01.037}.

\bibitem{bazlov2}
\bgroup\scshape{}Y.~Bazlov\egroup{} and
  \bgroup\scshape{}A.~Berenstein\egroup{}, Cocycle twists and extensions of
  braided doubles,  in \emph{Noncommutative birational geometry,
  representations and combinatorics}, \emph{Contemp. Math.} \textbf{592}, Amer.
  Math. Soc., Providence, RI, 2013, pp.~19--70. \mr{3087939}.  Available at
  \url{https://doi.org/10.1090/conm/592/11772}.

\bibitem{bourbaki_roots}
\bgroup\scshape{}N.~Bourbaki\egroup{}, \emph{Lie groups and {L}ie algebras.
  {C}hapters 4--6}, \emph{Elements of Mathematics (Berlin)}, Springer-Verlag,
  Berlin, 2002, Translated from the 1968 French original by Andrew Pressley.
  \mr{1890629}.  Available at \url{https://doi.org/10.1007/978-3-540-89394-3}.

\bibitem{GLP}
\bgroup\scshape{}F.~Caselli\egroup{} and
  \bgroup\scshape{}P.~Sentinelli\egroup{}, The generalized lifting property of
  {B}ruhat intervals,  \emph{J. Algebraic Combin.} \textbf{45} (2017),
  687--700. \mr{3627500}.  Available at
  \url{https://doi.org/10.1007/s10801-016-0721-7}.

\bibitem{deodhar}
\bgroup\scshape{}V.~V. Deodhar\egroup{}, Some characterizations of {B}ruhat
  ordering on a {C}oxeter group and determination of the relative {M}\"obius
  function,  \emph{Invent. Math.} \textbf{39} (1977), 187--198. \mr{0435249}.
  Available at \url{https://doi.org/10.1007/BF01390109}.

\bibitem{dyer-reflection-ordering}
\bgroup\scshape{}M.~J. Dyer\egroup{}, Hecke algebras and shellings of {B}ruhat
  intervals,  \emph{Compositio Math.} \textbf{89} (1993), 91--115.
  \mr{1248893}.  Available at
  \url{http://www.numdam.org/item?id=CM_1993__89_1_91_0}.

\bibitem{etingof}
\bgroup\scshape{}P.~Etingof\egroup{}, \bgroup\scshape{}S.~Gelaki\egroup{},
  \bgroup\scshape{}D.~Nikshych\egroup{}, and
  \bgroup\scshape{}V.~Ostrik\egroup{}, \emph{Tensor categories},
  \emph{Mathematical Surveys and Monographs} \textbf{205}, American
  Mathematical Society, Providence, RI, 2015. \mr{3242743}.  Available at
  \url{https://doi.org/10.1090/surv/205}.

\bibitem{quadratic}
\bgroup\scshape{}S.~Fomin\egroup{} and \bgroup\scshape{}A.~N.
  Kirillov\egroup{}, Quadratic algebras, {D}unkl elements, and {S}chubert
  calculus,  in \emph{Advances in geometry}, \emph{Progr. Math.} \textbf{172},
  Birkh\"auser Boston, Boston, MA, 1999, pp.~147--182. \mr{1667680}.

\bibitem{nilcoxeter}
\bgroup\scshape{}S.~Fomin\egroup{} and \bgroup\scshape{}R.~P. Stanley\egroup{},
  Schubert polynomials and the nil-{C}oxeter algebra,  \emph{Adv. Math.}
  \textbf{103} (1994), 196--207. \mr{1265793}.  Available at
  \url{https://doi.org/10.1006/aima.1994.1009}.

\bibitem{humphreys-coxeter}
\bgroup\scshape{}J.~E. Humphreys\egroup{}, \emph{Reflection groups and
  {C}oxeter groups}, \emph{Cambridge Studies in Advanced Mathematics}
  \textbf{29}, Cambridge University Press, Cambridge, 1990. \mr{1066460}.
  Available at \url{https://doi.org/10.1017/CBO9780511623646}.

\bibitem{kirillov}
\bgroup\scshape{}A.~N. Kirillov\egroup{}, Skew divided difference operators and
  {S}chubert polynomials,  \emph{SIGMA Symmetry Integrability Geom. Methods
  Appl.} \textbf{3} (2007), Paper 072, 14. \mr{2322799}.  Available at
  \url{https://doi.org/10.3842/SIGMA.2007.072}.

\bibitem{liu}
\bgroup\scshape{}R.~I. Liu\egroup{}, Positive expressions for skew divided
  difference operators,  \emph{J. Algebraic Combin.} \textbf{42} (2015),
  861--874. \mr{3403185}.  Available at
  \url{https://doi.org/10.1007/s10801-015-0606-1}.

\bibitem{macdonald1991notes}
\bgroup\scshape{}I.~G. Macdonald\egroup{}, \emph{Notes on {S}chubert
  Polynomials}, \emph{Publications du Laboratoire de combinatoire et
  d'informatique math{\'e}matique}, D{\'e}p. de math{\'e}matiques et
  d'informatique, Universit{\'e} du Qu{\'e}bec {\`a} Montr{\'e}al, 1991.
  Available at \url{https://books.google.fr/books?id=BvLuAAAAMAAJ}.

\bibitem{majid}
\bgroup\scshape{}S.~Majid\egroup{}, \emph{Foundations of quantum group theory},
  Cambridge University Press, Cambridge, 1995. \mr{1381692}.  Available at
  \url{https://doi.org/10.1017/CBO9780511613104}.

\end{thebibliography}


\end{document}